\documentclass[12pt]{article}
\usepackage{amsmath,fullpage,amssymb,amsthm,hyperref,enumerate}
\usepackage{tikz,color,mathtools}
\usetikzlibrary{patterns}
\usepackage{multirow}
\usepackage{pgfplots}
\usepackage{cases}
 
\newtheorem{theorem}{Theorem}[section]

\newtheorem{corollary}[theorem]{Corollary}

\theoremstyle{definition}

\theoremstyle{remark}

\newcommand\D{\mathcal D}

\newcommand\tS{\tilde S}
\newcommand\dB{\widehat B}
\newcommand\dD{\widehat D}
\newcommand\dG{\widehat G}
\newcommand\dI{\widehat I}
\newcommand\Na{N}
\newcommand\uu{\textsf{u}}
\newcommand\dd{\textsf{d}}
\newcommand\dA{\widehat A}
\newcommand\dU{\widehat U}
\newcommand\dP{\widehat P}
\newcommand\btc{\begin{tabular}{@{\hspace{-2pt}}c@{\hspace{-2pt}}}}
\newcommand\et{\end{tabular}}
\DeclareMathOperator\NB{NB}
\newcommand\up{-- ++(1,1) circle(1.2pt)}
\newcommand\dn{-- ++(1,-1) circle(1.2pt)}

\title{On individual leaf depths of trees}

\author{Sergi Elizalde\thanks{Department of Mathematics, Dartmouth College, Hanover, NH 03755, USA.
\url{sergi.elizalde@dartmouth.edu}.
}}

\date{}

\begin{document}

\maketitle

\begin{abstract}
We explore a generating function trick which allows us to keep track of infinitely many statistics using finitely many variables, by recording their individual distributions rather than their joint distributions.
Building on previous work of Panholzer and Prodinger, we apply this method to study the depth distributions of individual nodes or leaves in rooted binary trees, plane trees, noncrossing trees, and increasing trees; the height distributions of individual vertices and individual steps in Dyck paths; and
the number of diagonals separating two fixed sides of a convex polygon in a triangulation or dissection.
We obtain both exact and asymptotic results, which sometimes refine known formulas or provide combinatorial proofs of results from the probability literature.
\end{abstract}


\section{Introduction}

There is an extensive literature on rooted tree parameters that relate to the depth of its nodes, see \cite[Ch.\ VII.3.2]{flajolet_analytic_2009} for several examples. The {\it depth} (sometimes called {\it height}) of a node is defined as the number of edges in the path from the node to the root of the tree. For instance, it is known \cite[Prop.\ VII.3]{flajolet_analytic_2009} that,
for rooted trees with $n$ nodes from a so-called smooth simple variety (which includes binary trees and plane trees), the average depth of a random node behaves like $O(\sqrt{n})$. Specifically, it behaves like $\sqrt{\pi n}$ for binary trees, and like $\sqrt{\pi n}/2$ for plane trees. A related parameter is the height of a tree, which equals the maximum depth of any node.
It is known \cite[Thms.\ B and G]{flajolet_average_1982} that the average height 
is asymptotically $2\sqrt{\pi n}$ for binary trees, and $\sqrt{\pi n}$ for plane trees.

More generally, Drmota and Gittenberger~\cite{drmota_profile_1997} studied the {\em horizontal profile}, determined by the number of nodes at each depth, of a random plane tree as its size tends to infinity. They showed that it converges weakly to Brownian excursion local time, and that the same is true when considering instead the number of leaves at each depth.
As a variation of this notion, Bousquet-M\'elou~\cite{bousquet-melou_limit_2006} considered the {\em vertical profile} of a binary tree, determined by the number of nodes at each abscissa when the children of a node are placed one unit to its left and to its right. She 
gave a beautiful description of this vertical profile, and showed that the limit distribution is related to {\em integrated superBrownian excursion (ISE)}; see~\cite{bousquet-melou_density_2006} for more details.

In this paper we consider a related but different problem. Instead of focusing on the depth of a random node or a random leaf in a tree, we are interested in each individual leaf or node. For example, taking the natural order of the leaves of a binary tree from left to right, one can focus on the $r$th leaf of a tree for a particular value of $r$, and consider the distribution of its depth over all trees of a given size. 
This question for binary trees has been considered by Kirschenhofer~\cite{kirschenhofer_height_1983,kirschenhofer_new_1983}, Gutjahr and Pflug~\cite{gutjahr_asymptotic_1992}, Panholzer and Prodinger~\cite{panholzer_descendants_1997,panholzer_moments_2002}, and Drmota~\cite{drmota_distribution_1994}. 

Following the approach from~\cite{panholzer_descendants_1997}, we define multivariate generating functions that allow us to keep track of infinitely many statistics on combinatorial objects, such as the depth of the $r$th leaf on binary trees for each possible value of $r$, using finitely many variables.
This is achieved by having the coefficient of $x^r$ be the bivariate generating function enumerating the objects by size and by the value of the $r$th statistic. As a trade-off, these multivariate generating functions 
do not track the joint distributions of the various statistics, but rather only their individual distributions. 

In the case of trees, this approach is different from Bousquet-M\'elou's~\cite{bousquet-melou_limit_2006} in that, whereas her bivariate generating functions count trees by size and by the number of nodes for which the value of the parameter (e.g.\ depth or abscissa) is fixed, ours count trees by size and by the value of the parameter for a fixed node.

While most of the existing work using this technique focuses of binary trees, here we extend the method to a variety of other combinatorial objects, including plane trees, Schr\"oder trees, noncrossing trees, increasing binary trees, Dyck paths, and triangulations and dissections of polygons. 
For example, in the latter case, we study the number of diagonals that separate two given sides of the polygon in a random triangulation or dissection.
For Dyck paths, we study the height of the path at a given $x$-coordinate, and the height of the $r$th up-step of the path for given $r$, obtaining simple algebraic generating functions and recovering some results of Disanto and Munarini~\cite{disanto_local_2019} derived using a different approach.

Another novelty of our approach is that we are able use singularity analysis on our multivariate generating functions
to show that, in all of the above cases except for increasing trees, the distribution of the $r$th statistic for fixed $r$ converges to a discrete law as the size of the objects goes to infinity. We provide the probability generating functions for these limit distributions.

Generating functions for trees with respect to individual leaf depths also have applications to counting vertices of certain generalized associahedra, as defined by Bottman~\cite{bottman_2-associahedra_2019}. This avenue is currently being explored by the author in work in progress with Backman and Warrington.

A summary of the combinatorial objects and statistics that we study in the paper appears in Table~\ref{tab:summary}. For each combinatorial class that we consider, objects of size $n$ have a family of statistics $s_r$ associated to them, where typically 
$0\le r\le n$ or $0\le r\le 2n$.
For example, these statistics may be the depths of the individual leaves of a tree (ordered from left to right), or the heights of the vertices of a Dyck path.
We consider multivariate generating functions $F(x,y,z)$ where the coefficient of $x^ry^dz^n$ is the number of objects of size $n$ for which the statistic $s_r$ equals $d$.
Thus, the coefficient of $x^rz^n$ in $F(x,y,z)$, which we denote by $[x^rz^n]F(x,y,z)$, is a polynomial in $y$ that describes the distribution of the statistic $s_r$ on objects of size $n$.  The value $f_n=[x^rz^n]F(x,1,z)$ is the number of objects of size $n$ ---which in particular is independent of $r$, for $r$ in the appropriate range. The average value of $s_r$ on these objects equals
\begin{equation}\label{eq:average}
\frac{1}{f_n}[x^rz^n]\left.\frac{\partial F(x,y,z)}{\partial y}\right|_{y=1}.
\end{equation}

\begin{table}[h]
\centering
\begin{tabular}{c|c|c}
Combinatorial objects & Statistics & Sections in paper \\ \hline\hline
\multirow{2}{*}{binary trees} & leaf depths & \ref{sec:B} and \ref{sec:average_depth_binary} \\ \cline{2-3} 
& leaf abscissas & \ref{sec:abscissas} \\ \hline
\multirow{2}{*}{plane trees} & leaf depths & \ref{sec:P} \\ \cline{2-3} 
& node depths 
& \ref{sec:bijections} and \ref{sec:Dyck_upsteps} \\ \hline
Schr\"oder trees & leaf depths & \ref{sec:Schroeder} and \ref{sec:distribution_Schroeder} \\ \hline
noncrossing trees & node depths & \ref{sec:noncrossing} \\ \hline
\multirow{2}{*}{increasing binary trees} & leaf depths & \ref{sec:increasing-leaves} \\ \cline{2-3}
& internal node depths  & \ref{sec:increasing-nodes} \\ \hline
triangulations & \multirow{2}{*}{separating diagonals} & \ref{sec:triang}\\ \cline{1-1}\cline{3-3}
dissections &  & \ref{sec:dissections}\\ \hline
\multirow{2}{*}{Dyck paths} & vertex heights & \ref{sec:Dyck_vertices} and \ref{sec:Dyck_vertices_average}  \\ \cline{2-3} 
& up-step heights & \ref{sec:Dyck_upsteps} and \ref{sec:Dyck_upsteps_average} \\ \hline
\end{tabular}
\caption{A summary of the combinatorial objects and statistics studied in this paper.}
\label{tab:summary}
\end{table}

In Section~\ref{sec:binary} we illustrate the technique by enumerating binary trees with respect to the depth of each leaf. Some of the results in this section, such as the multivariate generating function, the average depth of the $r$th leaf in binary trees of size $n$, and 
its asymptotic behavior as $n\to\infty$, had been previously obtained by Kirschenhofer~\cite{kirschenhofer_height_1983,kirschenhofer_new_1983}, Gutjahr and Pflug~\cite{gutjahr_asymptotic_1992}, and Panholzer and Prodinger~\cite{panholzer_descendants_1997,panholzer_moments_2002}. 
Still, we describe the limiting distribution of the depth of the $r$th leaf for fixed $r$, which appears to be a new result. We also consider the abscissa of each leaf as defined in~\cite{bousquet-melou_limit_2006}.

In Section~\ref{sec:Dyck} we enumerate Dyck paths with respect to the height of each vertex, and also with respect to the height of each up-step. We deduce the average height of the $r$th vertex, the $r$th up-step, and the $r$th down-step, as well as the distribution of these statistics for fixed $r$ in the limit as $n\to\infty$.

In Section~\ref{sec:plane} we enumerate plane trees with respect to the depth of each leaf, and also with respect to the depth of each node in preorder ---the latter is equivalent to enumerating Dyck paths with respect to the height of each up-step. We then consider leaf depths in Schr\"oder trees and their asymptotic behavior.
In Section~\ref{sec:polygons} we translate the results from binary trees to triangulations of polygons, and from Schr\"oder trees to dissections of polygons. In both cases, the leaf depth statistic becomes the number of diagonals separating two given sides of the polygon.

In Section~\ref{sec:noncrossing} we enumerate noncrossing trees with respect to the depth of each node, and study the asymptotic behavior of these depths. Noncrossing trees are particularly suited to our approach because their nodes are naturally ordered. 
In Section~\ref{sec:increasing} we show that similar methods can also be applied to trees with labeled vertices, such as increasing binary trees.

\section{Binary trees}\label{sec:binary}

Complete binary trees are rooted trees where each internal (i.e., non-leaf) node has a left and a right child.
We will call these {\em binary trees} for short. 
Define the size of a binary tree to be its number of internal nodes, which is one less that the number of leaves. 
The number of binary trees of size $n$ is the $n$th Catalan number $c_n=\frac{1}{n+1}\binom{2n}{n}$, and so the corresponding generating function is
$$C(z)=\sum_{n\ge0} c_{n}z^n=\frac{1-\sqrt{1-4z}}{2z}.$$ 

There is a natural ordering of the leaves of a binary tree from left to right, obtained by traversing the tree in preorder. Equivalently,
a leaf $\ell_1$ is considered to precede another leaf $\ell_2$ if, at the first node where the paths from the root to these leaves differ, the path to $\ell_1$ goes left and the path to $\ell_2$ goes right. 
In a binary tree of size $n$, we index its $n+1$ leaves from $0$ to $n$ from left to right, so that the leftmost leaf is the $0$th leaf.
The {\em depth} of a leaf is the number of edges in the path from the root to that leaf.

\subsection{Enumeration with respect to leaf depths}\label{sec:B}

For $r\ge0$, let $B_r(y,z)$ be the bivariate generating function for binary trees where $z$ marks the size and $y$ marks the depth of the $r$th leaf (considering only the trees for which such leaf exists). Let $B(x,y,z)=\sum_{r\ge0}B_r(y,z)x^{r}$. The coefficient of $x^ry^dz^n$ in
$B(x,y,z)$ is the number of binary trees of size $n$ where the $r$th leaf has depth $d$. 
The following simple expression for $B(x,y,z)$ appears as \cite[Eq.~(8)]{panholzer_descendants_1997}.

\begin{theorem}[\cite{panholzer_descendants_1997}]
\label{thm:B}
$$B(x,y,z)=\frac{1}{1-yz C(z)-xyz C(xz)}
=\frac{2}{2-2y+y\sqrt{1-4z}+y\sqrt{1-4xz}}.$$
\end{theorem}

\begin{proof}
One can think of $B(x,y,z)$ as the generating function for binary trees with a distinguished\footnote{One could use the term {\em marked} instead of {\em distinguished}, but we prefer to avoid confusion with the notion of a variable marking a statistic.} leaf, where $x$ marks the index of this leaf, and $y$ marks its depth.  We will show that $B(x,y,z)$ satisfies the equation
\begin{equation}\label{eq:B} B(x,y,z)=1+yzB(x,y,z)C(z)+xyzC(xz)B(x,y,z).\end{equation}
Indeed, the binary tree with one node contributes $1$ to the generating function, and any other binary tree can be decomposed into the two subtrees of the root. Consider two cases. 

If the distinguished leaf is on the left subtree, then this subtree contributes $B(x,y,z)$ to the generating function, with the factor $yz$ accounting for the root and the fact the depth of the distinguished leaf in the original tree is one more than in the subtree. In this case, the right subtree simply contributes a $C(z)$ factor.

If the distinguished leaf is on the right subtree, then the left subtree contributes a factor of $xC(xz)$, since it shifts the indexing of the leaves in the right subtree (where the distinguished leaf lies) by the number of leaves in the left subtree. Now the right subtree contributes $B(x,y,z)$, and there is again a factor $yz$ coming from the root.

Solving~\eqref{eq:B} for $B(x,y,z)$ gives the stated expression.
\end{proof}

As an example, the coefficient of $z^3$ in $B(x,y,z)$ is the polynomial
$$(2y+2y^2+y^3)+(2y^2+3y^3)\,x+(2y^2+3y^3)\,x^2+(2y+2y^2+y^3)\,x^3,$$
whose four coefficients, as a polynomial in $x$, describe the distribution of the depths of each of the four leaves in the trees in Figure~\ref{fig:n=3}. The distribution of the depths of each leaf in binary trees of size 30 is plotted in Figure~\ref{fig:distribution30}.

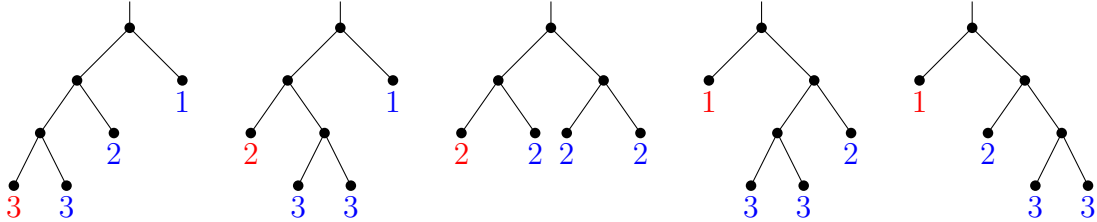
\begin{figure}[htb]
\centering
\begin{tikzpicture}[scale=.7]
\draw (0,0)--(0,.5);
\fill (0,0) circle (.1) coordinate (v0);
\fill (-1,-1) circle (.1) coordinate (v1);
\fill (1,-1) circle (.1) coordinate (v2) node[below,blue] {1};
\fill (-1.7,-2) circle (.1) coordinate (v3);
\fill (-.3,-2) circle (.1) coordinate (v4) node[below,blue] {2};
\fill (-2.2,-3) circle (.1) coordinate (v5) node[below,red] {3};
\fill (-1.2,-3) circle (.1) coordinate (v6) node[below,blue] {3};
\draw (v2)--(v0)--(v1)--(v3)--(v5);
\draw (v3)--(v6);
\draw (v1)--(v4);

\begin{scope}[shift={(4,0)}]
\draw (0,0)--(0,.5);
\fill (0,0) circle (.1) coordinate (v0);
\fill (-1,-1) circle (.1) coordinate (v1);
\fill (1,-1) circle (.1) coordinate (v2) node[below,blue] {1};
\fill (-1.7,-2) circle (.1) coordinate (v3) node[below,red] {2};
\fill (-.3,-2) circle (.1) coordinate (v4);
\fill (-.8,-3) circle (.1) coordinate (v5) node[below,blue] {3};
\fill (.2,-3) circle (.1) coordinate (v6) node[below,blue] {3};
\draw (v2)--(v0)--(v1)--(v4)--(v5);
\draw (v4)--(v6);
\draw (v1)--(v3);
\end{scope}

\begin{scope}[shift={(8,0)}]
\draw (0,0)--(0,.5);
\fill (0,0) circle (.1) coordinate (v0);
\fill (-1,-1) circle (.1) coordinate (v1);
\fill (1,-1) circle (.1) coordinate (v2);
\fill (-1.7,-2) circle (.1) coordinate (v3) node[below,red] {2};
\fill (-.3,-2) circle (.1) coordinate (v4) node[below,blue] {2};
\fill (.3,-2) circle (.1) coordinate (v5) node[below,blue] {2};
\fill (1.7,-2) circle (.1) coordinate (v6) node[below,blue] {2};
\draw (v6)--(v2)--(v0)--(v1)--(v3);
\draw (v1)--(v4);
\draw (v2)--(v5);
\end{scope}

\begin{scope}[shift={(12,0)}]
\draw (0,0)--(0,.5);
\fill (0,0) circle (.1) coordinate (v0);
\fill (1,-1) circle (.1) coordinate (v1);
\fill (-1,-1) circle (.1) coordinate (v2) node[below,red] {1};
\fill (1.7,-2) circle (.1) coordinate (v3) node[below,blue] {2};
\fill (.3,-2) circle (.1) coordinate (v4);
\fill (.8,-3) circle (.1) coordinate (v5) node[below,blue] {3};
\fill (-.2,-3) circle (.1) coordinate (v6) node[below,blue] {3};
\draw (v2)--(v0)--(v1)--(v4)--(v5);
\draw (v4)--(v6);
\draw (v1)--(v3);
\end{scope}

\begin{scope}[shift={(16,0)}]
\draw (0,0)--(0,.5);
\fill (0,0) circle (.1) coordinate (v0);
\fill (1,-1) circle (.1) coordinate (v1);
\fill (-1,-1) circle (.1) coordinate (v2) node[below,red] {1};
\fill (1.7,-2) circle (.1) coordinate (v3);
\fill (.3,-2) circle (.1) coordinate (v4) node[below,blue] {2};
\fill (2.2,-3) circle (.1) coordinate (v5) node[below,blue] {3};
\fill (1.2,-3) circle (.1) coordinate (v6) node[below,blue] {3};
\draw (v2)--(v0)--(v1)--(v3)--(v5);
\draw (v3)--(v6);
\draw (v1)--(v4);
\end{scope}
\end{tikzpicture}
\caption{The five binary trees of size 3, with each leaf labeled by its depth. The depths of the $0$th leaf, in red, are encoded by the polynomial $2y+2y^2+y^3$.}
\label{fig:n=3}
\end{figure}

\begin{figure}[htb]
\centering
\begin{tikzpicture}
\begin{axis}
[
    axis lines = left,
    xlabel = \(r\),
    ylabel = {depth},
    view={130}{50},
    scaled z ticks=false,
	ztick={0},
    xtick={0,10,20,30},
    ytick={0,10,20,30},
    xticklabels={30,20,10,0},
    yticklabels={0,10,20,30},
    axis line style={-},
]
\addplot3[
    surf,
] 
coordinates {
(0, 1, 1002242216651368) (0, 2, 1002242216651368) (0, 3, 738494264901008) (0, 4, 474746313150648) (0, 5, 280531912316292) (0, 6, 155851062397940) (0, 7, 82336410323440) (0, 8, 41620603020640) (0, 9, 20198233818840) (0, 10, 9425842448792) (0, 11, 4232010895376) (0, 12, 1827459250276) (0, 13, 758201178306) (0, 14, 301758997386) (0, 15, 114955808528) (0, 16, 41802112192) (0, 17, 14460614392) (0, 18, 4739192952) (0, 19, 1464140912) (0, 20, 423830264) (0, 21, 114108148) (0, 22, 28312548) (0, 23, 6399888) (0, 24, 1298528) (0, 25, 231880) (0, 26, 35464) (0, 27, 4464) (0, 28, 434) (0, 29, 29) (0, 30, 1) 

(1, 1, 0) (1, 2, 527495903500720) (1, 3, 791243855251080) (1, 4, 776857603337424) (1, 5, 623404249591760) (1, 6, 441087912447000) (1, 7, 285010651119600) (1, 8, 171378953614400) (1, 9, 96951522330432) (1, 10, 51938315534160) (1, 11, 26450068096100) (1, 12, 12831096863640) (1, 13, 5933748351960) (1, 14, 2615244644012) (1, 15, 1097305445040) (1, 16, 437463964800) (1, 17, 165264164480) (1, 18, 58950936720) (1, 19, 19765902312) (1, 20, 6194442320) (1, 21, 1801707600) (1, 22, 482078520) (1, 23, 117331280) (1, 24, 25599552) (1, 25, 4910400) (1, 26, 806000) (1, 27, 108810) (1, 28, 11340) (1, 29, 812) (1, 30, 30) 

(2, 1, 0) (2, 2, 139067101832008) (2, 3, 417201305496024) (2, 4, 621866096871432) (2, 5, 675656202297020) (2, 6, 607743300181500) (2, 7, 479712044943264) (2, 8, 342827857822112) (2, 9, 225931672573596) (2, 10, 138870531605520) (2, 11, 80182121323700) (2, 12, 43681516714044) (2, 13, 22509284783622) (2, 14, 10983474599006) (2, 15, 5074581991680) (2, 16, 2217551073600) (2, 17, 914696842064) (2, 18, 355066026552) (2, 19, 129185242776) (2, 20, 43823696840) (2, 21, 13768454700) (2, 22, 3972173964) (2, 23, 1040797472) (2, 24, 244154016) (2, 25, 50299050) (2, 26, 8859500) (2, 27, 1282554) (2, 28, 143262) (2, 29, 10991) (2, 30, 435) 

(3, 1, 0) (3, 2, 73469412288608) (3, 3, 220408236865824) (3, 4, 401920902520032) (3, 5, 540216266828000) (3, 6, 590268957672880) (3, 7, 554570584775144) (3, 8, 463462865075392) (3, 9, 351991126709376) (3, 10, 246364889589280) (3, 11, 160390045421760) (3, 12, 97717626772304) (3, 13, 55928609906832) (3, 14, 30136605041816) (3, 15, 15300242891920) (3, 16, 7316199511040) (3, 17, 3290263131904) (3, 18, 1388188879392) (3, 19, 547479118816) (3, 20, 200847963360) (3, 21, 68104356320) (3, 22, 21168947184) (3, 23, 5967299492) (3, 24, 1504090336) (3, 25, 332595200) (3, 26, 62826400) (3, 27, 9747504) (3, 28, 1166312) (3, 29, 95816) (3, 30, 4060) 

(4, 1, 0) (4, 2, 48619464014520) (4, 3, 145858392043560) (4, 4, 275642593943544) (4, 5, 405823689427320) (4, 6, 498470896473660) (4, 7, 530454608146260) (4, 8, 501662230520640) (4, 9, 429284744728812) (4, 10, 336689874218880) (4, 11, 244249362124020) (4, 12, 164942230002660) (4, 13, 104134101450210) (4, 14, 61626195581202) (4, 15, 34229610805320) (4, 16, 17845765922880) (4, 17, 8723975327760) (4, 18, 3990354009480) (4, 19, 1702137911112) (4, 20, 674015576760) (4, 21, 246250313190) (4, 22, 82343036160) (4, 23, 24937362450) (4, 24, 6745222512) (4, 25, 1599084450) (4, 26, 323584300) (4, 27, 53747010) (4, 28, 6881490) (4, 29, 604737) (4, 30, 27405) 

(5, 1, 0) (5, 2, 36117316125072) (5, 3, 108351948375216) (5, 4, 207811791564624) (5, 5, 316712635321680) (5, 6, 411945984070920) (5, 7, 472214696996952) (5, 8, 485996226084096) (5, 9, 454853008419264) (5, 10, 390816812091888) (5, 11, 310487062655484) (5, 12, 229287606580296) (5, 13, 157983330328200) (5, 14, 101808510158628) (5, 15, 61437548495232) (5, 16, 34723341813888) (5, 17, 18363280091520) (5, 18, 9068984284464) (5, 19, 4169534479812) (5, 20, 1776708457056) (5, 21, 697520128032) (5, 22, 250317371304) (5, 23, 81267414228) (5, 24, 23542130976) (5, 25, 5972377920) (5, 26, 1292381584) (5, 27, 229425534) (5, 28, 31380804) (5, 29, 2945124) (5, 30, 142506) 

(6, 1, 0) (6, 2, 28817007546600) (6, 3, 86451022639800) (6, 4, 167138643770280) (6, 5, 259353067919400) (6, 6, 347532685511700) (6, 7, 415538270319900) (6, 8, 451003330814400) (6, 9, 448906755941280) (6, 10, 412603705344600) (6, 11, 351942012160200) (6, 12, 279613079984700) (6, 13, 207451321219350) (6, 14, 143964839859750) (6, 15, 93517651336200) (6, 16, 56853066379200) (6, 17, 32312059825500) (6, 18, 17132802850800) (6, 19, 8448450634260) (6, 20, 3857435407800) (6, 21, 1621160991450) (6, 22, 622256434800) (6, 23, 215903052750) (6, 24, 66794775824) (6, 25, 18085300100) (6, 26, 4174606800) (6, 27, 790159500) (6, 28, 115192350) (6, 29, 11519235) (6, 30, 593775) 

(7, 1, 0) (7, 2, 24151396800960) (7, 3, 72454190402880) (7, 4, 140772517928640) (7, 5, 220834653624000) (7, 6, 301244869317600) (7, 7, 369507155340960) (7, 8, 414693578388480) (7, 9, 430022829277440) (7, 10, 414453629966400) (7, 11, 372649465425600) (7, 12, 313336564596000) (7, 13, 246746443600800) (7, 14, 182109680953200) (7, 15, 125968654398600) (7, 16, 81607686825600) (7, 17, 49441802891520) (7, 18, 27946829252160) (7, 19, 14689045500480) (7, 20, 7146776644800) (7, 21, 3199496529600) (7, 22, 1307675131488) (7, 23, 482935257024) (7, 24, 158964064512) (7, 25, 45776544000) (7, 26, 11234142400) (7, 27, 2259990720) (7, 28, 350073360) (7, 29, 37187280) (7, 30, 2035800) 

(8, 1, 0) (8, 2, 20992057103160) (8, 3, 62976171309480) (8, 4, 122762180094840) (8, 5, 193969758411000) (8, 6, 267699749183100) (8, 7, 333876824304660) (8, 8, 383063736310080) (8, 9, 408347681084640) (8, 10, 406784298904200) (8, 11, 379939341533400) (8, 12, 333323100861300) (8, 13, 274892057540100) (8, 14, 213116802648750) (8, 15, 155224265045850) (8, 16, 106080957309600) (8, 17, 67889609025420) (8, 18, 40576256847360) (8, 19, 22566255559380) (8, 20, 11622403632840) (8, 21, 5509368687690) (8, 22, 2384568797688) (8, 23, 932610218694) (8, 24, 325080586512) (8, 25, 99122591700) (8, 26, 25754185600) (8, 27, 5484311820) (8, 28, 899090010) (8, 29, 101060505) (8, 30, 5852925) 

(9, 1, 0) (9, 2, 18773384401200) (9, 3, 56320153203600) (9, 4, 110040914720880) (9, 5, 174736885580400) (9, 6, 243097764538200) (9, 7, 306677480109000) (9, 8, 357232261267200) (9, 9, 388169929730880) (9, 10, 395777704357200) (9, 11, 379902085714800) (9, 12, 343871333341800) (9, 13, 293659118181000) (9, 14, 236518826844300) (9, 15, 179477662349700) (9, 16, 128098514315200) (9, 17, 85791475827200) (9, 18, 53748393064080) (9, 19, 31374930109380) (9, 20, 16978678196480) (9, 21, 8463413407680) (9, 22, 3854370488760) (9, 23, 1586847619940) (9, 24, 582436837664) (9, 25, 187039864000) (9, 26, 51186049200) (9, 27, 11480747850) (9, 28, 1982280300) (9, 29, 234637260) (9, 30, 14307150) 

(10, 1, 0) (10, 2, 17184867259560) (10, 3, 51554601778680) (10, 4, 100896207932520) (10, 5, 160783694471400) (10, 6, 224958616096500) (10, 7, 286091156611260) (10, 8, 336833868127680) (10, 9, 370992295285920) (10, 10, 384575618833560) (10, 11, 376482940116360) (10, 12, 348642961335900) (10, 13, 305551206855180) (10, 14, 253307735570890) (10, 15, 198397668298830) (10, 16, 146524296094176) (10, 17, 101772005246912) (10, 18, 66256231185144) (10, 19, 40258749579048) (10, 20, 22710326133352) (10, 21, 11814826167606) (10, 22, 5621104719384) (10, 23, 2419509225210) (10, 24, 929018343792) (10, 25, 312235818510) (10, 26, 89454024660) (10, 27, 21008125710) (10, 28, 3798104310) (10, 29, 470705235) (10, 30, 30045015) 

(11, 1, 0) (11, 2, 16044839210400) (11, 3, 48134517631200) (11, 4, 94314701463840) (11, 5, 150676723111200) (11, 6, 211672078347600) (11, 7, 270741420658800) (11, 8, 321191043264000) (11, 9, 357182829843840) (11, 10, 374655121999200) (11, 11, 371980371444000) (11, 12, 350202793393200) (11, 13, 312784488073200) (11, 14, 264904035404008) (11, 15, 212465435702520) (11, 16, 161050461586560) (11, 17, 115056135434240) (11, 18, 77195588502240) (11, 19, 48426544711584) (11, 20, 28248165156640) (11, 21, 15217226388000) (11, 22, 7505522104080) (11, 23, 3352447213500) (11, 24, 1336831897632) (11, 25, 466893081600) (11, 26, 139061551200) (11, 27, 33961496400) (11, 28, 6385743000) (11, 29, 823051320) (11, 30, 54627300) 

(12, 1, 0) (12, 2, 15242597249880) (12, 3, 45727791749640) (12, 4, 89673981223320) (12, 5, 143517961119000) (12, 6, 202187996076300) (12, 7, 259649962861380) (12, 8, 309671686309440) (12, 9, 346695711354720) (12, 10, 366669961239400) (12, 11, 367673109947000) (12, 12, 350196602253476) (12, 13, 317008585939188) (12, 14, 272615295086134) (12, 15, 222430466672050) (12, 16, 171835532712480) (12, 17, 125338541077568) (12, 18, 86007255936264) (12, 19, 55270233768152) (12, 20, 33075374019480) (12, 21, 18303532538200) (12, 22, 9284849778204) (12, 23, 4269610837492) (12, 24, 1754283346816) (12, 25, 631723682150) (12, 26, 194097614100) (12, 27, 48915625590) (12, 28, 9492780570) (12, 29, 1262801085) (12, 30, 86493225) 

(13, 1, 0) (13, 2, 14709639304080) (13, 3, 44128917912240) (13, 4, 86586754781520) (13, 5, 138740987826000) (13, 6, 195826093921800) (13, 7, 252148628917080) (13, 8, 301782544887040) (13, 9, 339368699113920) (13, 10, 360886728288240) (13, 11, 364254723996240) (13, 12, 349630969004856) (13, 13, 319343446602648) (13, 14, 277448662278564) (13, 15, 229004325115020) (13, 16, 179208222984000) (13, 17, 132587203614208) (13, 18, 92403601408944) (13, 19, 60384505909392) (13, 20, 36790237342800) (13, 21, 20750609491920) (13, 22, 10739091470184) (13, 23, 5042599030392) (13, 24, 2117189402496) (13, 25, 779545771200) (13, 26, 245012440400) (13, 27, 63183762990) (13, 28, 12549180420) (13, 29, 1708573860) (13, 30, 119759850) 

(14, 1, 0) (14, 2, 14404600701000) (14, 3, 43213802103000) (14, 4, 84818262610440) (14, 5, 135999299032200) (14, 6, 192162844364100) (14, 7, 247807363353900) (14, 8, 297181573816000) (14, 9, 335043669630240) (14, 10, 357400091852280) (14, 11, 362089249907880) (14, 12, 349091383697100) (14, 13, 320491259316540) (14, 14, 280088751236658) (14, 15, 232728783511110) (14, 16, 183486834549600) (14, 17, 136880064437440) (14, 18, 96265043253720) (14, 19, 63531400782024) (14, 20, 39120758428680) (14, 21, 22316511437640) (14, 22, 11688735078900) (14, 23, 5557920025020) (14, 24, 2364253456512) (14, 25, 882333415200) (14, 26, 281174277800) (14, 27, 73533727800) (14, 28, 14813054550) (14, 29, 2045612295) (14, 30, 145422675) 

(15, 1, 0) (15, 2, 14305258627200) (15, 3, 42915775881600) (15, 4, 84242078582400) (15, 5, 135105220368000) (15, 6, 190966372266560) (15, 7, 246386003837120) (15, 8, 295669654640640) (15, 9, 333614295559680) (15, 10, 356236415015808) (15, 11, 361350372206208) (15, 12, 348880439827392) (15, 13, 320833646028480) (15, 14, 280926575363616) (15, 15, 233933618110304) (15, 16, 184887830305280) (15, 17, 138299974483968) (15, 18, 97554425948544) (15, 19, 64592143771776) (15, 20, 39913905045120) (15, 21, 22854721378944) (15, 22, 12018461480256) (15, 23, 5738713506240) (15, 24, 2451854322688) (15, 25, 919170557440) (15, 26, 294273982080) (15, 27, 77323409280) (15, 28, 15650822880) (15, 29, 2171645280) (15, 30, 155117520) 

(16, 1, 0) (16, 2, 14404600701000) (16, 3, 43213802103000) (16, 4, 84818262610440) (16, 5, 135999299032200) (16, 6, 192162844364100) (16, 7, 247807363353900) (16, 8, 297181573816000) (16, 9, 335043669630240) (16, 10, 357400091852280) (16, 11, 362089249907880) (16, 12, 349091383697100) (16, 13, 320491259316540) (16, 14, 280088751236658) (16, 15, 232728783511110) (16, 16, 183486834549600) (16, 17, 136880064437440) (16, 18, 96265043253720) (16, 19, 63531400782024) (16, 20, 39120758428680) (16, 21, 22316511437640) (16, 22, 11688735078900) (16, 23, 5557920025020) (16, 24, 2364253456512) (16, 25, 882333415200) (16, 26, 281174277800) (16, 27, 73533727800) (16, 28, 14813054550) (16, 29, 2045612295) (16, 30, 145422675) 

(17, 1, 0) (17, 2, 14709639304080) (17, 3, 44128917912240) (17, 4, 86586754781520) (17, 5, 138740987826000) (17, 6, 195826093921800) (17, 7, 252148628917080) (17, 8, 301782544887040) (17, 9, 339368699113920) (17, 10, 360886728288240) (17, 11, 364254723996240) (17, 12, 349630969004856) (17, 13, 319343446602648) (17, 14, 277448662278564) (17, 15, 229004325115020) (17, 16, 179208222984000) (17, 17, 132587203614208) (17, 18, 92403601408944) (17, 19, 60384505909392) (17, 20, 36790237342800) (17, 21, 20750609491920) (17, 22, 10739091470184) (17, 23, 5042599030392) (17, 24, 2117189402496) (17, 25, 779545771200) (17, 26, 245012440400) (17, 27, 63183762990) (17, 28, 12549180420) (17, 29, 1708573860) (17, 30, 119759850) 

(18, 1, 0) (18, 2, 15242597249880) (18, 3, 45727791749640) (18, 4, 89673981223320) (18, 5, 143517961119000) (18, 6, 202187996076300) (18, 7, 259649962861380) (18, 8, 309671686309440) (18, 9, 346695711354720) (18, 10, 366669961239400) (18, 11, 367673109947000) (18, 12, 350196602253476) (18, 13, 317008585939188) (18, 14, 272615295086134) (18, 15, 222430466672050) (18, 16, 171835532712480) (18, 17, 125338541077568) (18, 18, 86007255936264) (18, 19, 55270233768152) (18, 20, 33075374019480) (18, 21, 18303532538200) (18, 22, 9284849778204) (18, 23, 4269610837492) (18, 24, 1754283346816) (18, 25, 631723682150) (18, 26, 194097614100) (18, 27, 48915625590) (18, 28, 9492780570) (18, 29, 1262801085) (18, 30, 86493225) 

(19, 1, 0) (19, 2, 16044839210400) (19, 3, 48134517631200) (19, 4, 94314701463840) (19, 5, 150676723111200) (19, 6, 211672078347600) (19, 7, 270741420658800) (19, 8, 321191043264000) (19, 9, 357182829843840) (19, 10, 374655121999200) (19, 11, 371980371444000) (19, 12, 350202793393200) (19, 13, 312784488073200) (19, 14, 264904035404008) (19, 15, 212465435702520) (19, 16, 161050461586560) (19, 17, 115056135434240) (19, 18, 77195588502240) (19, 19, 48426544711584) (19, 20, 28248165156640) (19, 21, 15217226388000) (19, 22, 7505522104080) (19, 23, 3352447213500) (19, 24, 1336831897632) (19, 25, 466893081600) (19, 26, 139061551200) (19, 27, 33961496400) (19, 28, 6385743000) (19, 29, 823051320) (19, 30, 54627300) 

(20, 1, 0) (20, 2, 17184867259560) (20, 3, 51554601778680) (20, 4, 100896207932520) (20, 5, 160783694471400) (20, 6, 224958616096500) (20, 7, 286091156611260) (20, 8, 336833868127680) (20, 9, 370992295285920) (20, 10, 384575618833560) (20, 11, 376482940116360) (20, 12, 348642961335900) (20, 13, 305551206855180) (20, 14, 253307735570890) (20, 15, 198397668298830) (20, 16, 146524296094176) (20, 17, 101772005246912) (20, 18, 66256231185144) (20, 19, 40258749579048) (20, 20, 22710326133352) (20, 21, 11814826167606) (20, 22, 5621104719384) (20, 23, 2419509225210) (20, 24, 929018343792) (20, 25, 312235818510) (20, 26, 89454024660) (20, 27, 21008125710) (20, 28, 3798104310) (20, 29, 470705235) (20, 30, 30045015) 

(21, 1, 0) (21, 2, 18773384401200) (21, 3, 56320153203600) (21, 4, 110040914720880) (21, 5, 174736885580400) (21, 6, 243097764538200) (21, 7, 306677480109000) (21, 8, 357232261267200) (21, 9, 388169929730880) (21, 10, 395777704357200) (21, 11, 379902085714800) (21, 12, 343871333341800) (21, 13, 293659118181000) (21, 14, 236518826844300) (21, 15, 179477662349700) (21, 16, 128098514315200) (21, 17, 85791475827200) (21, 18, 53748393064080) (21, 19, 31374930109380) (21, 20, 16978678196480) (21, 21, 8463413407680) (21, 22, 3854370488760) (21, 23, 1586847619940) (21, 24, 582436837664) (21, 25, 187039864000) (21, 26, 51186049200) (21, 27, 11480747850) (21, 28, 1982280300) (21, 29, 234637260) (21, 30, 14307150) 

(22, 1, 0) (22, 2, 20992057103160) (22, 3, 62976171309480) (22, 4, 122762180094840) (22, 5, 193969758411000) (22, 6, 267699749183100) (22, 7, 333876824304660) (22, 8, 383063736310080) (22, 9, 408347681084640) (22, 10, 406784298904200) (22, 11, 379939341533400) (22, 12, 333323100861300) (22, 13, 274892057540100) (22, 14, 213116802648750) (22, 15, 155224265045850) (22, 16, 106080957309600) (22, 17, 67889609025420) (22, 18, 40576256847360) (22, 19, 22566255559380) (22, 20, 11622403632840) (22, 21, 5509368687690) (22, 22, 2384568797688) (22, 23, 932610218694) (22, 24, 325080586512) (22, 25, 99122591700) (22, 26, 25754185600) (22, 27, 5484311820) (22, 28, 899090010) (22, 29, 101060505) (22, 30, 5852925) 

(23, 1, 0) (23, 2, 24151396800960) (23, 3, 72454190402880) (23, 4, 140772517928640) (23, 5, 220834653624000) (23, 6, 301244869317600) (23, 7, 369507155340960) (23, 8, 414693578388480) (23, 9, 430022829277440) (23, 10, 414453629966400) (23, 11, 372649465425600) (23, 12, 313336564596000) (23, 13, 246746443600800) (23, 14, 182109680953200) (23, 15, 125968654398600) (23, 16, 81607686825600) (23, 17, 49441802891520) (23, 18, 27946829252160) (23, 19, 14689045500480) (23, 20, 7146776644800) (23, 21, 3199496529600) (23, 22, 1307675131488) (23, 23, 482935257024) (23, 24, 158964064512) (23, 25, 45776544000) (23, 26, 11234142400) (23, 27, 2259990720) (23, 28, 350073360) (23, 29, 37187280) (23, 30, 2035800) 

(24, 1, 0) (24, 2, 28817007546600) (24, 3, 86451022639800) (24, 4, 167138643770280) (24, 5, 259353067919400) (24, 6, 347532685511700) (24, 7, 415538270319900) (24, 8, 451003330814400) (24, 9, 448906755941280) (24, 10, 412603705344600) (24, 11, 351942012160200) (24, 12, 279613079984700) (24, 13, 207451321219350) (24, 14, 143964839859750) (24, 15, 93517651336200) (24, 16, 56853066379200) (24, 17, 32312059825500) (24, 18, 17132802850800) (24, 19, 8448450634260) (24, 20, 3857435407800) (24, 21, 1621160991450) (24, 22, 622256434800) (24, 23, 215903052750) (24, 24, 66794775824) (24, 25, 18085300100) (24, 26, 4174606800) (24, 27, 790159500) (24, 28, 115192350) (24, 29, 11519235) (24, 30, 593775) 

(25, 1, 0) (25, 2, 36117316125072) (25, 3, 108351948375216) (25, 4, 207811791564624) (25, 5, 316712635321680) (25, 6, 411945984070920) (25, 7, 472214696996952) (25, 8, 485996226084096) (25, 9, 454853008419264) (25, 10, 390816812091888) (25, 11, 310487062655484) (25, 12, 229287606580296) (25, 13, 157983330328200) (25, 14, 101808510158628) (25, 15, 61437548495232) (25, 16, 34723341813888) (25, 17, 18363280091520) (25, 18, 9068984284464) (25, 19, 4169534479812) (25, 20, 1776708457056) (25, 21, 697520128032) (25, 22, 250317371304) (25, 23, 81267414228) (25, 24, 23542130976) (25, 25, 5972377920) (25, 26, 1292381584) (25, 27, 229425534) (25, 28, 31380804) (25, 29, 2945124) (25, 30, 142506) 

(26, 1, 0) (26, 2, 48619464014520) (26, 3, 145858392043560) (26, 4, 275642593943544) (26, 5, 405823689427320) (26, 6, 498470896473660) (26, 7, 530454608146260) (26, 8, 501662230520640) (26, 9, 429284744728812) (26, 10, 336689874218880) (26, 11, 244249362124020) (26, 12, 164942230002660) (26, 13, 104134101450210) (26, 14, 61626195581202) (26, 15, 34229610805320) (26, 16, 17845765922880) (26, 17, 8723975327760) (26, 18, 3990354009480) (26, 19, 1702137911112) (26, 20, 674015576760) (26, 21, 246250313190) (26, 22, 82343036160) (26, 23, 24937362450) (26, 24, 6745222512) (26, 25, 1599084450) (26, 26, 323584300) (26, 27, 53747010) (26, 28, 6881490) (26, 29, 604737) (26, 30, 27405) 

(27, 1, 0) (27, 2, 73469412288608) (27, 3, 220408236865824) (27, 4, 401920902520032) (27, 5, 540216266828000) (27, 6, 590268957672880) (27, 7, 554570584775144) (27, 8, 463462865075392) (27, 9, 351991126709376) (27, 10, 246364889589280) (27, 11, 160390045421760) (27, 12, 97717626772304) (27, 13, 55928609906832) (27, 14, 30136605041816) (27, 15, 15300242891920) (27, 16, 7316199511040) (27, 17, 3290263131904) (27, 18, 1388188879392) (27, 19, 547479118816) (27, 20, 200847963360) (27, 21, 68104356320) (27, 22, 21168947184) (27, 23, 5967299492) (27, 24, 1504090336) (27, 25, 332595200) (27, 26, 62826400) (27, 27, 9747504) (27, 28, 1166312) (27, 29, 95816) (27, 30, 4060) 

(28, 1, 0) (28, 2, 139067101832008) (28, 3, 417201305496024) (28, 4, 621866096871432) (28, 5, 675656202297020) (28, 6, 607743300181500) (28, 7, 479712044943264) (28, 8, 342827857822112) (28, 9, 225931672573596) (28, 10, 138870531605520) (28, 11, 80182121323700) (28, 12, 43681516714044) (28, 13, 22509284783622) (28, 14, 10983474599006) (28, 15, 5074581991680) (28, 16, 2217551073600) (28, 17, 914696842064) (28, 18, 355066026552) (28, 19, 129185242776) (28, 20, 43823696840) (28, 21, 13768454700) (28, 22, 3972173964) (28, 23, 1040797472) (28, 24, 244154016) (28, 25, 50299050) (28, 26, 8859500) (28, 27, 1282554) (28, 28, 143262) (28, 29, 10991) (28, 30, 435) 

(29, 1, 0) (29, 2, 527495903500720) (29, 3, 791243855251080) (29, 4, 776857603337424) (29, 5, 623404249591760) (29, 6, 441087912447000) (29, 7, 285010651119600) (29, 8, 171378953614400) (29, 9, 96951522330432) (29, 10, 51938315534160) (29, 11, 26450068096100) (29, 12, 12831096863640) (29, 13, 5933748351960) (29, 14, 2615244644012) (29, 15, 1097305445040) (29, 16, 437463964800) (29, 17, 165264164480) (29, 18, 58950936720) (29, 19, 19765902312) (29, 20, 6194442320) (29, 21, 1801707600) (29, 22, 482078520) (29, 23, 117331280) (29, 24, 25599552) (29, 25, 4910400) (29, 26, 806000) (29, 27, 108810) (29, 28, 11340) (29, 29, 812) (29, 30, 30) 

(30, 1, 1002242216651368) (30, 2, 1002242216651368) (30, 3, 738494264901008) (30, 4, 474746313150648) (30, 5, 280531912316292) (30, 6, 155851062397940) (30, 7, 82336410323440) (30, 8, 41620603020640) (30, 9, 20198233818840) (30, 10, 9425842448792) (30, 11, 4232010895376) (30, 12, 1827459250276) (30, 13, 758201178306) (30, 14, 301758997386) (30, 15, 114955808528) (30, 16, 41802112192) (30, 17, 14460614392) (30, 18, 4739192952) (30, 19, 1464140912) (30, 20, 423830264) (30, 21, 114108148) (30, 22, 28312548) (30, 23, 6399888) (30, 24, 1298528) (30, 25, 231880) (30, 26, 35464) (30, 27, 4464) (30, 28, 434) (30, 29, 29) (30, 30, 1)
};
\end{axis}
\end{tikzpicture}
\caption{The distribution of the depths of each of the leaves (indexed by $0\le r\le 30$) on binary trees of size $30$, given by the coefficients of the polynomial $[z^{30}]B(x,y,z)$.}
\label{fig:distribution30}
\end{figure}
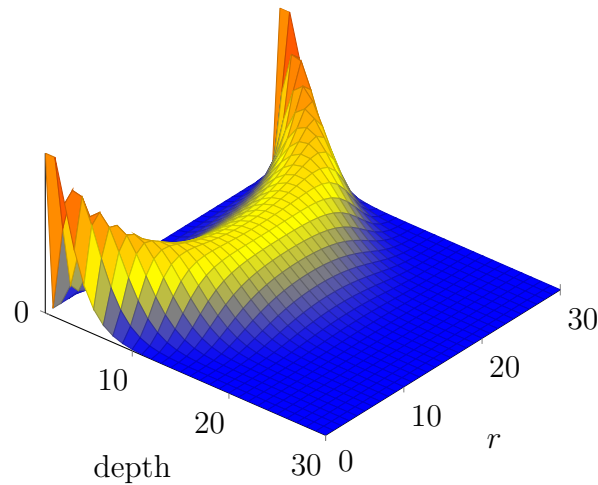

An easy special case is obtained when considering the depth of the leftmost leaf on binary trees. When $r=0$, we have
\begin{equation}\label{eq:x0} B_0(y,z)=B(0,y,z)=\frac{1}{1-yzC(z)}.\end{equation}
There are many well-known statistics on Catalan objects whose distribution is given by this generating function. 
In Section~\ref{sec:bijections} we discuss some of these statistics on Dyck paths and plane trees, and show bijectively how they correspond to each other. 
Leaf depth statistics can be thought of as generalizations of these classical statistics.

\subsection{The average depth of the $r$th leaf} \label{sec:average_depth_binary}

In this section we describe the average depth of each leaf in binary trees of size $n$, and study the asymptotic behavior of these average depths as $n\to\infty$.
The average depths can be computed using equation~\eqref{eq:average}. By Theorem~\ref{thm:B}, letting
\begin{equation} \label{eq:dBxz}
\dB(x,z)\coloneqq\left.\frac{\partial B(x,y,z)}{\partial y}\right|_{y=1}=\frac{zC(z)+xzC(xz)}{\left(1-zC(z)-xzC(xz)\right)^2}=B(x,1,z)(B(x,1,z)-1),
\end{equation}
the coefficient of $x^rz^n$ in $\dB(x,z)$ is the sum, over all binary trees of size $n$, of the depths of the $r$th leaf.
Dividing this coefficient by $c_{n}$, which is the number of such trees, will give us the average depth of this leaf.

A direct combinatorial explanation of equation~\eqref{eq:dBxz} is obtained by interpreting $\dB(x,z)$ as counting binary trees with a distinguished leaf $a_1$ (whose index is marked by the variable $x$) and a distinguished non-root node $a_2$ in the path from $a_1$ to the root. This is because the number of choices for $a_2$ is equal to the depth of $a_1$. By splitting at $a_2$, such a tree can be decomposed as a pair of binary trees, each with a distinguished leaf: the subtree rooted at $a_2$ with the distinguished leaf $a_1$, and the non-empty subtree with the original root and distinguished leaf $a_2$. The index of the leaf $a_1$ in the original tree is the sum of the indices of the distinguished leaves in these two subtrees. Finally, to obtain the product expression in~\eqref{eq:dBxz}, we note that $B(x,1,z)$ is the generating function for binary trees with a distinguished leaf, where $x$ marks the index of this leaf and $z$ marks the size.

Using the notation $$[k]_x=\frac{1-x^k}{1-x}=1+x+x^2+\dots+x^{k-1},$$ this interpretation of $B(x,1,z)$ allows us to rewrite it as
\begin{equation} \label{eq:Bx1z}
B(x,1,z)=\frac{1}{1-zC(z)-xzC(xz)}=\frac{C(z)-xC(xz)}{1-x}=\sum_{n\ge0}c_{n}[n+1]_x z^n.
\end{equation}
Extracting the coefficient of $z^n$ in equation~\eqref{eq:dBxz} and using equation~\eqref{eq:Bx1z}, we get
$$[z^n]\dB(x,z)=\sum_{i=1}^{n}c_{i}c_{n-i}[i+1]_x[n+1-i]_x.$$
To extract now the coefficient of $x^r$, we use the fact that, 
for $0\le i,r\le n$, the coefficient of $x^r$ in $[i+1]_x[n+1-i]_x$ equals $\min(r+1,n+1-r,i+1,n+1-i)$.
It follows that, for $0\le r\le n$,
\begin{align}\nonumber [x^rz^n]\dB(x,z)&=\sum_{i=1}^{n}c_{i}c_{n-i}\min(r+1,n+1-r,i+1,n+1-i)\\
\nonumber &= c_{n+1}-c_{n}+\sum_{i=0}^{n}c_{i}c_{n-i}\min(r,n-r,i,n-i)\\
\label{eq:coefdB_sum}  &=c_{n+1}-c_{n}+2\sum_{i=0}^{r-1}ic_ic_{n-i}+r\sum_{i=r}^{n-r}c_ic_{n-i}\\                                                                                                          
\nonumber &=(r+1)c_{n+1}-c_{n}-2\sum_{i=0}^{r-1}(r-i)c_ic_{n-i}\\
&=-c_{n}+\frac{2(2r+1)(2(n-r)+1)}{(n+1)(n+2)}\binom{2r}{r}\binom{2(n-r)}{n-r}.
\label{eq:coefdB} 
\end{align}
The second equality above uses the well-known recurrence $c_{n+1}=\sum_{i=0}^{n}c_ic_{n-i}$.
In equation~\eqref{eq:coefdB_sum} one can assume, by symmetry, that $0\le r\le n/2$; alternatively, if allowing $n/2< r\le n$, the summation from $r$ to $n-r$ should be interpreted in this case as a summation from $n-r+1$ to $r-1$ with negative sign.
The closed form in equation~\eqref{eq:coefdB} was obtained using Maple, and it can be verified by induction on~$r$, for example by first showing that 
$$2\sum_{i=0}^{r-1}c_ic_{n-i}=c_{n+1}-\frac{n-2r+1}{(n+1)(n+2)}\binom{2r}{r}\binom{2(n-r+1)}{n-r+1}.$$

\begin{figure}[h]
\centering
\begin{tikzpicture}[scale=1.5] 
\draw[->] (-.3,0)--(2.3,0) node[right] {$i$}; 
\draw[->] (0,-.3)--(0,2.3) node[above] {$r$}; 
\draw (0,0) rectangle (2,2);
\draw (0,0)--(2,2);
\draw (0,2) node[left,scale=.7]{$n$}--(2,0)node[below,scale=.7]{$n$};
\draw (1,.4) node{$r$};
\draw (1,1.6) node{$n-r$};
\draw (.4,1) node{$i$};
\draw (1.6,1) node{$n-i$};
\end{tikzpicture}
\caption{A diagram of the values of the function $\min(r,n-r,i,n-i)$ for $0\le i,r\le n$.}
\end{figure}
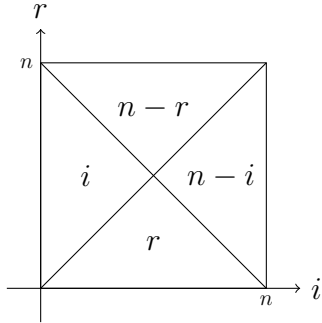

Dividing equation~\eqref{eq:coefdB} by $c_{n}$, we obtain the following formula for the average depth of the $r$th leaf over binary trees of size $n$, which appeared first in~\cite{kirschenhofer_height_1983}, see also~\cite[Thm.~1 \& Tab.~1]{panholzer_descendants_1997}.
Note that, for fixed $r$, the average depth of the $r$th leaf is finite even as $n\to\infty$, despite the fact that the depth of this leaf can take arbitrarily large values.

\begin{theorem}[\cite{kirschenhofer_height_1983,panholzer_descendants_1997}]
\label{thm:average_depths}
For $0\le r\le n$, the average depth of the $r$th leaf over all binary trees of size $n$ equals
\begin{equation}\label{eq:avg_depth}
\frac{2(2r+1)(2(n-r)+1)}{n+2}\cdot\frac{\binom{2r}{r}\binom{2(n-r)}{n-r}}{\binom{2n}{n}}-1.
\end{equation}
As $n\to\infty$, this average is asymptotically equal to
\begin{numcases}{}
   \displaystyle\frac{2r+1}{4^{r-1}}\binom{2r}{r}-1 & if $r$ is fixed, \label{eq:avg_depth_r_fixed}
   \\
   \displaystyle\frac{8}{\sqrt{\pi}}\sqrt{r\left(1-\frac{r}{n}\right)} & if $r=r(n)$ is such that $r\to\infty$ and $n-r\to\infty$. \label{eq:avg_depth_r_inf}
\end{numcases}
\end{theorem}

\begin{proof}
By equation~\eqref{eq:average}, the average depth of the $r$th leaf is $\frac{[x^rz^n]\dB(x,z)}{c_n}$, so the first formula follows from equation~\eqref{eq:coefdB}.
The asymptotic expressions in both cases are obtained by letting $n\to\infty$  and using the approximation
\begin{equation}\label{eq:Ckapprox} c_k=\frac{1}{k+1}\binom{2k}{k}\sim \frac{4^k}{\sqrt{\pi k^3}} \end{equation}
for $k\to\infty$, where the notation $f(k)\sim g(k)$ means that $\lim_{k\to\infty} f(k)/g(k)=1$. In the case of fixed $r$, the average depth of the $r$th leaf is finite in the limit as $n\to\infty$, but this is no longer true when $r=r(n)$ is such that $r\to\infty$ and $n-r\to\infty$.

In the latter case, let us include an alternative proof of equation~\eqref{eq:avg_depth_r_inf} that does not rely on the closed form~\eqref{eq:coefdB}, since this method will be useful for other situations where we do not have a closed form. From the expression~\eqref{eq:coefdB_sum} and the approximation~\eqref{eq:Ckapprox}, noting that the errors arising from approximating $c_i$ for small values of $i$ are asymptotically negligible, it follows that, as $r\to\infty$ and $n-r\to\infty$,
\begin{align}\label{eq:integral}\frac{[x^rz^n]\dB(x,z)}{c_{n}}&\sim 3+\frac{n^{3/2}}{\sqrt{\pi}}\left(\sum_{i=1}^{r-1}\frac{2}{\sqrt{i(n-i)^3}}+r \sum_{i=r}^{n-r}\frac{1}{\sqrt{i^3(n-i)^3}}\right)\\
\nonumber &\sim \frac{2 n^{3/2}}{\sqrt{\pi}}\left(\int_{0}^{r}\frac{1}{\sqrt{t(n-t)^3}}\,dt+r \int_{r}^{n/2}\frac{1}{\sqrt{t^3(n-t)^3}}\,dt\right)\\
\nonumber &=\frac{2 n^{3/2}}{\sqrt{\pi}}\left(\frac{2\sqrt{r}}{n\sqrt{n-r}}+r \frac{2(n-2r)}{n^2\sqrt{r(n-r)}}\right)\\
\nonumber &=\frac{8}{\sqrt{\pi}}\sqrt{r\left(1-\frac{r}{n}\right)}.
\qedhere
\end{align}
\end{proof}

The first few values of the sequence of asymptotic average depths given by equation~\eqref{eq:avg_depth_r_fixed} appear in Table~\ref{tab:small_r}, and are plotted on the right of Figure~\ref{fig:dist_binary_smallr}. 
 Letting $r=r(n)=\alpha n$ for some constant $0<\alpha<1$, equation~\eqref{eq:avg_depth_r_inf} implies that the average depth of the $r$th leaf is asymptotically given by
\begin{equation}\label{eq:alpha}
\frac{8}{\sqrt{\pi}}\sqrt{\alpha(1-\alpha)n}.
\end{equation}

\begin{table}[htb]
\centering
\begin{tabular}{c|c|c|c|c|c|c|c|c|c}
& & \multicolumn{8}{c}{$r$} \\ \cline{3-10}
Theorem & Statistic  & 0&1&2&3&4&5&6&7\\ \hline

\ref{thm:average_depths} & \btc  depth of $r$th leaf\\  in binary trees\et & 3 &  5 &  $\frac{13}{2}$ &  $\frac{31}{4}$ & $\frac{283}{32}$ & $\frac{629}{64}$ & $\frac{2747}{256}$ & $\frac{5923}{512}$ \\ \hline

\ref{thm:average_depths_Schroeder} & \btc depth of $r$th leaf \\ in Schr\"oder trees \et &
{\scriptsize \btc $\sqrt{2}$ \\ $+1$ \et} &{\scriptsize \btc $2\sqrt{2}$ \\ $+1$ \et} & {\scriptsize \btc $7\sqrt{2}$ \\ $- 5$ \et} & {\scriptsize \btc $84\sqrt{2}$ \\ $-113$\et} & {\scriptsize \btc $1701\sqrt{2}$ \\ $-2399$\et} & {\scriptsize \btc $40038\sqrt{2}$ \\ $+56615$\et} & {\scriptsize \btc $1015307\sqrt{2}$ \\ $-1435853$\et } & {\scriptsize \btc $27021736\sqrt{2}$ \\ $-38214497$\et } \\ \hline

\ref{thm:average_depths_noncrossing} & \btc depth of node $r$ \\ in noncrossing trees \et &
0 & 2 & $\frac{28}{9}$ & $\frac{962}{243}$ & $\frac{30640}{6561}$ & 
$\frac{312634}{59049}$ & $\frac{28017284}{4782969}$ & $\frac{823239002}{129140163}$ \\ \hline

\ref{thm:average_heights} & \btc height of $r$th vertex\\ in Dyck paths \et &
0 & 1 & $\frac32$ & 2 & $\frac{19}{8}$ & $\frac{11}{4}$ & $\frac{49}{16}$ & $\frac{27}{8}$ 
\\ \hline

\ref{thm:average_height_upstep} & \btc height of $r$th up-step\\ in Dyck paths \et &
& 1 & $\frac{7}{4}$ & $\frac{19}{8}$ & $\frac{187}{64}$ & $\frac{437}{128}$ & $\frac{1979}{512}$ & $\frac{4387}{1024}$
\\ \hline

\ref{thm:average_height_upstep} & \btc height of $r$th down-step\\ in Dyck paths \et &
& 3 & 4 & $\frac{19}{4}$ & $\frac{43}{8}$ & $\frac{379}{64}$ & $\frac{821}{128}$ & $\frac{3515}{512}$
\end{tabular}
\caption{Asymptotic averages as $n\to\infty$ of various statistics $s_r$ for small $r$.}
\label{tab:small_r}
\end{table}

\begin{figure}[h]
\centering
\begin{tikzpicture}[scale=.9]
\begin{axis}[
    axis lines = left,
    xlabel = {depth},
    ymin=0,
    ytick={0,0.05,0.1,0.15,0.2,0.25},
    yticklabels={0,0.05,0.1,0.15,0.2,0.25},    
]
\addplot[
    color=blue!0!red,
    mark=*,
    ]
    coordinates {
    (1, 1/4) (2, 1/4) (3, 3/16) (4, 1/8) (5, 5/64) (6, 3/64) (7, 7/256) (8, 1/64) (9, 9/1024) (10, 5/1024) (11, 11/4096) (12, 3/2048) (13, 13/16384) (14, 7/16384) (15, 15/65536) (16, 1/8192) (17, 17/262144) (18, 9/262144) (19, 19/1048576) (20, 5/524288)
    };
   \addlegendentry{$r=0$}
    
\addplot[
    color=blue!14.3!red,
    mark=*,
    ]
    coordinates {
(1, 0) (2, 1/8) (3, 3/16) (4, 3/16) (5, 5/32) (6, 15/128) (7, 21/256) (8, 7/128) (9, 9/256) (10, 45/2048) (11, 55/4096) (12, 33/4096) (13, 39/8192) (14, 91/32768) (15, 105/65536) (16, 15/16384) (17, 17/32768) (18, 153/524288) (19, 171/1048576) (20, 95/1048576)
    };
    \addlegendentry{$r=1$}
    
\addplot[
    color=blue!28.6!red,
    mark=*,
    ]
    coordinates {
(1, 0) (2, 1/32) (3, 3/32) (4, 9/64) (5, 5/32) (6, 75/512) (7, 63/512) (8, 49/512) (9, 9/128) (10, 405/8192) (11, 275/8192) (12, 363/16384) (13, 117/8192) (14, 1183/131072) (15, 735/131072) (16, 225/65536) (17, 17/8192) (18, 2601/2097152) (19, 1539/2097152) (20, 1805/4194304) };
    \addlegendentry{$r=2$}
    
\addplot[
    color=blue!42.9!red,
    mark=*,
    ]
    coordinates {
(1, 0) (2, 1/64) (3, 3/64) (4, 11/128) (5, 15/128) (6, 135/1024) (7, 133/1024) (8, 119/1024) (9, 99/1024) (10, 1245/16384) (11, 935/16384) (12, 1353/32768) (13, 949/32768) (14, 5187/262144) (15, 3465/262144) (16, 1135/131072) (17, 731/131072) (18, 14841/4194304) (19, 9291/4194304) (20, 11495/8388608)
    };
   \addlegendentry{$r=3$}
    
\addplot[
    color=blue!57.1!red,
    mark=*,
    ]
    coordinates {
(1, 0) (2, 5/512) (3, 15/512) (4, 57/1024) (5, 85/1024) (6, 855/8192) (7, 945/8192) (8, 945/8192) (9, 873/8192) (10, 12105/131072) (11, 9955/131072) (12, 15675/262144) (13, 11895/262144) (14, 69979/2097152) (15, 50085/2097152) (16, 17505/1048576) (17, 11985/1048576) (18, 257805/33554432) (19, 170487/33554432) (20, 222205/67108864)
    };
    \addlegendentry{$r=4$}

\addplot[
    color=blue!71.4!red,
    mark=*,
    ]
    coordinates {
(1, 0) (2, 7/1024) (3, 21/1024) (4, 81/2048) (5, 125/2048) (6, 1329/16384) (7, 1575/16384) (8, 1701/16384) (9, 1701/16384) (10, 25515/262144) (11, 22649/262144) (12, 38379/524288) (13, 31239/524288) (14, 196469/4194304) (15, 149835/4194304) (16, 55629/2097152) (17, 40341/2097152) (18, 916623/67108864) (19, 638685/67108864) (20, 875045/134217728)
    };
    \addlegendentry{$r=5$}
    
\addplot[
    color=blue!85.7!red,
    mark=*,
    ]
    coordinates {
(1, 0) (2, 21/4096) (3, 63/4096) (4, 245/8192) (5, 385/8192) (6, 4215/65536) (7, 5201/65536) (8, 5901/65536) (9, 6237/65536) (10, 99225/1048576) (11, 93555/1048576) (12, 168399/2097152) (13, 145483/2097152) (14, 969787/16777216) (15, 782565/16777216) (16, 306845/8388608) (17, 234549/8388608) (18, 5606685/268435456) (19, 4102119/268435456) (20, 5890665/536870912)
    };
    \addlegendentry{$r=6$}
    
\addplot[
    color=blue!100!red,
    mark=*,
    ]
    coordinates {
(1, 0) (2, 33/8192) (3, 99/8192) (4, 387/16384) (5, 615/16384) (6, 6855/131072) (7, 8673/131072) (8, 10159/131072) (9, 11151/131072) (10, 185085/2097152) (11, 182655/2097152) (12, 344817/4194304) (13, 312741/4194304) (14, 2189187/33554432) (15, 1854525/33554432) (16, 762855/16777216) (17, 611167/16777216) (18, 15295257/536870912) (19, 11702043/536870912) (20, 17550015/1073741824)
    };
    \addlegendentry{$r=7$}
    
    \addplot[color=blue!0!red,mark=x,mark options={scale=2}]   coordinates {(3,0)};
    \addplot[color=blue!14.3!red,mark=x,mark options={scale=2}]   coordinates {(5,0)};
    \addplot[color=blue!28.6!red,mark=x,mark options={scale=2}]   coordinates {(13/2,0)};
    \addplot[color=blue!42.9!red,mark=x,mark options={scale=2}]   coordinates {(31/4,0)};
    \addplot[color=blue!57.1!red,mark=x,mark options={scale=2}]   coordinates {(283/32,0)};
    \addplot[color=blue!71.4!red,mark=x,mark options={scale=2}]   coordinates {(629/64,0)};
    \addplot[color=blue!85.7!red,mark=x,mark options={scale=2}]   coordinates {(2747/256,0)};
    \addplot[color=blue!100!red,mark=x,mark options={scale=2}]   coordinates {(5923/512,0)};
\end{axis}
\end{tikzpicture}
\quad 
\begin{tikzpicture}[scale=.9]
\begin{axis}[
    axis lines = left,
    xlabel = \(r\),
    ylabel = {average depth of $r$th leaf},
    ymin=0,
    colormap={my colormap}{
                color=(red)
                color=(blue)
            },
]
\addplot[
scatter,
point meta=x,
    mark=x,
    mark options={scale=2},
    only marks,
    ]
    coordinates {
    (0,3)(1,5)(2,13/2)(3,31/4)(4,283/32)(5,629/64)(6,2747/256)(7,5923/512)
    };
\end{axis}
\end{tikzpicture}
\caption{The limiting distribution (left) and average (right) of the depth of $r$th leaf in binary trees of size $n\to\infty$, for $0\le r\le 7$.}
\label{fig:dist_binary_smallr}
\end{figure}
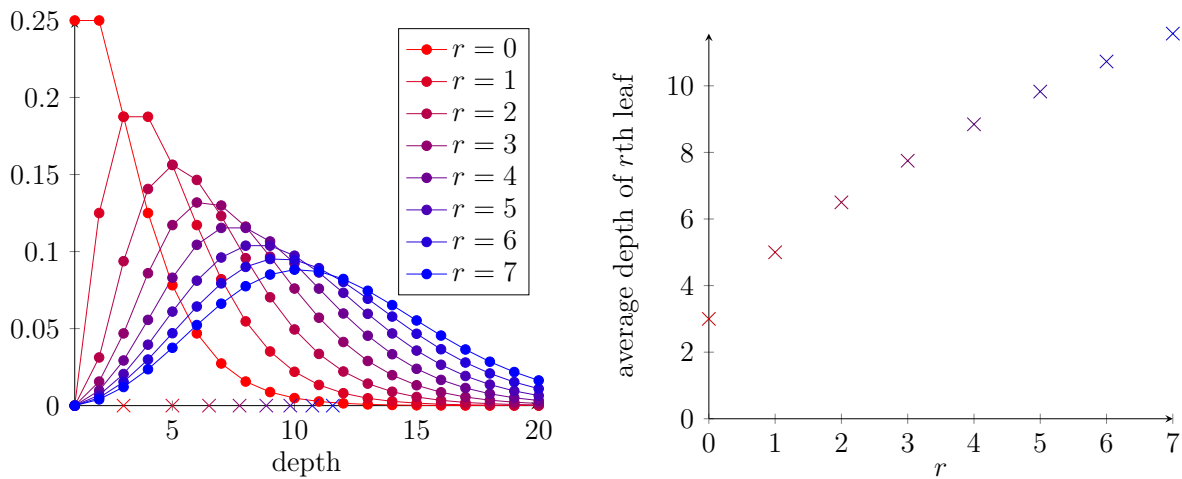

The average depth of each leaf in binary trees of size $n=20$, computed using equation~\eqref{eq:avg_depth}, is shown on the left of Figure~\ref{fig:average_depths_large_r}. 
The limit shape as $n\to\infty$, plotted on the right of the figure, is given by~\eqref{eq:alpha}, where the $x$-axis represents $\alpha=r/n$ and the $y$-axis has been normalized by dividing by the average depth of all leaves.

\begin{figure}[h]
\centering
\begin{tikzpicture}[scale=.9]
\begin{axis}[
	title={$n=20$},
    axis lines = left,
    xlabel = \(r\),
    ylabel = {average depth of $r$th leaf},
    ymin=0, 
]
\addplot[
   mark=*,
    only marks,
    ]
    coordinates {
(0, 30/11) (1, 49/11) (2, 807/143) (3, 34609/5291) (4, 38314/5291) (5, 41221/5291) (6, 103663/12617) (7, 3122507/365893) (8, 3203257/365893) (9, 9752537/1097679) (10, 34150511/3825245) (11, 9752537/1097679) (12, 3203257/365893) (13, 3122507/365893) (14, 103663/12617) (15, 41221/5291) (16, 38314/5291) (17, 34609/5291) (18, 807/143) (19, 49/11) (20, 30/11)
    };
\end{axis}
\end{tikzpicture}\quad
\begin{tikzpicture}[scale=.9]
\begin{axis}[
	title={$n\to\infty$},
    axis lines = left,
    xlabel = \(r/n\),
    ylabel = {normalized average depth of $r$th leaf},
    ymin=0,
]
\addplot [
    domain=0:1, 
    samples=100, 
]
{8*sqrt(x*(1-x))/3.14159265358};
\end{axis}
\end{tikzpicture}
\caption{The average depth of each leaf in binary trees of size $n=20$ (left), and in the limit as $n\to\infty$, normalized by dividing by $\sqrt{\pi n}$ (right).}
\label{fig:average_depths_large_r}
\end{figure}
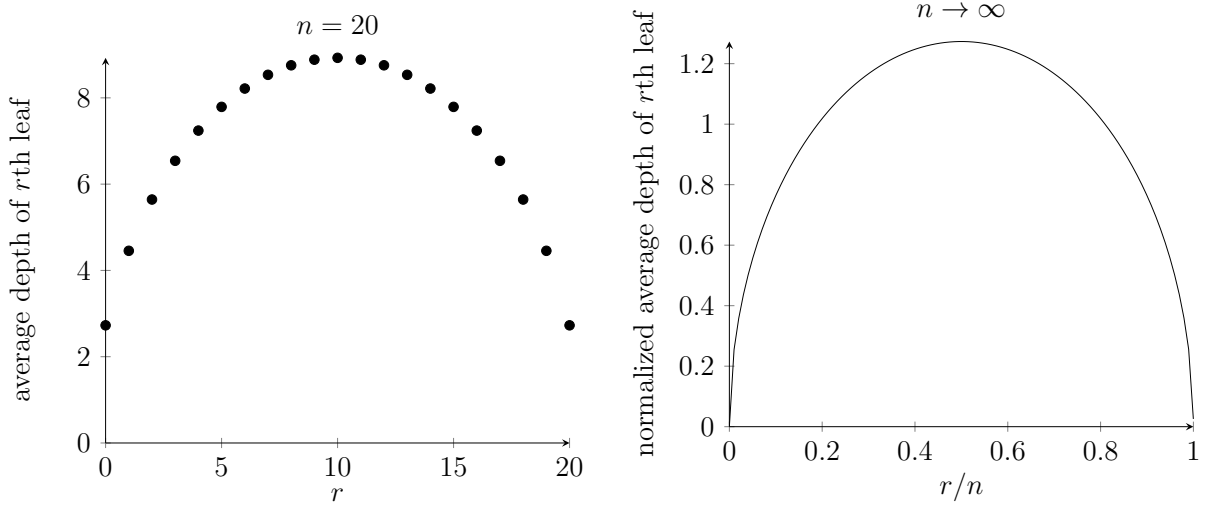

The average depth of a uniformly random leaf can be obtained by integrating equation~\eqref{eq:alpha} for $\alpha$ between $0$ and $1$, or more directly 
by setting $x=1$ in $\dB(x,z)$. Indeed, the sum of the depths of all the leaves over all binary trees of size $n$ is
\begin{equation}\label{eq:sum_depths}
[z^n]\dB(1,z)=[z^n]\frac{1-\sqrt{1-4z}}{1-4z}=4^n-\binom{2n}{n},
\end{equation}
and so the average depth of a random leaf equals
\begin{equation}\label{eq:average_random_leaf}
\frac{[z^n]\dB(1,z)}{(n+1)c_{n}}=\frac{4^{n}}{\binom{2n}{n}}-1\sim\sqrt{\pi n}.
\end{equation}

In biology, the sum of the leaf depths of a tree is known as the {\em Sackin index}, and it is used as a measure of imbalance of phylogenetic trees. The average of this statistic over binary trees of size $n$ is obtained by dividing equation~\eqref{eq:sum_depths} by $c_n$. Different proofs of this expression have been given in~\cite{mir_new_2013,king_simple_2021}. It is worth noting that an arguably more direct proof, which does not require computing $\dB(x,z)$ or manipulating summations as in~\cite{mir_new_2013,king_simple_2021}, can be obtained using standard techniques as follows. Let $K(u,z)$ be the generating function for binary trees where $u$ marks the sum of the leaf depths and $z$ marks the size. The usual decomposition yields the equation
 $$K(u,z) = 1 + u^2 z K(u,uz)^2.$$
 Differentiating with respect to $u$, setting $u=1$, and noting that $K(1,z)=C(z)$, one can deduce that
$$\left.\frac{\partial K(u,z)}{\partial u}\right|_{u=1}=\frac{1-\sqrt{1-4z}}{1-4z},$$
and extract its coefficients as in equation~\eqref{eq:sum_depths}.

The asymptotic behavior from equation~\eqref{eq:average_random_leaf} is the same as if one considers the average depth of a random node rather than a random leaf; see~\cite[Prop.\ VII.3]{flajolet_analytic_2009}. For comparison, the average height of binary trees, defined as the maximum depth of any leaf, behaves like $2\sqrt{\pi n}$; see~\cite[Thm.\ B]{flajolet_average_1982}.

Higher moments of the distribution of the depth of the $r$th leaf in binary trees have been computed in~\cite{panholzer_moments_2002}.

\subsection{The distribution of the depth of the $r$th leaf}\label{sec:distribution_binary}

Gutjahr and Pflug~\cite[Cor.~1]{gutjahr_asymptotic_1992} showed that, when $r=\alpha n$ for some $\alpha\in(0,1)$ and $n\to\infty$, the normalized limiting distribution of the depth of the $r$th leaf is a Maxwell distribution. This result was generalized by Drmota~\cite{drmota_distribution_1994} to any simply generated family of trees, by using double Hankel contour integration.

Here we show that, when $r$ is fixed, the distribution of the depth of the $r$th leaf converges to a discrete law. We obtain this limiting distribution by using singularity analysis on the multivariate generating function $B(x,y,z)$.
In the next statement, trees are chosen uniformly at random from the set of all complete binary trees of size~$n$.

\begin{theorem}\label{thm:distribution_binary}
Let $r\ge0$ and $d\ge1$. The limit as $n\to\infty$ of the probability that the $r$th leaf in a random binary tree of size $n$ has depth $d$ exists. Denoting this limit by $p_{r,d}$, we have
$$\sum_{r\ge0}\sum_{d\ge1} p_{r,d}\, x^r y^d = \frac{y}{(2-2y+y\sqrt{1-x})^2}.$$
\end{theorem}

\begin{proof}
Making the change of variable $x=t/z$ in the expression for $B(x,y,z)$ given in Theorem~\ref{thm:B}, and taking its singular expansion, as a function of $z$, at $z=1/4$, we have
\begin{align*}B(t/z,y,z)&=\frac{2}{2-2y+y\sqrt{1-4z}+y\sqrt{1-4t}}\\
&=\frac{2}{2-2y+y\sqrt{1-4t}}-\frac{2y}{\left(2-2y+y\sqrt{1-4t}\right)^2}\sqrt{1-4z}+o(\sqrt{1-4z}).\end{align*}
Using tools from singularity analysis in one variable~\cite{flajolet_analytic_2009}, it follows that, when $r$ is fixed and $n\to\infty$,
$$[x^rz^n]B(x,y,z)=[t^rz^{n-r}]B(t/z,y,z)\sim [t^r]\frac{y}{\left(2-2y+y\sqrt{1-4t}\right)^2}\frac{4^{n-r}}{\sqrt{\pi n^3}}.$$
Dividing by $[x^rz^n]B(x,1,z)=c_n\sim 4^n/\sqrt{\pi n^3}$ (by equation~\eqref{eq:Ckapprox}), we obtain the probability generating function of the limiting distribution of the depth of the $r$th leaf:
$$\sum_{d\ge1}p_{r,d}\,y^d=\frac{[x^rz^n]B(x,y,z)}{c_n}\sim \frac{1}{4^r}[t^r]\frac{y}{\left(2-2y+y\sqrt{1-4t}\right)^2}=[x^r]\frac{y}{\left(2-2y+y\sqrt{1-x}\right)^2}.\qedhere $$
\end{proof}

In the special case of the leftmost leaf ($r=0$) in binary trees, considered in equation~\eqref{eq:x0},
the limiting distribution of its depth as $n\to\infty$ is given by the probability generating function 
$$[x^0]\frac{y}{\left(2-2y+y\sqrt{1-x}\right)^2}=\frac{y}{(2-y)^2},$$
which describes a (shifted) negative binomial $\NB(2,1/2)$.

For the next leaf ($r=1$), the limiting distribution of its depth is given by the probability generating function 
$$[x^1]\frac{y}{\left(2-2y+y\sqrt{1-x}\right)^2}=\frac{y^2}{(2-y)^3},$$
which describes a (shifted) negative binomial $\NB(3,1/2)$. And for $r=2$, it is given by
$$[x^2]\frac{y}{\left(2-2y+y\sqrt{1-x}\right)^2}=\frac{y^2(1+y)}{(2- y)^4}.$$ The limiting distribution of the depth of $r$th leaf for small values of~$r$ is plotted on the left of Figure~\ref{fig:dist_binary_smallr}.

Differentiating the generating function in Theorem~\ref{thm:distribution_binary} with respect to $y$ and evaluating at $y=1$, we 
obtain 
$$
\frac{4}{(1-x)^{3/2}}-\frac{1}{1-x}.
$$
Extracting the coefficient of $x^r$ we recover the expression in equation~\eqref{eq:avg_depth_r_fixed} for the average depth of the $r$th leaf in the limit as $n\to\infty$, for fixed~$r$.

\subsection{Leaf abscissas}\label{sec:abscissas}

Trees are often drawn on the plane so that nodes at depth $d$ have the ordinate equal to $d$ (or to $-d$, when using our convention whereby the root is at the top).
There are several possible ways to specify the abscissas of the nodes. One possibility, called the {\em natural embedding} by Bousquet-M\'elou~\cite{bousquet-melou_limit_2006}, is to place the left (resp.\ right) child of each node one unit to the 
left (resp.\ right) of the parent, with the root being at the origin; see Figure~\ref{fig:abscissas}.

\begin{figure}[htb]
\centering
\begin{tikzpicture}[scale=.7]
\foreach \i in {-3,...,1} {
\draw (\i,.5) node[above,scale=.8] {$\i$};
\draw[thin,dotted] (\i,.5)--(\i,-4.5);
}
\draw (0,0)--(0,.5);
\fill (0,0) circle (.1)  coordinate (v0);
\fill (-1,-1) circle (.1) coordinate (v1);
\fill (1,-1) circle (.1) coordinate (v2);
\fill (-2,-2) circle (.1) coordinate (v3);
\fill (0,-2) circle (.1) coordinate (v4);
\fill (-3,-3) circle (.1) coordinate (v5);
\fill (-1,-2.85) circle (.1) coordinate (v6);
\fill (-1,-3.15) circle (.1) coordinate (v7);
\fill (1,-3) circle (.1) coordinate (v8);
\fill (-2,-4) circle (.1) coordinate (v9);
\fill (0,-4) circle (.1) coordinate (v10);
\draw (v2)--(v0)--(v1)--(v3)--(v5);
\draw (v3)--(v6);
\draw (v1)--(v4)--(v7)--(v9);
\draw (v4)--(v8);
\draw (v7)--(v10);
\end{tikzpicture}
\caption{The natural embedding a binary tree. The abscissas of the six leaves, in our usual indexing order, are $-3,-1,-2,0,1,1$.}
\label{fig:abscissas}
\end{figure}
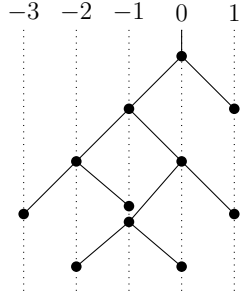

The ordinate of each node, and in particular each leaf, is given by its depth, which is the parameter that we studied in Section~\ref{sec:B}. 
Let us now refine Theorem~\ref{thm:B} by considering the statistic that assigns to each leaf its abscissa in the natural embedding. 
We introduce a new variable $u$ so that the coefficient of $u^ax^ry^dz^n$ in
$B(u,x,y,z)$ is the number of binary trees of size $n$ where the $r$th (as defined in Section~\ref{sec:B}) has depth $d$ and abscissa $a$. Note that leaf abscissas in binary trees of size $n$ can range between $-n$ and $n$. One can easily refine Theorem~\ref{thm:B} as follows.

\begin{theorem}\label{thm:Bv}
$$B(u,x,y,z)=\frac{1}{1-\dfrac{yz}{u} C(z)-uxyzC(xz)}.$$
\end{theorem}

\begin{proof}
As in the proof of Theorem~\ref{thm:B}, thinking of $B(u,x,y,z)$ as the generating function for binary trees with a distinguished leaf, we obtain the equation
\begin{equation}\label{eq:Bv} B(u,x,y,z)=1+\frac{yz}{u} B(u,x,y,z)C(z)+uxyzC(xz)B(u,x,y,z).\end{equation}
Indeed, if the distinguished leaf is in the left (resp.\ right) subtree of the root, then its abscissa in the main tree equals its abscissa in the subtree --- i.e., when setting the root of this subtree to have abscissa $0$ --- minus (resp.\ plus) one.
Solving~\eqref{eq:Bv} for $B(u,x,y,z)$ gives the stated expression.
\end{proof}

To find the average abscissa of each leaf in binary trees of size $n$, we set $y=1$ and compute
$$\left.\frac{\partial B(u,x,1,z)}{\partial u}\right|_{u=1}=\frac{x(1-3z-xz)C(xz)-(1-z-3xz)C(z)}{z(1-x)^2}+\frac{1}{z(1-x)},$$
after some simplifications. Extracting the coefficient of $x^rz^n$ and simplifying again, we get
$$[x^rz^n]\left.\frac{\partial B(u,x,1,z)}{\partial u}\right|_{u=1}=\frac{3c_{n}}{n+2}(2r-n).$$
Dividing by $c_{n}$, we deduce that the average abscissa of the $r$th leaf equals
$$\frac{6r-3n}{n+2}.$$
For each $n$, this is a linear function of $r$ whose range is inside the $(-3,3)$ interval.

Combining this with equation~\eqref{eq:alpha}, it follows that if $r=r(n)=\alpha n$ for some constant $\alpha$, then, as $n\to\infty$, the
expected coordinates of the $r$th leaf of a random binary tree in the natural embedding are asymptotically given by
$$\left(6\alpha-3,\frac{8}{\sqrt{\pi}}\sqrt{\alpha(1-\alpha)n}\right).$$

\section{Dyck paths}
\label{sec:Dyck}

A {\em Dyck path} of semilength $n$ is a lattice path in $\mathbb{Z}^2$ with steps $\uu=(1,1)$ and $\dd=(1,-1)$ that starts at $(0,0)$, ends at $(2n,0)$, and never goes below the $x$-axis. Let $\D_n$ denote the set of Dyck paths of semilength $n$. It is well known that $|\D_n|=c_n$.

\subsection{Enumeration with respect to vertex heights}
\label{sec:Dyck_vertices}

In this section we enumerate Dyck paths while keeping track of the coordinates of their vertices, by using an argument similar to the one that we used to enumerate binary trees with respect to the depths of their leaves. For a path $D\in\D_n$ denote by $V(D)$ its set of vertices, given in Cartesian coordinates. In particular, $|V(D)|=2n+1$.
We are interested in the generating function
$$D(x,y,z)=\sum_{n\ge0}\sum_{D\in\D_n}\sum_{(a,b)\in V(D)} x^ay^b z^n.$$
For example, the Dyck path in Figure~\ref{fig:Dyck} contributes $(1+xy+x^2+x^3y+x^4y^2+x^5y+x^6y^2+x^7y+x^8)z^4$ to the generating function.

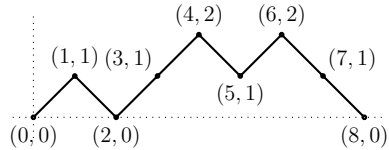
\begin{figure}[h]
\centering
\begin{tikzpicture}[scale=.55]
\draw[dotted] (-.5,0)--(8.5,0);
\draw[dotted] (0,-.5)--(0,2.5);
\draw[thick](0,0) circle(1.2pt) \up\dn\up\up\dn\up\dn\dn;
\foreach \i in {(0,0),(2,0),(5,1),(8,0)}{
\draw \i node[below,scale=.7] {$\i$};
}
\foreach \i in {(1,1),(4,2),(6,2)}{
\draw \i node[above,scale=.7] {$\i$};
}
\foreach \i in {(3,1)}{
\draw \i node[above left=-1pt,scale=.7] {$\i$};
}
\foreach \i in {(7,1)}{
\draw \i node[above right=-1pt,scale=.7] {$\i$};
}
\end{tikzpicture}
\caption{A Dyck path and the coordinates of its vertices.}
\label{fig:Dyck}
\end{figure}

\begin{theorem}\label{thm:D}
$$D(x,y,z)=\frac{C(z)}{1-xyzC(z)-x^2zC(x^2z)}=\frac{1-\sqrt{1-4z}}{z\left(1-xy+xy\sqrt{1-4z}+\sqrt{1-4x^2z}\right)}.$$
\end{theorem}

\begin{proof}
One can think of $D(x,y,z)$ as the generating function for Dyck paths with a distinguished vertex, where variables $x$ and $y$ mark the coordinates of such vertex. We claim that $D(x,y,z)$ satisfies the equation
\begin{equation}\label{eq:D} D(x,y,z)=C(z)+xyzD(x,y,z)C(z)+x^2zC(x^2z)D(x,y,z).\end{equation}
Indeed, Dyck paths whose distinguished vertex is $(0,0)$ contribute $C(z)$ to the generating function.
Any other Dyck path $K$ can be decomposed uniquely as $K=\uu K_1\dd K_2$, where $K_1$ and $K_2$ are Dyck paths, and the $\dd$ between them is the first step of $K$ that returns to the $x$-axis. Paths $K$ whose distinguished vertex is a vertex of $K_1$ contribute
$xyzD(x,y,z)C(z)$, since both coordinates of the distinguished vertex increase by one when viewed as a vertex of $K$ rather than as vertex of $K_1$.
On the other hand, paths $K$ whose distinguished vertex is a vertex of $K_2$ contribute
$x^2zC(x^2z)D(x,y,z)$, since the $x$-coordinate of the distinguished vertex increases by the length of $K_1$ plus two when viewed as a vertex of $K$ rather than as a vertex of $K_2$.

Solving~\eqref{eq:D} for $D(x,y,z)$ gives the stated expression.
\end{proof}

For example, the coefficient of $z^3$ in $D(x,y,z)$ is the polynomial
$$5+5yx+(2+3y^2)x^2+(4y+y^3)x^3+(2+3y^2)x^4+5yx^5+5x^6,$$
where the coefficient of $x^r$, for $0\le r\le 6$, describes the distribution of the heights of the $r$th vertex (starting from $0$) in the five paths in Figure~\ref{fig:Dyck_3}.

\begin{figure}[h]
\centering
\begin{tikzpicture}[scale=.44]
\draw[dotted] (-.5,0)--(6.5,0);
\draw[dotted] (0,-.5)--(0,2.5);
\draw[thick](0,0) circle(1.2pt) \up\dn\up\dn\up\dn;
\begin{scope}[shift={(7.5,0)}]
\draw[dotted] (-.5,0)--(6.5,0);
\draw[dotted] (0,-.5)--(0,2.5);
\draw[thick](0,0) circle(1.2pt) \up\dn\up\up\dn\dn;
\end{scope}
\begin{scope}[shift={(15,0)}]
\draw[dotted] (-.5,0)--(6.5,0);
\draw[dotted] (0,-.5)--(0,2.5);
\draw[thick](0,0) circle(1.2pt) \up\up\dn\dn\up\dn;
\end{scope}
\begin{scope}[shift={(22.5,0)}]
\draw[dotted] (-.5,0)--(6.5,0);
\draw[dotted] (0,-.5)--(0,2.5);
\draw[thick](0,0) circle(1.2pt) \up\up\dn\up\dn\dn;
\end{scope}
\begin{scope}[shift={(30,0)}]
\draw[dotted] (-.5,0)--(6.5,0);
\draw[dotted] (0,-.5)--(0,2.5);
\draw[thick](0,0) circle(1.2pt) \up\up\up\dn\dn\dn;
\end{scope}
\end{tikzpicture}
\caption{The five paths in $\D_3$.}
\label{fig:Dyck_3}
\end{figure}
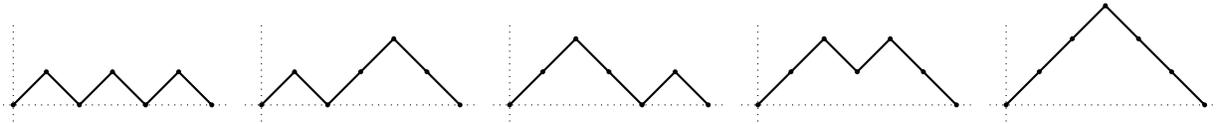

The symmetry resulting from reflecting the paths implies that $D(x,y,z)=D(1/x,t,x^2z)$. An an alternative expression to the one given in Theorem~\ref{thm:D}, from which this property is is more apparent, is
\begin{equation}\label{eq:Dsym}
D(x,y,z)=\frac{C(z)C(x^2z)}{1-xyzC(z)C(x^2z)}=\frac{4}{(1+\sqrt{1-4z})(1+\sqrt{1-4x^2z})-4xyz}.
\end{equation}

\subsection{The height of the $r$th vertex}
\label{sec:Dyck_vertices_average}

An argument similar to the one used in Section~\ref{sec:average_depth_binary} allows us to easily recover the average height of the $r$th vertex in Dyck paths of semilength $n$.
Letting
\begin{equation}\label{eq:dDxz}
\dD(x,z)\coloneqq\left.\frac{\partial D(x,y,z)}{\partial y}\right|_{y=1}=\frac{xzC(z)^2}{\left(1-xzC(z)-x^2zC(x^2z)\right)^2}=xzD(x,1,z)^2,
\end{equation}
the coefficient of $x^rz^n$ in $\dD(x,z)$, for $0\le r\le 2n$, is the sum, over all paths in $\D_n$, of the heights of the $r$th vertex.

A direct combinatorial explanation of equation~\eqref{eq:dDxz} is obtained by interpreting $\dD(x,z)$ as counting Dyck paths $K$ with a distinguished vertex $v_1$ together with an integer $j$ between $1$ and the height of the distinguished vertex.
By considering the maximal subpath $K_1$ containing the distinguished vertex and not going below height $y=j$, we can decompose $K$ into two Dyck paths, each with a distinguished vertex: the subpath $K_1$ with $v_1$ as its distinguished vertex, and the subpath obtained by removing $\uu K_1\dd$ from $K$ and identifying the endpoints of the removed piece into one distinguished vertex $v_2$.
The position of the distinguished vertex in $K$ is the sum of the positions of the distinguished vertices in $K_1$ and $K_2$, plus one because of the additional $\uu$ step. 
Finally, note that $D(x,1,z)$ is the generating function for Dyck paths with a distinguished vertex, whose first coordinate is marked by $x$.

In particular, we can write
$$D(x,1,z)=\frac{C(z)-xC(x^2z)}{1-x}=\sum_{n\ge0}c_{n}[2n+1]_xz^n.$$
Thus, extracting the coefficient of $z^n$ in equation~\eqref{eq:dDxz}, we get
$$[z^n]\dD(x,z)
=x\sum_{i=0}^{n-1}c_ic_{n-i-1}[2i+1]_x[2n-2i-1]_x,$$
from where
\begin{align}\nonumber
[x^rz^n]\dD(x,z)&=\sum_{i=0}^{n-1}c_ic_{n-i-1}\min(r,2n-r,2i+1,2n-2i-1)\\
\label{eq:interpretation}
&=2\sum_{i=0}^{\lfloor r/2 \rfloor-1}(2i+1)c_ic_{n-i-1}+r\sum_{i=\lfloor r/2 \rfloor}^{n-\lfloor r/2 \rfloor-1} c_ic_{n-i-1}\\
\nonumber 
&=rc_n-2\sum_{i=0}^{\lfloor r/2 \rfloor-1}(r-2i-1)c_ic_{n-i-1}\\
&=-c_n+\frac{\delta_r+r(2n-r)}{n(n+1)}\binom{r}{\lfloor r/2 \rfloor}\binom{2n-r}{n-\lfloor r/2\rfloor}, \label{eq:coefdD}
\end{align}
where 
\begin{equation}\label{eq:delta}
\delta_r=\begin{cases} n & \text{if $r$ is even,}\\
2n+1 & \text{if $r$ is odd.} \end{cases}
\end{equation}
In equation~\eqref{eq:interpretation} one can assume, by symmetry, that $0\le r\le n$. 
The closed form in equation~\eqref{eq:coefdD} was obtained using Maple, and it can be verified by induction on~$r$.

Dividing equation~\eqref{eq:coefdD} by $c_{n}$, we obtain the following formula for the average height of the $r$th vertex over paths in $\D_n$, as well as its asymptotic behavior. 
Note that, unlike in the case of the leaf depths in trees, here the height of the $r$th vertex of the path is bounded even as $n\to\infty$, and thus it is obvious that the average height is bounded as well.

\begin{theorem}\label{thm:average_heights}
For $0\le r\le 2n$, the average height of the $r$th vertex over paths in $\D_n$ equals
\begin{equation}\label{eq:average_heights}
\frac{\delta_r+r(2n-r)}{n}\cdot\frac{\binom{r}{\lfloor r/2 \rfloor}\binom{2n-r}{n-\lfloor r/2\rfloor}}{\binom{2n}{n}}-1,
\end{equation}
with $\delta_r$ given by equation~\eqref{eq:delta}. As $n\to\infty$, this average is asymptotically equal to
$$\begin{cases}
\displaystyle\frac{2r+1}{2^r}\binom{r}{r/2}-1 & \text{if $r$ is fixed and even},\smallskip\\
\displaystyle\frac{2r+2}{2^r}\binom{r}{(r-1)/2}-1 & \text{if $r$ is fixed and odd},\smallskip\\
\displaystyle\frac{2}{\sqrt{\pi}}\sqrt{r\left(2-\frac{r}{n}\right)} & \text{if $r=r(n)$ is such that $r\to\infty$ and $2n-r\to\infty$}.
\end{cases}$$
\end{theorem}

\begin{proof}
By equation~\eqref{eq:average}, the exact formula for the average is obtained by dividing equation~\eqref{eq:coefdD} by $c_n$. For the asymptotic expressions, we use the approximation
\begin{equation}\label{eq:binomapprox}
\binom{2k}{k}\sim \frac{4^k}{\sqrt{\pi k}}
\end{equation}
for $k\to\infty$.
\end{proof}

The first few values of the sequence of asymptotic average heights given by Theorem~\ref{thm:average_heights} for fixed $r$ appear in Table~\ref{tab:small_r}, and are plotted on Figure~\ref{fig:dist_Dyck_smallr}.

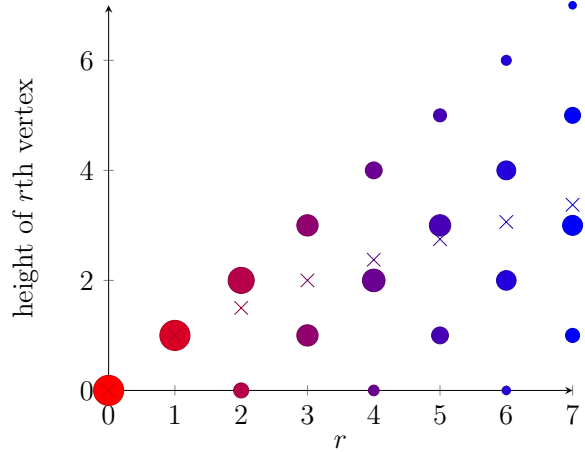
\begin{figure}[h]
\centering
\begin{tikzpicture}[scale=.9]
\begin{axis}[
    axis lines = left,
    xlabel = \(r\),
    ylabel = {height of $r$th vertex},
    ymin=0,
    colormap={my colormap}{
                color=(red)
                color=(blue)
            },
]

\addplot[scatter,scatter src=x,mark=*,only marks,mark options={scale=sqrt(10)}] coordinates {(0,0) (1,1)};
\addplot[scatter,scatter src=x,mark=*,only marks,mark options={scale=sqrt(30/4)}] coordinates {(2,2)};
\addplot[scatter,scatter src=x,mark=*,only marks,mark options={scale=sqrt(10/4)}] coordinates {(2,0)};
\addplot[scatter,scatter src=x,mark=*,only marks,mark options={scale=sqrt(10/2)}] coordinates {(3,1) (3,3) (5,3)};
\addplot[scatter,scatter src=x,mark=*,only marks,mark options={scale=sqrt(10/8)}] coordinates {(4,0)};
\addplot[scatter,scatter src=x,mark=*,only marks,mark options={scale=sqrt(90/16)}] coordinates {(4,2)};
\addplot[scatter,scatter src=x,mark=*,only marks,mark options={scale=sqrt(50/16)}] coordinates {(4,4) (5,1)};
\addplot[scatter,scatter src=x,mark=*,only marks,mark options={scale=sqrt(30/16)}] coordinates {(5,5)};
\addplot[scatter,scatter src=x,mark=*,only marks,mark options={scale=sqrt(50/64)}] coordinates {(6,0)};
\addplot[scatter,scatter src=x,mark=*,only marks,mark options={scale=sqrt(270/64)}] coordinates {(6,2)};
\addplot[scatter,scatter src=x,mark=*,only marks,mark options={scale=sqrt(250/64)}] coordinates {(6,4)};
\addplot[scatter,scatter src=x,mark=*,only marks,mark options={scale=sqrt(70/64)}] coordinates {(6,6)};
\addplot[scatter,scatter src=x,mark=*,only marks,mark options={scale=sqrt(70/32)}] coordinates {(7,1)};
\addplot[scatter,scatter src=x,mark=*,only marks,mark options={scale=sqrt(70/16)}] coordinates {(7,3)};
\addplot[scatter,scatter src=x,mark=*,only marks,mark options={scale=sqrt(90/32)}] coordinates {(7,5)};
\addplot[scatter,scatter src=x,mark=*,only marks,mark options={scale=sqrt(10/16)}] coordinates {(7,7)};

\addplot[
scatter,
scatter src=x,
    mark=x,
    only marks,
    mark options={scale=2},
    ]
    coordinates {
    (0, 0) (1, 1) (2, 3/2) (3, 2) (4, 19/8) (5, 11/4) (6, 49/16) (7, 27/8)
    };
    
 \end{axis}
\end{tikzpicture}
\caption{The limiting distribution of the height of the $r$th vertex in Dyck paths of semilength $n\to\infty$, for $0\le r\le 7$, represented by the relative area of the circles with $x$-coordinate equal to $r$. The crosses indicate the average height of the $r$th vertex.}
\label{fig:dist_Dyck_smallr}
\end{figure}

Letting $r=r(n)=2\alpha n$ for some constant $\alpha\in(0,1)$, Theorem~\ref{thm:average_heights} implies that the average height of the path at $x=2\alpha n$ is asymptotically given by
\begin{equation}\label{eq:height_alpha}
\frac{4}{\sqrt{\pi}}\sqrt{\alpha (1-\alpha)n}.
\end{equation}

We can also use equation~\eqref{eq:dDxz} to compute the average area under Dyck paths.
The coefficients of the generating function 
$$\dD(1,z)
=\frac{1}{1-4z}+\frac{1}{2z}\left(1-\frac{1}{\sqrt{1-4z}}\right)$$ 
give the total sum of vertex heights over all Dyck paths in $\D_n$.
Since the area under a Dyck path is equal to the sum of the heights of its vertices, we recover the well-known fact \cite[Ex.\ VII.26]{flajolet_analytic_2009} that the average area under a path in $\D_n$ is
$$\frac{[z^n]\dD(1,z)}{c_n}=\frac{4^n-\frac{1}{2}\binom{2n+2}{n+1}}{c_n}\sim\sqrt{\pi n^3}.$$

In analogy to Section~\ref{sec:distribution_binary} for leaf depths, we can describe, for fixed $r$, the limiting distribution of the height of the $r$th vertex in Dyck paths of semilength tending to infinity.

\begin{theorem}\label{thm:distribution_Dyck}
Let $r,d\ge0$. The limit as $n\to\infty$ of the probability that the $r$th vertex in a random path in $\D_n$ has height $d$ exists. Denoting this limit by $p_{r,d}$, we have
$$\sum_{r\ge0}\sum_{d\ge0} p_{r,d}\, x^r y^d = \frac{2}{1+y^2-2xy+(1-y^2)\sqrt{1-x^2}}.$$
\end{theorem}

\begin{proof}
Making the substitutions $x=t/q$ and $z=q^2$ in equation~\eqref{eq:Dsym}, we can write
$q D(t/q,y,q^2)=g(q C(q^2))$, where
$$g(w)\coloneqq g(t,y,w)=\frac{C(t^2)w}{1-yt C(t^2)w}.$$
The function $q C(q^2)$ has dominant singularities at $q=\pm1/2$, with singular expansions
\begin{align*} q C(q^2)&=1-\sqrt{2}\sqrt{1-2q}+o(\sqrt{1-2q}), \\ 
q C(q^2)&=-1+\sqrt{2}\sqrt{1+2q}+o(\sqrt{1+2q}).\end{align*}
Thus, arguing as in the proof of \cite[Prop.\ IX.1]{flajolet_analytic_2009}, the singular expansions of $q D(t/q,y,q^2)$ at $q=\pm1/2$ are
\begin{align*}
q D(t/q,y,q^2)&=g(1)-\sqrt{2}g'(1)\sqrt{1-2q}+o(\sqrt{1-2q}), \\
q D(t/q,y,q^2)&=g(-1)+\sqrt{2}g'(-1)\sqrt{1+2q}+o(\sqrt{1+2q}),
\end{align*}
respectively, where the derivatives of $g$ are with respect to $w$.

Using singularity analysis, fixing $r$ and letting $n\to\infty$, and adding the contributions of the two dominant singularities, we get
\begin{align*}[x^rz^n]D(x,y,z)=[t^rq^{2n-r+1}]qD(t/q,y,q^2)&\sim [t^r] \frac{  \sqrt{2}g'(1)2^{2n-r+1}-\sqrt{2}g'(-1)(-2)^{2n-r+1} }{2\sqrt{\pi(2n)^3}}\\
&=\frac{2^{2n-r-1}}{\sqrt{\pi n^3}}\left([t^r]g'(1)+(-1)^r[t^r]g'(-1)\right).
\end{align*}
Dividing by $[x^rz^n]D(x,1,z)=c_n$ and using the approximation~\eqref{eq:Ckapprox}, we find the probability generating function of the limiting distribution of the height of the $r$th vertex in Dyck paths:
$$\sum_{d\ge0}p_{r,d}\,y^d=\frac{[x^rz^n]D(x,y,z)}{c_n}\sim \frac{[t^r]g'(1)+(-1)^r[t^r]g'(-1)}{2^{r+1}}=[x^r]\frac{g'(1)|_{t=x/2}+g'(-1)|_{t=-x/2}}{2}.$$
The expression on the right simplifies to the stated generating function.
\end{proof}

One important difference between the limiting distributions in Theorems~\ref{thm:distribution_binary} and~\ref{thm:distribution_Dyck} is that, for fixed $r$, the depth of the $r$th leaf in a binary tree of size $n$ can be arbitrarily large as $n\to\infty$, whereas the height of the $r$th vertex in a Dyck path of semilength $n$ is always at most $r$. For this reason, the coefficient of $x^r$ in the generating function from
Theorem~\ref{thm:distribution_Dyck} is a polynomial in $y$ of degree $r$, whereas the coefficient of $x^r$ in the generating function in Theorem~\ref{thm:distribution_binary} was a power series of unbounded degree.

The limiting distribution of the height of the $r$th vertex for small values of $r$ is plotted in Figure~\ref{fig:dist_Dyck_smallr}.
Differentiating the generating function from Theorem~\ref{thm:distribution_Dyck} with respect to $y$ and evaluating at $y=1$, we obtain
\[
\frac{\sqrt{1-x^2}}{(1-x)^2}-\frac{1}{1-x}=\frac{(1+x)^2}{(1-x^2)^{3/2}}-\frac{1}{1-x}.
\]
Extracting the coefficient of $x^r$, we recover the expressions in Theorem~\ref{thm:average_heights} giving the average height of the $r$th vertex in the limit as $n\to\infty$, for fixed~$r$.

When $r=2 \alpha n$ for some constant $\alpha\in(0,1)$, the height distribution of the $r$th vertex in Dyck paths of semilength $n\to\infty$ is well understood. Indeed, it is a classical result from probability~\cite{kaigh_invariance_1976} that, after dividing by $\sqrt{2n}$, this distribution converges to a Brownian excursion process.

\subsection{Enumeration with respect to up-step heights}
\label{sec:Dyck_upsteps}

A small modification of the argument in Section~\ref{sec:Dyck_vertices} allows us to enumerate Dyck paths with respect to the heights of their up-steps. We define the height of a step to be the $y$-coordinate of its highest point.
We index the up-steps of a path in $\D_n$ from left to right from $1$ to $n$.
Let $U(x,y,z)$ be the generating function where the coefficient of $x^ry^dz^n$ is the number of Dyck paths of semilength $n$ whose $r$th up-step has height $d$.
In particular, the first step always has height $1$.

\begin{theorem}\label{thm:U} 
$$U(x,y,z)=\frac{xyzC(z)^2}{1-xz(yC(z)+C(xz))}=\frac{xyzC(z)^2C(xz)}{1-xyzC(z)C(xz)}.$$
\end{theorem}

\begin{proof}
Viewing $U(x,y,z)$ as the generating function for Dyck paths with a distinguished up-step, where $x$ marks the index and $y$ marks the height of this step, the usual first-return decomposition of a Dyck path as $\uu K_1\dd K_2$ yields the equation
$$U(x,y,z)=xyzC(z)^2+xyzU(x,y,z)C(z)+xzC(xz)U(x,y,z).$$
Indeed, we consider three cases depending on the location of the distinguished up-step. 
If it is the first $\uu$, then each of the subpaths $K_1$ and $K_2$ contributes $C(z)$, giving a term $xyzC(z)^2$.
If the distinguished step is in $K_1$, then $\uu K_1\dd $ contributes $xyzU(x,y,z)$, and $K_2$ contributes $C(z)$.
Finally, if the distinguished step is in $K_2$,  then $K_2$ contributes $U(x,y,z)$ and $\uu K_1\dd $ contributes $xzC(xz)$, because it causes a shift in the index of the distinguished step.

Solving for $U(x,y,z)$ we get the stated expressions.
\end{proof}

\subsection{The height of the $r$th up-step}
\label{sec:Dyck_upsteps_average}

To obtain the average height of the $r$th up-step over paths in $\D_n$, we first compute
\[
\dU(x,z)\coloneqq\left.\frac{\partial U(x,y,z)}{\partial y}\right|_{y=1}=U(x,1,z)\left(1+\frac{U(x,1,z)}{C(z)}\right).
\]
It is possible to give a combinatorial proof of this equation by interpreting $U(x,1,z)$ as the generating function for Dyck paths with a distinguished up-step, and $U(x,1,z)/C(z)$ as the generating function for Dyck paths with a distinguished up-step that belongs to a peak, where $x$ marks the index of the distinguished up-step in both cases. 

However, it will be more convenient to instead use the following expression in terms of the generating function $\dB(x,z)$ given by equation~\eqref{eq:dBxz}, obtained after some algebraic manipulations:
\[
\dU(x,z)=\frac{\dB(x,z)}{2}+\frac{(1-z-3xz)\sqrt{1-4z}-(1-3z-xz)\sqrt{1-4xz}}{(1-x)^2z^2}.
\]
Extracting coefficients as in Section~\ref{sec:average_depth_binary}, using equation~\eqref{eq:coefdB} and simplifying, we obtain
\begin{align}\nonumber
[x^rz^n]\dU(x,z)&=2rc_n-\sum_{i=0}^{r-1}(r-i)c_ic_{n-i}\\
\label{eq:coefdU}
&=2rc_n-\frac{r+1}{2}c_{n+1}+\frac{(2r+1)(2(n-r)+1)}{(n+1)(n+2)}\binom{2r}{r}\binom{2(n-r)}{n-r}
\end{align}
for $0\le r\le n$. 

Incidentally, comparing equations~\eqref{eq:coefdU} and~\eqref{eq:coefdB}, we observe that the total depth of he $r$th leaf over binary trees and the total height of the $r$th up-step over Dyck paths are related by
$$2[x^rz^n]\dU(x,z)-[x^rz^n]\dB(x,z)=(4r+1)c_n-(r+1)c_{n+1}.$$

\begin{theorem}\label{thm:average_height_upstep}
For $1\le r\le n$, the average height of the $r$th up-step over paths in $\D_n$ equals
\begin{equation}\label{eq:average_height_upstep}
\frac{(2r+1)(2(n-r)+1)}{n+2}\cdot\frac{\binom{2r}{r}\binom{2(n-r)}{n-r}}{\binom{2n}{n}}+\frac{3(r+1)}{n+2}-2.
\end{equation}
As $n\to\infty$, this average is asymptotically equal to
\begin{numcases}{}
\displaystyle\frac{4r+2}{4^{r}}\binom{2r}{r}-2
 & \text{if $r$ is fixed},\smallskip \label{eq:asymtotic_average_r}\\
\displaystyle\frac{4s+2}{4^{s}}\binom{2s}{s}+1
 & \text{if $s=n-r$ is fixed},\smallskip \label{eq:asymtotic_average_s}\\
\displaystyle\frac{4}{\sqrt{\pi}}\sqrt{r\left(1-\frac{r}{n}\right)} & \text{if $r=r(n)$ is such that $r\to\infty$ and $n-r\to\infty$}. \label{eq:asymtotic_average_notfixed}
\end{numcases}
\end{theorem}

\begin{proof}
To obtain the average, we divide equation~\eqref{eq:coefdU} by $c_n$. For the asymptotic expressions, we use equation~\eqref{eq:binomapprox}.
\end{proof}

By reflecting a path in $\D_n$ vertically, the $r$th up-step becomes the $(s+1)$st down-step (from the left), where $s=n-r$, and so equation~\eqref{eq:asymtotic_average_s} describes the asymptotic average of the height of the $(s+1)$st down-step.
The first few values of the sequence of asymptotic average heights given by equations~\eqref{eq:asymtotic_average_r} and~\eqref{eq:asymtotic_average_s} appear in Table~\ref{tab:small_r}, and are plotted on the right of Figure~\ref{fig:dist_upsteps_smallr}.

Note that the height of the first down-step in a Dyck path equals the length of the initial run of up-steps, which, via the bijection $f_L$ that we will describe in Section~\ref{sec:bijections}, is equidistributed with the depth of the leftmost leaf in binary trees. In particular,  equation~\eqref{eq:avg_depth} when $r=0$ and equation~\eqref{eq:average_height_upstep} when $r=n$ give the same expression to $3n/(n+2)$.

\begin{figure}[h]
\centering
\begin{tikzpicture}[scale=.9]
\begin{axis}[
    axis lines = left,
    xlabel = {height},
    ymin=0,
]
    
\addplot[color=blue!14.3!red,mark=*]
    coordinates { (1, 1) (2, 0) (3, 0) (4, 0) (5, 0) (6, 0) (7, 0) (8, 0)    };
    \addlegendentry{$r=1$}
    
\addplot[color=blue!28.6!red,mark=*]
    coordinates { (1, 1/4) (2, 3/4) (3, 0) (4, 0) (5, 0) (6, 0) (7, 0) (8, 0) };
    \addlegendentry{$r=2$}
    
\addplot[color=blue!42.9!red,mark=*]
    coordinates { (1, 1/8) (2, 3/8) (3, 1/2) (4, 0) (5, 0) (6, 0) (7, 0) (8, 0)    };
   \addlegendentry{$r=3$}
    
\addplot[color=blue!57.1!red,mark=*]
    coordinates { (1, 5/64) (2, 15/64) (3, 3/8) (4, 5/16) (5, 0) (6, 0) (7, 0) (8, 0)    };
    \addlegendentry{$r=4$}

\addplot[color=blue!71.4!red,mark=*]
    coordinates { (1, 7/128) (2, 21/128) (3, 9/32) (4, 5/16) (5, 3/16) (6, 0) (7, 0) (8, 0)    };
    \addlegendentry{$r=5$}

\addplot[color=blue!85.7!red,mark=*]
    coordinates {(1, 21/512) (2, 63/512) (3, 7/32) (4, 35/128) (5, 15/64) (6, 7/64) (7, 0) (8, 0)    };
    \addlegendentry{$r=6$}
    
\addplot[color=blue!100!red,mark=*]
    coordinates { (1, 33/1024) (2, 99/1024) (3, 45/256) (4, 15/64) (5, 15/64) (6, 21/128) (7, 1/16) (8, 0)    };
    \addlegendentry{$r=7$}

    \addplot[color=blue!14.3!red,mark=x,mark options={scale=2}]   coordinates {(1,0)};
    \addplot[color=blue!28.6!red,mark=x,mark options={scale=2}]   coordinates {(7/4,0)};
    \addplot[color=blue!42.9!red,mark=x,mark options={scale=2}]   coordinates {(19/8,0)};
    \addplot[color=blue!57.1!red,mark=x,mark options={scale=2}]   coordinates {(187/64,0)};
    \addplot[color=blue!71.4!red,mark=x,mark options={scale=2}]   coordinates {(437/128,0)};
    \addplot[color=blue!85.7!red,mark=x,mark options={scale=2}]   coordinates {(1979/512,0)};
    \addplot[color=blue!100!red,mark=x,mark options={scale=2}]   coordinates {(4387/1024,0)};

\end{axis}
\end{tikzpicture}
\quad 
\begin{tikzpicture}[scale=.9]
\begin{axis}[
    axis lines = left,
    xlabel = \(r\),
    ylabel = {average height of $r$th up-step},
    ymin=0,
    colormap={my colormap}{
                color=(red)
                color=(blue)
            },
]
\addplot[
scatter,
scatter src=x,
    mark=x,
    mark options={scale=2},
    only marks,
    ]
    coordinates {
     (1, 1) (2, 7/4) (3, 19/8) (4, 187/64) (5, 437/128) (6, 1979/512) (7, 4387/1024)
    };
\end{axis}
\end{tikzpicture}

\begin{tikzpicture}[scale=.9]
\begin{axis}[
    axis lines = left,
    xlabel = {height},
    ymin=0,
    ytick={0,0.05,0.1,0.15,0.2,0.25},
    yticklabels={0,0.05,0.1,0.15,0.2,0.25},    
]
    
\addplot[color=blue!14.3!red,mark=*]
    coordinates { (1, 1/4) (2, 1/4) (3, 3/16) (4, 1/8) (5, 5/64) (6, 3/64) (7, 7/256) (8, 1/64) (9, 9/1024) (10, 5/1024) (11, 11/4096) (12, 3/2048) (13, 13/16384) (14, 7/16384) (15, 15/65536) (16, 1/8192)    };
    \addlegendentry{$r=1$})    
\addplot[color=blue!28.6!red,mark=*]
    coordinates { (1, 1/8) (2, 3/16) (3, 3/16) (4, 5/32) (5, 15/128) (6, 21/256) (7, 7/128) (8, 9/256) (9, 45/2048) (10, 55/4096) (11, 33/4096) (12, 39/8192) (13, 91/32768) (14, 105/65536) (15, 15/16384) (16, 17/32768) };
    \addlegendentry{$r=2$}
    
\addplot[color=blue!42.9!red,mark=*]
    coordinates { (1, 5/64) (2, 9/64) (3, 21/128) (4, 5/32) (5, 135/1024) (6, 105/1024) (7, 77/1024) (8, 27/512) (9, 585/16384) (10, 385/16384) (11, 495/32768) (12, 39/4096) (13, 1547/262144) (14, 945/262144) (15, 285/131072) (16, 85/65536)   };
   \addlegendentry{$r=3$}
    
\addplot[color=blue!57.1!red,mark=*]
    coordinates { (1, 7/128) (2, 7/64) (3, 9/64) (4, 75/512) (5, 275/2048) (6, 231/2048) (7, 91/1024) (8, 273/4096) (9, 1575/32768) (10, 275/8192) (11, 187/8192) (12, 1989/131072) (13, 5187/524288) (14, 3325/524288) (15, 525/131072) (16, 1309/524288)   };
    \addlegendentry{$r=4$}

\addplot[color=blue!71.4!red,mark=*]
    coordinates { (1, 21/512) (2, 45/512) (3, 495/4096) (4, 275/2048) (5, 2145/16384) (6, 1911/16384) (7, 3185/32768) (8, 315/4096) (9, 3825/65536) (10, 2805/65536) (11, 31977/1048576) (12, 11115/524288) (13, 60515/4194304) (14, 40425/4194304) (15, 26565/4194304) (16, 4301/1048576)   };
    \addlegendentry{$r=5$}

\addplot[color=blue!85.7!red,mark=*]
    coordinates { (1, 33/1024) (2, 297/4096) (3, 429/4096) (4, 1001/8192) (5, 4095/32768) (6, 1911/16384) (7, 833/8192) (8, 1377/16384) (9, 8721/131072) (10, 53295/1048576) (11, 39501/1048576) (12, 57057/2097152) (13, 161161/8388608) (14, 111573/8388608) (15, 18975/2097152) (16, 25415/4194304)   };
    \addlegendentry{$r=6$}
    
\addplot[color=blue!100!red,mark=*]
    coordinates { (1, 429/16384) (2, 1001/16384) (3, 3003/32768) (4, 455/4096) (5, 7735/65536) (6, 7497/65536) (7, 6783/65536) (8, 2907/32768) (9, 305235/4194304) (10, 241395/4194304) (11, 370139/8388608) (12, 69069/2097152) (13, 805805/33554432) (14, 575575/33554432) (15, 201825/16777216) (16, 69615/8388608)   };
    \addlegendentry{$r=7$}

    \addplot[color=blue!14.3!red,mark=x,mark options={scale=2}]   coordinates {(3,0)};
    \addplot[color=blue!28.6!red,mark=x,mark options={scale=2}]   coordinates {(4,0)};
    \addplot[color=blue!42.9!red,mark=x,mark options={scale=2}]   coordinates {(19/4,0)};
    \addplot[color=blue!57.1!red,mark=x,mark options={scale=2}]   coordinates {(43/8,0)};
    \addplot[color=blue!71.4!red,mark=x,mark options={scale=2}]   coordinates {(379/64,0)};
    \addplot[color=blue!85.7!red,mark=x,mark options={scale=2}]   coordinates {(821/128,0)};
    \addplot[color=blue!100!red,mark=x,mark options={scale=2}]   coordinates {(3515/512,0)};

\end{axis}
\end{tikzpicture}
\quad 
\begin{tikzpicture}[scale=.9]
\begin{axis}[
    axis lines = left,
    xlabel = \(r\),
    ylabel = {average height of $r$th down-step},
    ymin=0,
    colormap={my colormap}{
                color=(red)
                color=(blue)
            },
]
\addplot[
scatter,
scatter src=x,
    mark=x,
    only marks,
    mark options={scale=2}
    ]
    coordinates {
    (1, 3) (2, 4) (3, 19/4) (4, 43/8) (5, 379/64) (6, 821/128) (7, 3515/512)
    };
\end{axis}
\end{tikzpicture}

\caption{The limiting distribution (left) and average (right) of the height of $r$th up-step (above) and $r$th down-step (below) in Dyck paths of semilength $n\to\infty$, for $1\le r\le 7$.}
\label{fig:dist_upsteps_smallr}
\end{figure}
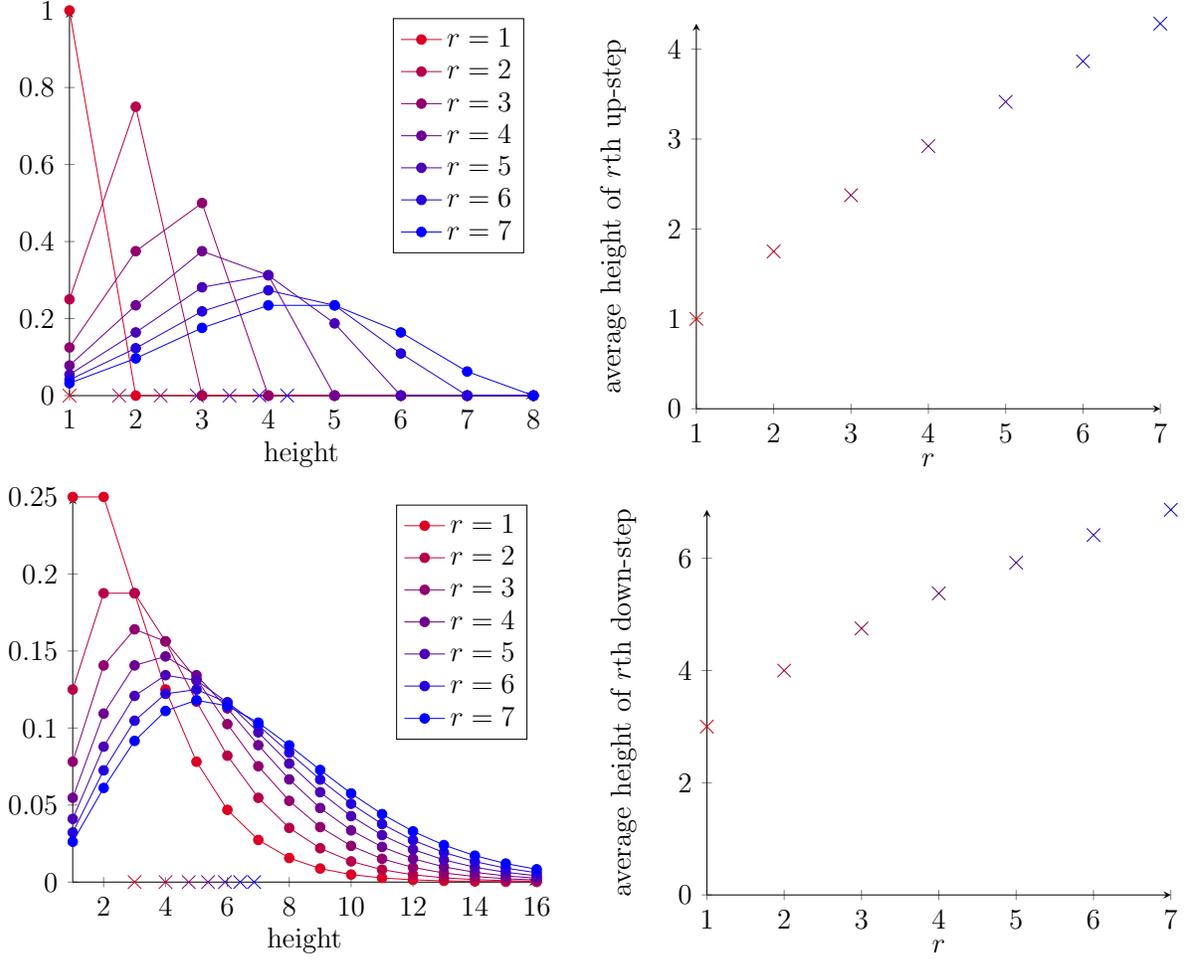

The average height of each vertex and each up-step in Dyck paths of semilength $n=20$, computed using 
equations~\eqref{eq:average_heights} and~\eqref{eq:average_height_upstep}, respectively, is shown in Figure~\ref{fig:average_height_upstep}. In both cases, the limit shape as $n\to\infty$ is the same as in the case of leaf depths of binary trees, plotted on the right of Figure~\ref{fig:average_depths_large_r}, where now the normalizing factor for the $y$-axis is $\sqrt{\pi n}/2$. Indeed, since $$\dU(1,z)=\frac{1}{2}\left(\frac{1}{1-4z}-\frac{1}{\sqrt{1-4z}}\right),$$
the average height of a random up-step in Dyck paths of semilegnth $n$ is
$$\frac{[z^n]\dU(1,z)}{nc_n}=\frac{4^n-\binom{2n}{n}}{2nc_n}\sim\frac{\sqrt{\pi n}}{2}.$$

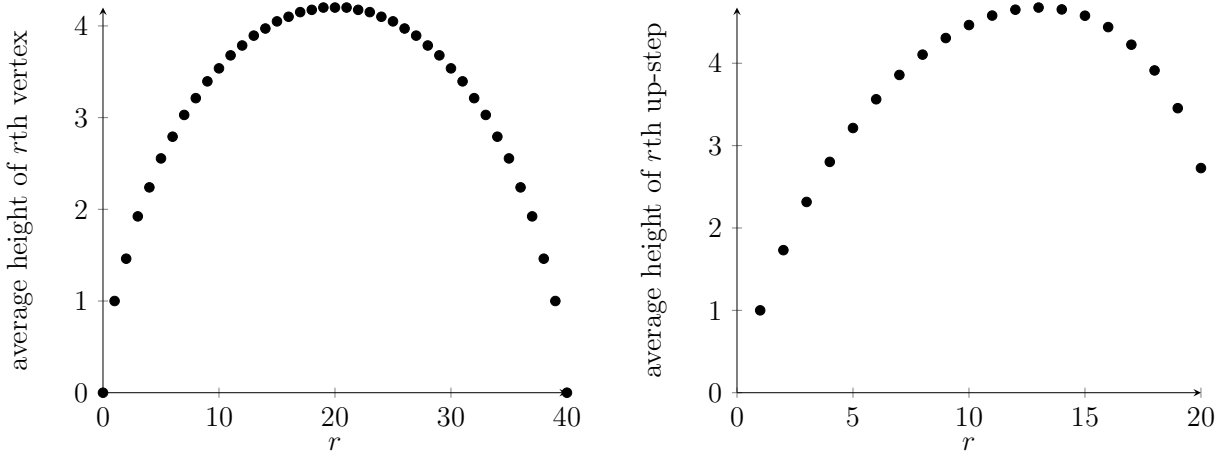
\begin{figure}[h]
\centering
\begin{tikzpicture}[scale=.9]
\begin{axis}[
    axis lines = left,
    xlabel = \(r\),
    ylabel = {average height of $r$th vertex},
    ymin=0, xmax=40, xmin=0,
]
\addplot[
    mark=*,
    only marks,
    ]
    coordinates {
(0, 0) (1, 1) (2, 19/13) (3, 25/13) (4, 1077/481) (5, 1229/481) (6, 1343/481) (7, 1457/481) (8, 16996/5291) (9, 17965/5291) (10, 580171/164021) (11, 54857/14911) (12, 1637365/432419) (13, 129529/33263) (14, 132113/33263) (15, 134697/33263) (16, 681883/166315) (17, 690281/166315) (18, 47914921/11475735) (19, 48200453/11475735) (20, 48200453/11475735) (21, 48200453/11475735) (22, 47914921/11475735) (23, 690281/166315) (24, 681883/166315) (25, 134697/33263) (26, 132113/33263) (27, 129529/33263) (28, 1637365/432419) (29, 54857/14911) (30, 580171/164021) (31, 17965/5291) (32, 16996/5291) (33, 1457/481) (34, 1343/481) (35, 1229/481) (36, 1077/481) (37, 25/13) (38, 19/13) (39, 1) (40, 0)
    };
\end{axis}
\end{tikzpicture}
\quad
\begin{tikzpicture}[scale=.9]
\begin{axis}[
    axis lines = left,
    xlabel = \(r\),
    ylabel = {average height of $r$th up-step},
    ymin=0, xmax=20, xmin=0,
]
\addplot[
    mark=*,
    only marks,
    ]
    coordinates {
(1, 1) (2, 45/26) (3, 1114/481) (4, 1348/481) (5, 17003/5291) (6, 89899/25234) (7, 1411570/365893) (8, 3003679/731786) (9, 4726585/1097679) (10, 34150511/7650490) (11, 5025952/1097679) (12, 3402835/731786) (13, 1710937/365893) (14, 117427/25234) (15, 24218/5291) (16, 23486/5291) (17, 22355/5291) (18, 1119/286) (19, 38/11) (20, 30/11)  
    };
\end{axis}
\end{tikzpicture}
\caption{The average height of each vertex (left) and each up-step (right) in Dyck paths of semilength $n=20$.}
\label{fig:average_height_upstep}
\end{figure}

We remark that, unlike in equation~\eqref{eq:avg_depth}, where the leaves in positions $r$ and $n-r$ had the same average depth, and in equation~\eqref{eq:average_heights}, where the vertices in positions $r$ and $2n-r$ had the same average height, equation~\eqref{eq:average_height_upstep} (and hence the right of Figure~\ref{fig:average_height_upstep}) does not exhibit such symmetry.  
However, this small difference between the averages of the heights of up-steps $r$ and $n-r$ becomes asymptotically negligible in equation~\eqref{eq:asymtotic_average_notfixed}.

The following curious fact is an immediate consequence of Theorem~\ref{thm:average_height_upstep}.

\begin{corollary}\label{cor:upstep-downstep}
For any fixed $r\ge1$, the limit as $n\to\infty$ of the average height of the $(r+1)$st down-step minus the average height of the $r$th up-step over paths in $\D_n$ equals $3$.
\end{corollary}

The height of the $r$th up-step in Dyck paths has been recently studied by Disanto and Munarini~\cite{disanto_local_2019} using a different approach. By means of recurrences and the WZ-method, the authors obtain a formula for the average height \cite[Eq.~(7)]{disanto_local_2019} that is equivalent to our equation~\eqref{eq:average_height_upstep}.

For fixed $r$, as we did with previous statistics, we can more generally describe the limiting distribution of the height of the $r$th up-step and the $r$th down-step in Dyck paths of semilength tending to infinity. The next theorem is proved similarly to Theorems~\ref{thm:distribution_binary} and~\ref{thm:distribution_Dyck}.

\begin{theorem}\label{thm:distribution_upstep}
Let $r,d\ge0$. The limit as $n\to\infty$ of the probability that the $r$th up-step (respectively, down-step)
in a random Dyck path of semilength $n$ leaves has height $d$ exists. Denoting this limit by $p_{r,d}$ (respectively, $p'_{r,d}$), we have
\begin{align}\label{eq:L_up-step}
\sum_{r\ge1}\sum_{d\ge0} p_{r,d}\, x^r y^d 
&=\frac{xy(4-4y+xy^2)}{2(1-y)(1+\sqrt{1 - x})- 3 x y + 4xy^2 - x^2 y^3},\\
\label{eq:L_down-step}
\sum_{r\ge1}\sum_{d\ge0} p_{r,d}'\, x^r y^d
&=\frac{xy}{2(1-y)(1+\sqrt{1 - x})-x+y^2}.
\end{align}
\end{theorem}

\begin{proof}
Making the change of variable $x=t/z$ in the second expression in Theorem~\ref{thm:U}, we can write
$U(t/z,y,z)=g_1(C(z))=g_2(tC(t))$, where
$$
g_1(w)\coloneqq g_1(t,y,w)=\frac{tyw^2C(t)}{1-tywC(t)},\qquad
g_2(w)\coloneqq g_2(y,z,w)=\frac{ywC(z)^2}{1-ywC(z)}.
$$

Since the singular expansion of $C(z)$ at $z=1/4$ is $C(z)=2-2\sqrt{1-4z}+o(\sqrt{1-4z})$, the singular expansion of $U(t/z,y,z)$, as a function of $z$, at the same point is
$$U(t/z,y,z)=g_1(2)-2g_1'(2)\sqrt{1-4z}+o(\sqrt{1-4z}).$$
Using singularity analysis, when $r$ is fixed and $n\to\infty$, we get
$$ [x^rz^n]U(x,y,z)=[t^rz^{n-r}]U(t/z,y,z)\sim [t^r] g'(2) \frac{ 4^{n-r}}{\sqrt{\pi n^3}}.$$
Dividing by $[x^rz^n]U(x,1,z)=c_n$ and using the approximation~\eqref{eq:Ckapprox}, we obtain
$$\sum_{d\ge1}p_{r,d}\,y^d=\frac{[x^rz^n]U(x,y,z)}{c_n}\sim \frac{1}{4^r}[t^r]g'(2)=[x^r]g'(2)|_{t=x/4},$$
which simplifies to equation~\eqref{eq:L_up-step}.

Similarly, the singular expansion of $U(t/z,y,z)$, as a function of $t$, at $t=1/4$ is
$$U(t/z,y,z)=g_2(1/2)-\frac{1}{2}g_2'(1/2)\sqrt{1-4t}+o(\sqrt{1-4t}).$$
When $s$ is fixed and $n\to\infty$, we get
$$ [x^{n-s}z^n]U(x,y,z)=[t^{n-s}z^s]U(t/z,y,z)\sim [z^s] g_2'(1/2) \frac{ 4^{n-s-1}}{\sqrt{\pi n^3}}.$$
Dividing by $c_n$, and noting that reflection of the path turns the $(n-s)$th up-step into the $(s+1)$st down step, we obtain
$$\sum_{d\ge1}p'_{s+1,d}\,y^d=\frac{[x^{n-s}z^n]U(x,y,z)}{c_n}\sim \frac{[z^s]g_2'(1/2)}{4^{s+1}}=[x^s]\frac{g_2'(1/2)|_{z=x/4}}{4}.$$
Multiplying the generating function on the right by $x$ and simplifying, we obtain equation~\eqref{eq:L_down-step}.
\end{proof}

The limiting distributions of the height of the $r$th up-step and the $r$th down-step for small values of $r$ are plotted on the left of Figure~\ref{fig:dist_upsteps_smallr}. Note that, despite their averages being intimately related by Corollary~\ref{cor:upstep-downstep}, the height of the $r$th up-step can take only finitely many values (since it is bounded by $r$), whereas the height of the $r$th down-step can be arbitrarily large.

Differentiating equation~\eqref{eq:L_up-step} with respect to $y$ and evaluating at $y=1$, we obtain
$$\frac{2}{(1-x)^{3/2}}-\frac{2}{1-x}.$$
Extracting the coefficient of $x^r$, we recover the average given by equation~\eqref{eq:asymtotic_average_r}.
Applying the same operations to equation~\eqref{eq:L_down-step}, we recover
equation~\eqref{eq:asymtotic_average_s} with $s=r-1$.

\section{Plane trees}\label{sec:plane}

Plane trees, sometimes called ordered trees, are rooted trees where each node can have an arbitrary number of children, which are linearly ordered. The number of plane trees with $n+1$ nodes, equivalently $n$ edges, is equal to $c_{n}$. Define the size of a plane tree to be its number of edges.

Unlike for binary trees, where the number of leaves was determined by the number of nodes, the number of leaves of a plane tree of size $n\ge1$ can range between $1$ and $n$. The bivariate generating function $\Na(v,z)$ for plane trees where $v$ marks the number of leaves and $z$ marks the size satisfies 
\begin{equation}\label{eq:Na}
\Na(v,z)=\frac{1}{1-z\Na(v,z)}-1+v,
\end{equation}
from where we obtain the well-known expression
$$\Na(v,z)=\frac{1-(1-v)z-\sqrt{1-2(1+v)z+(1-v)^2z^2}}{2z}=v+\sum_{n\ge1}\sum_{k=1}^n \Na_{n,k} v^kz^n,$$
where the coefficients $\Na_{n,k}=\frac{1}{n}\binom{n}{k}\binom{n}{k-1}$ are the Narayana numbers.

As in the case of binary trees, there is a natural ordering of the leaves of a plane tree from left to right, where the relative order of two leaves is determined by comparing the first node where the paths from the root to these leaves disagree.
We index the leaves starting at $0$, so that the leftmost leaf is the $0$th leaf.

More generally, one can order all the nodes of a plane tree as given by a preorder traversal, defined recursively as follows: start at the root, followed by the nodes of each subtree from left to right, where each subtree is traversed in preorder. When indexing all the nodes, the $0$th node is the root. See the left of Figure~\ref{fig:plane-Dyck} for an example of this indexing.

As before, we define the {\em depth} of any node to be the number of edges in the path from the root to that node. 

\subsection{Enumeration with respect to leaf depths}\label{sec:P}

Let $P(v,x,y,z)$ be the generating function where the coefficient of $v^kx^ry^dz^n$
is the number of plane trees with $n$ edges and $k$ leaves whose $r$th leaf has depth $d$. 
We have the following simple expression.

\begin{theorem}\label{thm:P}
\[ P(v,x,y,z)=\frac{v}{1-yz(\Na(vx,z)-vx+1)(\Na(v,z)-v+1)}. \]
\end{theorem}

\begin{proof}
One can think of $P(v,x,y,z)$ as the generating function for plane trees with a distinguished leaf, where $x$ marks the index of this leaf, $y$ marks its depth, and $v$ marks the number of leaves. We will show that it satisfies the equation
\begin{equation}\label{eq:P} P(v,x,y,z)=v+\frac{yz P(v,x,y,z)}{(1-z\Na(vx,z))(1-z\Na(v,z))}.\end{equation}

The tree with one node contributes $v$ to the generating function. In all other trees, there is a subtree $\tau$ of the root that contains the distinguished leaf. The root together with $\tau$ contributes $yz P(v,x,y,z)$. To the left of $\tau$ there is a (possibly empty) sequence of plane trees with no distinguished leaf, but whose leaves affect the indexing of the distinguished leaf, therefore contributing $1/(1-z\Na(vx,z))$ to the generating function. To the right of $\tau$ there is another (possibly empty) sequence of plane trees, contributing $1/(1-z\Na(v,z))$ to the generating function. 
 
Solving~\eqref{eq:P} for $P(v,x,y,z)$ and using equation~\eqref{eq:Na} gives the stated expression.
\end{proof}

For example, the coefficient of $z^3$ in $P(v,x,y,z)$ equals
$$y^3v+\left((y+2y^2)+(y+2y^2)x\right)v^2+(y+yx+yx^2)v^3.$$
The term $(y+2y^2)x$ describes the distribution of the depths of the rightmost leaf in the three plane trees with $3$ edges and $2$ leaves.

To find the average depth of each leaf in plane trees with $n$ edges and $k$ leaves, we compute
\begin{equation} \label{eq:dPvxz}
\dP(v,x,z)=\left.\frac{\partial P(v,x,y,z)}{\partial y}\right|_{y=1}=\frac{1}{v}P(v,x,1,z)\left(P(v,x,1,z)-v\right).
\end{equation}
There is a combinatorial proof of this product formula analogous to that of equation~\eqref{eq:dBxz}: interpret $\dP(v,x,z)$ as counting plane trees with a distinguished leaf $a_1$ and a distinguished non-root node $a_2$ in the path from $a_1$ to the root, split at $a_2$ to decompose the tree into a pair of plane trees with a distinguished leaf, and use that $P(v,x,1,z)$ is the generating function for such trees.

This interpretation also allows us to write 
$$ P(v,x,1,z)=\frac{\Na(v,z)-\Na(vx,z)}{1-x}=v+\sum_{n\ge1}\sum_{k=1}^n \Na_{n,k}v^k[k]_x z^n.$$
Extracting coefficients in equation~\eqref{eq:dPvxz}, we obtain, for $0\le r< k\le n$,
$$[x^rv^kz^n]\dP(v,x,z)=\sum_{j=1}^{n-1}\sum_{i=1}^k \Na_{j,i}\Na_{n-j,k+1-i}\min(r+1,k-r,i,k+1-i)+\Na_{n,k}.$$
Dividing this expression by $\Na_{n,k}$ gives the average depth of the $r$th leaf over all plane trees with $n$ edges and $k$ leaves.

The distribution of the depth of the $0$th (leftmost) leaf in plane trees is given by 
\begin{equation}\label{eq:x0plane} P(1,0,y,z)=\frac{1}{1-yzC(z)}, \end{equation}
which coincides with equation~\eqref{eq:x0}. In Section~\ref{sec:bijections} we will explain this bijectively.

For $r>0$, not all plane trees of a given size have an $r$th leaf, and so studying the distribution of the depth of the $r$th leaf is less meaningful, unless one also restricts to trees with a given number of leaves. Instead, we will consider two variations of the problem of enumerating plane trees by leaf depths.

The first one is to study the depths of all the nodes of a plane tree indexed in preorder, rather than just leaves. We will see in Section~\ref{sec:bijections} that this is equivalent to studying the heights of up-steps in Dyck paths, which was done in Section~\ref{sec:Dyck_upsteps}.

The second variation is to restrict to trees where no node has exactly one child, which are known as {\em Schr\"oder trees}. The purpose of this restriction is to make the number of trees with a given number of leaves be finite. We will discuss Schr\"oder trees in Sections~\ref{sec:Schroeder} and~\ref{sec:distribution_Schroeder}.

\subsection{Bijections between plane trees, binary trees, and Dyck paths}\label{sec:bijections}

Equations~\eqref{eq:x0} and~\eqref{eq:x0plane} describe the distribution not only of the depth of the leftmost leaf on binary trees and on plane trees, but also of several other well-known statistics on Catalan objects.
On Dyck paths, two such statistics are the number of returns to the $x$-axis, studied in~\cite{deutsch_dyck_1999}, and the length of the initial run of up-steps, studied in~\cite{deutsch_involution_1999}. On plane trees, another such statistic is the number of children of the root, studied in~\cite{dershowitz_enumerations_1980}. Next we describe some known bijections that explain these correspondences.

Recall the standard bijection between plane trees with $n$ edges and Dyck paths of semilength $n$, resulting from traversing the edges of the tree in preorder: start at the root and, for each child from left to right, take the edge to that child, recursively traverse the corresponding subtree, and take the edge back to the root. To construct the Dyck path, draw an up-step every time that an edge is traversed from the parent to the child, and a down-step every time that an edge is traversed from the child to the parent. See Figure~\ref{fig:plane-Dyck} or \cite[Fig.\ 5-14]{stanley_enumerative_1999} for an example.

\begin{figure}[h]
\centering
\begin{tikzpicture}[scale=.9]
\draw (0,0)--(0,.5);
\fill (0,0) circle (.1) coordinate (v0);
\fill (-1.2,-1) circle (.1) coordinate (v1);
\fill (0,-1) circle (.1) coordinate (v2);
\fill (1.2,-1) circle (.1) coordinate (v4);
\fill (0,-2) circle (.1) coordinate (v3);
\fill (.9,-2) circle (.1) coordinate (v5);
\fill (1.5,-2) circle (.1) coordinate (v6);
\draw (v1)--(v0)--(v4)--(v5);
\draw (v4)--(v6);
\draw (v0)--(v2)--(v3);
\draw[->] (-.3,-.1)--(-.6,-.4);
\draw[->] (-.4,-.45) to [bend left=60] (-.1,-.5);
\draw[<-] (.4,-.45) to [bend right=60] (.1,-.5);
\draw[<-] (.3,-.1)--(.7,-.4);
\foreach \i in {0} {
\draw (v\i) node[above left,scale=.7] {$\i$};
}
\foreach \i in {1,2,3,5} {
\draw (v\i) node[left,scale=.7] {$\i$};
}
\foreach \i in {4,6} {
\draw (v\i) node[right,scale=.7] {$\i$};
}

\begin{scope}[shift={(4,-1.5)},scale=.7]
\draw[->] (-2.25,.5)--(-1.5,.5);
\draw[dotted] (-.5,0)--(12.5,0);
\draw[dotted] (0,-.5)--(0,2.5);
     \draw[thick](0,0) circle(1.2pt) \up\dn\up\up\dn\dn\up\up\dn\up\dn\dn;
     \draw (.5,.5) node[above left=-1pt,scale=.7] {$1$};
     \draw (2.5,.5) node[above left=-1pt,scale=.7] {$2$};
     \draw (3.5,1.5) node[above left=-1pt,scale=.7] {$3$};
     \draw (6.5,.5) node[above left=-1pt,scale=.7] {$4$};
     \draw (7.5,1.5) node[above left=-1pt,scale=.7] {$5$};
     \draw (9.5,1.5) node[above left=-1pt,scale=.7] {$6$};
\end{scope}
\end{tikzpicture}
\caption{The standard bijection between plane trees and Dyck paths. }
\label{fig:plane-Dyck}
\end{figure}
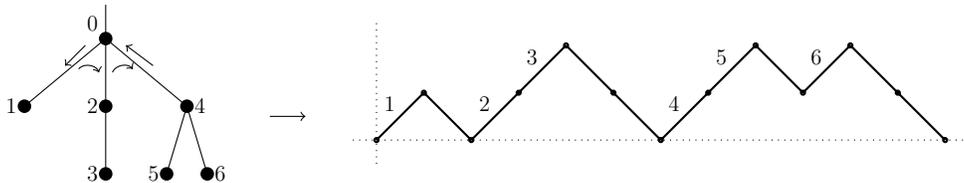

Under this bijection, the number of children of the root of the tree becomes the number of returns of the path, and the depth of the leftmost leaf of the tree becomes the length of the initial run of up-steps. Additionally, the depth of the $r$th node of the tree in preorder (for $1\le r\le n$, counting the root as the $0$th node) becomes the height of the $r$th up-step of the path. It follows that, for $1\le r\le n$, the coefficient of $x^ry^dz^n$ in the generating function $U(x,y,z)$ from Theorem~\ref{thm:U} is also the number of plane trees with $n$ edges whose $r$th node in preorder has depth $d$.

Next we describe two related bijections $f_R$ and $f_L$ from binary trees of size $n$ to Dyck paths of semilength $n$. The depth of the leftmost leaf is mapped to the number of returns to the $x$-axis by $f_R$, and to the the length of the initial run of up-steps by $f_L$.

The bijection $f_R$ (respectively, $f_L$) is defined as follows. Given a binary tree, we traverse its edges in preorder,
and construct a Dyck path by drawing an up-step every time that 
we traverse an edge from a parent to its right (resp., left) child, and a down-step every time that we traverse an edge from a right (resp., left) child to its parent.
See Figure~\ref{fig:binary-Dyck} for examples of both bijections.

\begin{figure}[h]
\centering
\begin{tikzpicture}[scale=.9]
\draw (0,0)--(0,.5);
\fill (0,0) circle (.1) coordinate (v0);
\fill (-1,-1) circle (.1) coordinate (v1);
\fill (1,-1) circle (.1) coordinate (v2);
\fill (-1.7,-2) circle (.1) coordinate (v3);
\fill (-.3,-2) circle (.1) coordinate (v4);
\fill (.3,-2) circle (.1) coordinate (v5);
\fill (1.7,-2) circle (.1) coordinate (v6);
\fill (-2.1,-3) circle (.1) coordinate (v7);
\fill (-1.3,-3) circle (.1) coordinate (v8);
\fill (-.7,-3) circle (.1) coordinate (v9);
\fill (.1,-3) circle (.1) coordinate (v10);
\draw[->] (-.2,0)--(-.6,-.4);
\draw[->] (-.3,-.5) to [bend left=60] (.3,-.5);
\draw[<-] (.2,0)--(.6,-.4);
\draw (v0)--(v1)--(v3)--(v7);
\draw[blue,thick] (v0)--node[below left=-1pt,scale=.5]{$7$} node[right,scale=.5]{$10$}(v2)--node[below left=-1pt,scale=.5]{$8$} node[right,scale=.5]{$9$}(v6);
\draw[blue,thick] (v3)--node[below left=-1pt,scale=.5]{$1$} node[right,scale=.5]{$2$} (v8);
\draw[blue,thick] (v1)--node[below left=-1pt,scale=.5]{$3$} node[right,scale=.5]{$6$} (v4)--node[below left=-1pt,scale=.5]{$4$} node[right,scale=.5]{$5$} (v10);
\draw (v4)--(v9);
\draw (v3)--(v8);
\draw (v2)--(v5);
\foreach \i in {0,...,10} {\fill (v\i) circle (.1);}

\begin{scope}[shift={(4,-2)},scale=.7]
\draw[->] (-2,.5)--node[above]{$f_R$} (-1,.5);
\draw[dotted] (-.5,0)--(10.5,0);
\draw[dotted] (0,-.5)--(0,2.5);
     \draw[blue,thick](0,0) circle(1.2pt) \up\dn\up\up\dn\dn\up\up\dn\dn;
\end{scope}

\begin{scope}[shift={(0,-4.3)}]
\draw (0,0)--(0,.5);
\fill (0,0) circle (.1) coordinate (v0);
\fill (-1,-1) circle (.1) coordinate (v1);
\fill (1,-1) circle (.1) coordinate (v2);
\fill (-1.7,-2) circle (.1) coordinate (v3);
\fill (-.3,-2) circle (.1) coordinate (v4);
\fill (.3,-2) circle (.1) coordinate (v5);
\fill (1.7,-2) circle (.1) coordinate (v6);
\fill (-2.1,-3) circle (.1) coordinate (v7);
\fill (-1.3,-3) circle (.1) coordinate (v8);
\fill (-.7,-3) circle (.1) coordinate (v9);
\fill (.1,-3) circle (.1) coordinate (v10);
\draw[->] (-.2,0)--(-.6,-.4);
\draw[->] (-.3,-.5) to [bend left=60] (.3,-.5);
\draw[<-] (.2,0)--(.6,-.4);
\draw (v0)--(v2)--(v6);
\draw[red,thick] (v0)--node[left,scale=.5]{$1$} node[below right=-1pt,scale=.5]{$8$}(v1)--node[left,scale=.5]{$2$} node[below right=-1pt,scale=.5]{$5$}(v3)--node[left,scale=.5]{$3$} node[below right=-1pt,scale=.5]{$4$}(v7);
\draw[red,thick] (v2)--node[left,scale=.5]{$9$} node[below right=-1pt,scale=.5]{$10$}(v5);
\draw (v3)--(v8);
\draw (v1)--(v4)--(v10);
\draw[red,thick] (v4)--node[left,scale=.5]{$6$} node[below right=-1pt,scale=.5]{$7$}(v9);
\foreach \i in {0,...,10} {\fill (v\i) circle (.1);}

\begin{scope}[shift={(4,-2)},scale=.7]
\draw[->] (-2,.5)--node[above]{$f_L$} (-1,.5);
\draw[dotted] (-.5,0)--(10.5,0);
\draw[dotted] (0,-.5)--(0,2.5);
\draw[red,thick](0,0) circle(1.2pt) \up\up\up\dn\dn\up\dn\dn\up\dn;
\end{scope}
\end{scope}

\end{tikzpicture}
\caption{Two bijections between binary trees and Dyck paths. The depth of the leftmost leaf, which equals 3 in this example, is mapped to the number of returns by $f_R$, and to the length of the initial run of up-steps by $f_L$.}
\label{fig:binary-Dyck}
\end{figure}
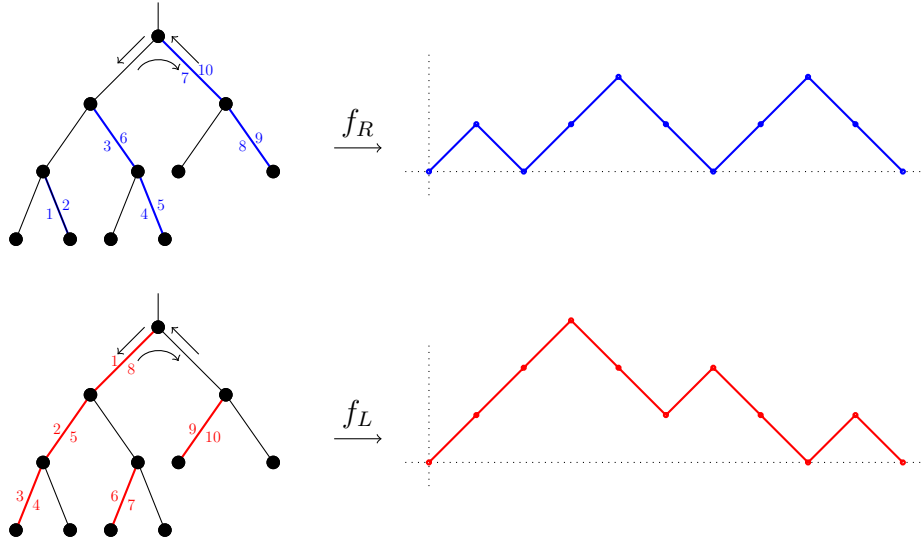

\subsection{Schr\"oder trees}\label{sec:Schroeder}

In this section we consider plane trees where every internal node has at least two children, often referred to as Schr\"oder trees. Let $\tS(z)$ be the generating function for Schr\"oder trees, where $z$ marks the number of leaves. 
The usual decomposition of trees, noting that the root can either be a leaf or have at least two subtrees, gives the equation
\begin{equation}\label{eq:Schroeder}
\tS(z)=z+\frac{\tS(z)^2}{1-\tS(z)}.
\end{equation}
It follows that $\tS(z)=zS(z)$, where
\begin{equation}\label{eq:SchroederGF}
S(z)=\frac{1+z-\sqrt{1-6z+z^2}}{4z}=\sum_{n\ge1}s_{n}z^n.
\end{equation}
Here $s_n$ is the $n$th little Schr\"oder number, and it counts Schr\"oder trees with $n+1$ leaves.

Next we find an expression for the generating function $A(x,y,z)$ where, for $0\le r\le n$, the coefficient of $x^ry^dz^n$ is the number of Schr\"oder trees with $n+1$ leaves whose $r$th leaf has depth $d$.

\begin{theorem}\label{thm:A}
$$A(x,y,z)=\frac{1}{1+y-\dfrac{y}{(1-\tS(z))(1-\tS(xz))}}.$$
\end{theorem}

\begin{proof}
Let $\tilde{A}(x,y,z)=zA(x,y,z)$, which can be thought of as the generating function for Schr\"oder trees with a distinguished leaf,
where $x$ marks the index of this leaf, $y$ marks its depth, and $n$ marks the number of leaves. We claim that
$$\tilde{A}(x,y,z)=z+y\tilde{A}(x,y,z)\left(\frac{1}{(1-\tS(xz))(1-\tS(z))}-1\right).$$
Indeed, if such a tree has more than one node, then the subtree with the distinguished leaf contributes $y\tilde{A}(x,y,z)$. To the left of this subtree we have a sequence of subtrees that contributes $\frac{1}{1-\tS(xz)}$, since it shifts the indexing of the distinguished leaf. To the right we have a sequence of subtrees that contributes $\frac{1}{1-\tS(z)}$. Finally, we have to subtract the case where the root has only one subtree, which is not possible in a Schr\"oder tree.

Solving for $A(x,y,z)$ yields the stated formula.
\end{proof}

\subsection{The depth of the $r$th leaf in Schr\"oder trees}\label{sec:distribution_Schroeder}

The average depth of each leaf in Schr\"oder trees with $n+1$ leaves can be computed using equation~\eqref{eq:average}. After some simplifications, or directly using a combinatorial interpretation analogous to that of equation~\eqref{eq:dBxz}, we obtain
$$\dA(x,z)\coloneqq\left.\frac{\partial A(x,y,z)}{\partial y}\right|_{y=1}
=A(x,1,z)\left(A(x,1,z)-1\right),$$
where
$$A(x,1,z)=\frac{S(z)-xS(xz)}{1-x}$$
is the generating function for Schr\"oder trees with a distinguished leaf.
With the same manipulations as in Section~\ref{sec:average_depth_binary} but with $s_n$ playing the role of $c_n$, and using that $2\sum_{i=0}^n s_is_{n-i} =s_{n+1}+s_n$ by equation~\eqref{eq:Schroeder}, we obtain
\begin{equation}
\label{eq:coefdA}
[x^rz^n]\dA(x,z)=\frac{r+1}{2}(s_{n+1}+s_n)-s_{n}-2\sum_{i=0}^{r-1}(r-i)s_is_{n-i}.
\end{equation}
Dividing by $s_n$, we deduce the following expression for the average depth of each leaf, as well as its asymptotic behavior, where we define $\rho\coloneqq 3-2\sqrt{2}$. As in Theorem~\ref{thm:average_depths} for binary trees, the average depth of the $r$th leaf for fixed $r$ is finite as $n\to\infty$.

\begin{theorem}\label{thm:average_depths_Schroeder}
For $0\le r\le n$, the average depth of the $r$th leaf over all Schr\"oder trees with $n+1$ leaves equals
$$\frac{(r+1)s_{n+1}}{2s_n}+\frac{r-1}{2}-\frac{2}{s_n}\sum_{i=0}^{r-1}(r-i)s_is_{n-i}.$$
As $n\to\infty$, this average is asymptotically equal to
\begin{multline}\label{eq:average_Schreder_asym_fixedr}
(2+\sqrt{2})(r+1)-1-2\displaystyle\sum_{i=0}^{r-1}(r-i)\rho^{i} s_{i}\\
=\displaystyle\frac{\rho^{r-1}}{2}\binom{r+2}{2}s_{r+1}-\rho^r\binom{r+1}{2}s_{r}+\frac{\rho^{r+1}}{2}\binom{r}{2}s_{r-1}-\frac{1}{2}
\end{multline}
if $r$ is fixed, and to 
$$\frac{\sqrt{1-\rho^2}}{\rho\sqrt{\pi}}\sqrt{r\left(1-\frac{r}{n}\right)}$$
if $r=r(n)$ is such that $r\to\infty$ and $n-r\to\infty$.
\end{theorem}

\begin{proof}
The exact formula for the average is obtained by dividing equation~\eqref{eq:coefdA} by $s_n$.
To get the expressions when $n\to\infty$, we use the asymptotic estimate 
\begin{equation}\label{eq:sksim} s_k\sim \frac{\sqrt{1-\rho^2}}{8\sqrt{\pi k^3}}\rho^{-k-1}=\frac{\sqrt{3\sqrt{2}-4}}{4\sqrt{\pi k^3}}\rho^{-k-1} \end{equation}
as $k\to\infty$, which is obtained from equation~\eqref{eq:SchroederGF} using standard tools from asymptotic analysis~\cite{flajolet_analytic_2009}. The closed form in equation~\eqref{eq:average_Schreder_asym_fixedr} will follow from Theorem~\ref{thm:distribution_Schroeder}.
\end{proof}

The first few values of the sequence of asymptotic average depths given by Theorem~\ref{thm:average_depths_Schroeder} appear in Table~\ref{tab:small_r}, and are plotted on the right of Figure~\ref{fig:dist_Schroeder_smallr}.

\begin{figure}[h]
\centering
\begin{tikzpicture}[scale=.9]
\begin{axis}[
    axis lines = left,
    xlabel = {depth},
    ymin=0,
]
\addplot[
    color=blue!0!red,
    mark=*,
    ]
    coordinates {
(1, 0.3432) (2, 0.2843) (3, 0.1766) (4, 0.09752) (5, 0.05049) (6, 0.02510) (7, 0.01214) (8, 0.005745) (9, 0.002678) (10, 0.001233) (11, 0.0005618) (12, 0.0002538) (13, 0.0001138) (14, 0.00005076) (15, 0.00002250) (16, 9.936*10^-6) 
    };
   \addlegendentry{$r=0$}
    
\addplot[
    color=blue!14.3!red,
    mark=*,
    ]
    coordinates {
(1, 0.05890) (2, 0.2154) (3, 0.2372) (4, 0.1884) (5, 0.1271) (6, 0.07796) (7, 0.04476) (8, 0.02456) (9, 0.01302) (10, 0.006724) (11, 0.003392) (12, 0.001682) (13, 0.0008200) (14, 0.0003966) (15, 0.0001892) (16, 0.00008960)  
    };
    \addlegendentry{$r=1$}
    
\addplot[
    color=blue!28.6!red,
    mark=*,
    ]
    coordinates {
(1, 0.02021) (2, 0.1025) (3, 0.1777) (4, 0.1947) (5, 0.1683) (6, 0.1257) (7, 0.08555) (8, 0.05408) (9, 0.03252) (10, 0.01872) (11, 0.01045) (12, 0.005707) (13, 0.003028) (14, 0.001573) (15, 0.0008105) (16, 0.0004106) 
 };
    \addlegendentry{$r=2$}
    
\addplot[
    color=blue!42.9!red,
    mark=*,
    ]
    coordinates {
(1, 0.01040) (2, 0.05765) (3, 0.1183) (4, 0.1592) (5, 0.1662) (6, 0.1468) (7, 0.1153) (8, 0.08304) (9, 0.05600) (10, 0.03582) (11, 0.02197) (12, 0.01302) (13, 0.007489) (14, 0.004205) (15, 0.002311) (16, 0.001247) 
    };
   \addlegendentry{$r=3$}
    
\addplot[
    color=blue!57.1!red,
    mark=*,
    ]
    coordinates {
(1, 0.006550) (2, 0.03743) (3, 0.08244) (4, 0.1235) (5, 0.1460) (6, 0.1461) (7, 0.1291) (8, 0.1039) (9, 0.07760) (10, 0.05453) (11, 0.03656) (12, 0.02347) (13, 0.01452) (14, 0.008770) (15, 0.005159) (16, 0.002953) 
    };
    \addlegendentry{$r=4$}

\addplot[
    color=blue!71.4!red,
    mark=*,
    ]
    coordinates {
(1, 0.004594) (2, 0.02663) (3, 0.06074) (4, 0.09653) (5, 0.1231) (6, 0.1343) (7, 0.1297) (8, 0.1140) (9, 0.09270) (10, 0.07071) (11, 0.05120) (12, 0.03539) (13, 0.02350) (14, 0.01513) (15, 0.009453) (16, 0.005752) 
    };
    \addlegendentry{$r=5$}
    
\addplot[
    color=blue!85.7!red,
    mark=*,
    ]
    coordinates {
(1, 0.003452) (2, 0.02015) (3, 0.04691) (4, 0.07717) (5, 0.1033) (6, 0.1196) (7, 0.1232) (8, 0.1158) (9, 0.1011) (10, 0.08274) (11, 0.06371) (12, 0.04703) (13, 0.03338) (14, 0.02283) (15, 0.01497) (16, 0.009770) 
    };
    \addlegendentry{$r=6$}
    
\addplot[
    color=blue!100!red,
    mark=*,
    ]
    coordinates {
(1, 0.002712) (2, 0.01593) (3, 0.03756) (4, 0.06325) (5, 0.08746) (6, 0.1055) (7, 0.1142) (8, 0.1132) (9, 0.1042) (10, 0.09030) (11, 0.07376) (12, 0.05755) (13, 0.04309) (14, 0.03108) (15, 0.02155) (16, 0.01468) 
    };
    \addlegendentry{$r=7$}
    
    \addplot[color=blue!0!red,mark=x,mark options={scale=2}]   coordinates {(2.414213562,0)};
    \addplot[color=blue!14.3!red,mark=x,mark options={scale=2}]   coordinates {(3.828427118,0)};
    \addplot[color=blue!28.6!red,mark=x,mark options={scale=2}]   coordinates {(4.899494922,0)};
    \addplot[color=blue!42.9!red,mark=x,mark options={scale=2}]   coordinates {(5.793939213,0)};
    \addplot[color=blue!57.1!red,mark=x,mark options={scale=2}]   coordinates {(6.577269557,0)};
    \addplot[color=blue!71.4!red,mark=x,mark options={scale=2}]   coordinates {(7.282610240,0)};
    \addplot[color=blue!85.7!red,mark=x,mark options={scale=2}]   coordinates {(7.929372267,0)};
    \addplot[color=blue!100!red,mark=x,mark options={scale=2}]   coordinates {(8.530065214,0)};
\end{axis}
\end{tikzpicture}
\quad 
\begin{tikzpicture}[scale=.9]
\begin{axis}[
    axis lines = left,
    xlabel = \(r\),
    ylabel = {average depth of $r$th leaf},
    ymin=0,
    colormap={my colormap}{
                color=(red)
                color=(blue)
            },
]
\addplot[
scatter,
point meta=x,
    mark=x,mark options={scale=2},
    only marks,
    ]
    coordinates {
    (0, 2.414213562) (1, 3.828427118) (2, 4.899494922) (3, 5.793939213) (4, 6.577269557) (5, 7.282610240) (6, 7.929372267) (7, 8.530065214)
    };
\end{axis}
\end{tikzpicture}
\caption{The limiting distribution (left) and average (right) of the depth of $r$th leaf in Schr\"oder trees, for $0\le r\le 7$.}
\label{fig:dist_Schroeder_smallr}
\end{figure}
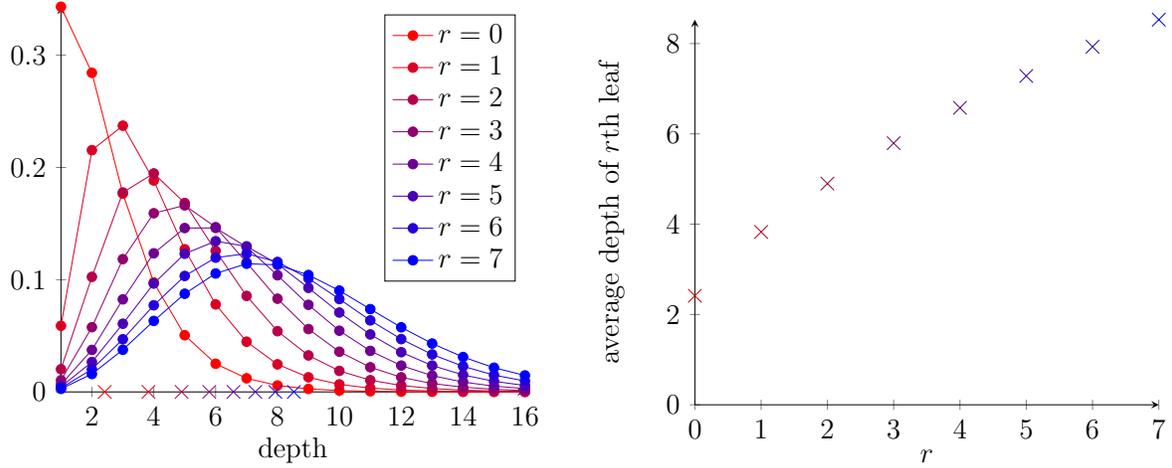

For fixed $r$, we can more generally describe the limiting distribution of the depth of the $r$th leaf as $n\to\infty$. As in the case of binary trees studied in Section~\ref{sec:distribution_binary}, we show that it follows a discrete law.

\begin{theorem}\label{thm:distribution_Schroeder}
Let $r\ge0$ and $d\ge1$. The limit as $n\to\infty$ of the probability that the $r$th leaf in a random Schr\"oder tree with $n+1$ leaves has depth $d$ exists. Denoting this limit by $p_{r,d}$, the generating function $\sum_{r\ge0}\sum_{d\ge1} p_{r,d}\, x^r y^d $ equals 
$$\frac{8y}{9-4y+13y^2+6\sqrt{2}(1-y)^2-(1+2y+5y^2)x+ \rho^{-1}(1-y)(1+3y)\sqrt{(1-x)(1-\rho^2 x)}}.$$
\end{theorem}

\begin{proof}
Making the substitution $x=t/z$ in the generating function from Theorem~\ref{thm:A}, we can write
$A(t/z,y,z)=g(\tS(z))$ where
$$g(w)\coloneqq g(t,y,w)=\frac{1}{1+y-\dfrac{y}{(1-w)(1-\tS(t))}}.$$
The singular expansion of $\tS(z)$ at its dominant singularity $z=\rho$ is
$$\tS(z)=\tau-\frac{\sqrt{1-\rho^2}}{4}\sqrt{1-z/\rho}+o(\sqrt{1-z/\rho}),$$
where $\tau=1-1/\sqrt{2}$.
Thus, the singular expansion of $A(t/z,y,z)$ at $z=\rho$ is
$$A(t/z,y,z)=g(\tau)-\frac{\sqrt{1-\rho^2}}{4}g'(\tau)\sqrt{1-z/\rho}+o(\sqrt{1-z/\rho}).$$

Using singularity analysis, if $r$ is fixed and $n\to\infty$, we have
$$[x^rz^n]A(x,y,z)=[t^rz^{n-r}]A(t/z,y,z)\sim \frac{\sqrt{1-\rho^2}}{4} \frac{\rho^{-(n-r)}}{2\sqrt{\pi n^3}} [t^r]g'(\tau).$$
Dividing by $[x^rz^n]A(x,1,z)=s_n$ and using equation~\eqref{eq:sksim}, the probability generating function of the limiting distribution of the depth of the $r$th leaf in Schr\"oder trees is
$$\sum_{d\ge1} p_{r,d}\,y^d=\frac{[x^rz^n]A(x,y,z)}{s_n}\sim \rho^{r+1}\,[t^r] g'(\tau)= [x^r]\rho g'(\tau)|_{t=\rho x},$$
which simplifies to the stated expression.
\end{proof}

For example, the limiting distribution of the depth of the leftmost leaf in Schr\"oder trees is obtained by setting $x=0$ in Theorem~\ref{thm:distribution_Schroeder}:
$$\sum_{d\ge1} p_{0,d}\,y^d=\frac{2\rho y}{(1-(\sqrt{2}-1)y)^2},$$
which describes a (shifted) negative binomial $\NB(2,\sqrt{2}-1)$. See the left of Figure~\ref{fig:dist_Schroeder_smallr} for the limiting distribution of the depth of the $r$th leaf for small values of~$r$.

Differentiating the expression in Theorem \ref{thm:distribution_Schroeder} with respect to $y$ and evaluating at $y=1$, we obtain
$$\frac{\sqrt{(1-x)(1-\rho^2 x)}}{2\rho(1-x)^{2}}-\frac{1}{2(1-x)}=\frac{(1-\rho^2 x)^{2}}{2\rho\left((1-x)(1-\rho^2 x)\right)^{3/2}}-\frac{1}{2(1-x)}.$$
Extracting the coefficient of $x^r$, noting that $(1-x)(1-\rho^2 x)=1-6\rho x+(\rho x)^2$, we obtain the closed formula in equation~\eqref{eq:average_Schreder_asym_fixedr} for the average depth of the $r$th leaf in the limit.

\section{Triangulations and dissections of polygons}\label{sec:polygons}

In this section we translate the above results regarding leaf depths of binary and Schr\"oder trees into statements about statistics on triangulations and dissections of convex polygons.

\subsection{Triangulations}\label{sec:triang}

A triangulation of a convex $(n+2)$-gon $P_{n+2}$ is a subdivision of $P_{n+2}$ into triangles, obtained by placing $n-1$ noncrossing diagonals.
There is a classical bijection between such triangulations and binary trees of size $n$, where each side and each diagonal of the triangulation becomes a node of the tree (for a total of $2n+1$ nodes). First, fix a side $S$ of $P_{n+2}$ to become the root of the tree.
Our convention is to place this side on the top.
Now, given any triangulation, the triangle that contains $S$ has two more sides (which may be diagonals or sides of $P_{n+2}$); these become the two children of the root. Recursively, if any of these two nodes corresponds to a diagonal of $P_{n+2}$, this diagonal is part of another triangle, whose two other sides become the children of this node, and so on. See Figure~\ref{fig:triangulation} for an example.

This construction produces a binary tree whose $n+1$ leaves correspond to the sides of $P_{n+2}$ other than $S$. Additionally, the ordering of the leaves from left to right corresponds to the counterclockwise ordering of these $n+1$ sides, so that the $0$th and the $n$th leaf are the sides adjacent to $S$.

Next we define a statistic on triangulations. Fix two sides of $P_{n+2}$, say $S$ and $S'$, and suppose there are $r$ sides strictly between $S$ and $S'$ when going from $S$ to $S'$ in counterclockwise direction (and thus $n-r$ when going from $S'$ to $S$). 
For a given triangulation of $P_{n+2}$, consider the number of diagonals that separate $S$ and $S'$, in the sense that $S$ and $S'$ lie on different sides of the diagonal. Equivalently, this is the number of diagonals that are crossed in a straight segment from any point inside $S$ to a point inside $S'$.

Under the above bijection between triangulations and binary trees, side $S'$ corresponds to the $r$th leaf in the tree, and its depth equals one plus the number of diagonals that separate $S$ and $S'$ in the triangulation.
 
Thus, the coefficient of $x^ry^dz^n$ in the generating function
$B(x,y,z)$ from Theorem~\ref{thm:B} is the number of triangulations of $P_{n+2}$ where $S$ and $S'$ are separated by $d-1$ diagonals.
Consequently, the results from Sections~\ref{sec:average_depth_binary} and~\ref{sec:distribution_binary} about the distribution of leaf depths in binary trees can be interpreted in terms of the distribution of the number of diagonals separating two sides in triangulations.

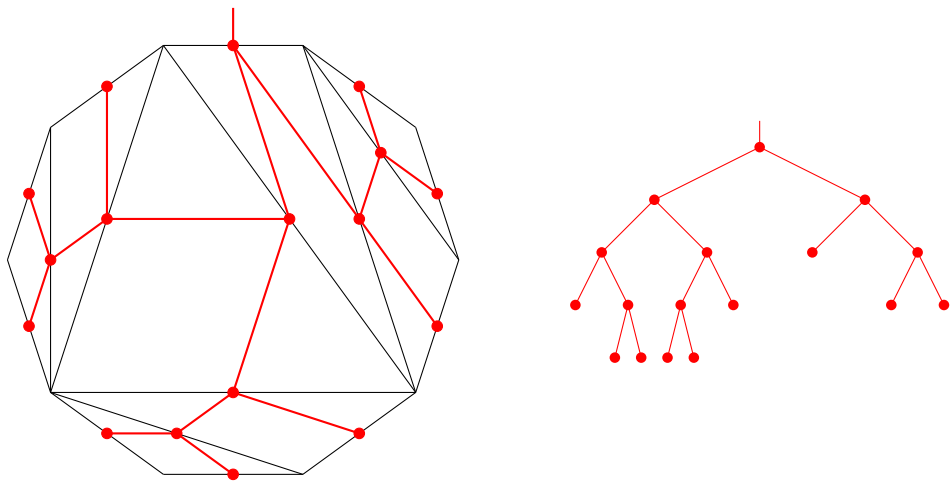
\begin{figure}[h]
\centering
\begin{tikzpicture}[scale=1]
\foreach \i in {1,...,10} {
\draw (72+\i*36:3)coordinate (v\i) -- coordinate (e\i) (108+\i*36:3);
\filldraw[red] (e\i) circle (.07);
}
\draw (v1)--coordinate (v14) (v4);
\draw (v1)--coordinate (v17) (v7);
\draw (v2)--coordinate (v24) (v4);
\draw (v4)--coordinate (v46) (v6);
\draw (v4)--coordinate (v47) (v7);
\draw (v7)--coordinate (v710) (v10);
\draw (v8)--coordinate (v810) (v10);
\foreach \i in {14,17,24,46,47,710,810} {
\filldraw[red] (v\i) circle (.07);
}
\draw[red,thick] (e10)-- ++(0,.5);
\draw[red,thick] (e10)--(v17)--(v14)--(e1);
\draw[red,thick] (v14)--(v24)--(e2);
\draw[red,thick] (v24)--(e3);
\draw[red,thick] (v17)--(v47)--(v46)--(e4);
\draw[red,thick] (v46)--(e5);
\draw[red,thick] (v47)--(e6);
\draw[red,thick] (e10)--(v710)--(e7);
\draw[red,thick] (v710)--(v810)--(e8);
\draw[red,thick] (v810)--(e9);

\begin{scope}[shift={(7,1.5)},scale=.7,red]
\draw (0,0)--(0,.5);
\fill (0,0) circle (.1) coordinate (v0);
\fill (-2,-1) circle (.1) coordinate (v1);
\fill (2,-1) circle (.1) coordinate (v2);
\fill (-3,-2) circle (.1) coordinate (v3);
\fill (-1,-2) circle (.1) coordinate (v4);
\fill (1,-2) circle (.1) coordinate (v5);
\fill (3,-2) circle (.1) coordinate (v6);
\fill (-3.5,-3) circle (.1) coordinate (v7);
\fill (-2.5,-3) circle (.1) coordinate (v8);
\fill (-1.5,-3) circle (.1) coordinate (v9);
\fill (-.5,-3) circle (.1) coordinate (v10);
\fill (2.5,-3) circle (.1) coordinate (v11);
\fill (3.5,-3) circle (.1) coordinate (v12);
\fill (-2.75,-4) circle (.1) coordinate (v13);
\fill (-2.25,-4) circle (.1) coordinate (v14);
\fill (-1.75,-4) circle (.1) coordinate (v15);
\fill (-1.25,-4) circle (.1) coordinate (v16);
\draw (v12)--(v6)--(v2)--(v0)--(v1)--(v3)--(v7);
\draw (v1)--(v4)--(v9)--(v15);
\draw (v4)--(v10);
\draw (v9)--(v16);
\draw (v3)--(v8)--(v13);
\draw (v8)--(v14);
\draw (v2)--(v5);
\draw (v6)--(v11);
\end{scope}
\end{tikzpicture}
\caption{A triangulation of a $10$-gon and the corresponding binary tree.}
\label{fig:triangulation}
\end{figure}

\subsection{Dissections}\label{sec:dissections}

Generalizing the notion of a triangulation, a dissection of a convex polygon is a subdivision into polygons obtained by placing any number of noncrossing diagonals. 
The bijection described in Section~\ref{sec:triang} between triangulations and binary trees easily extends to a bijection bewteen dissections of $P_{n+2}$ and Sch\"oder trees with $n+1$ leaves, as defined in Section~\ref{sec:Schroeder}.
To describe this bijection, again we fix a side $S$ of $P_{n+2}$ to become the root of the plane tree. Each side and each diagonal of the dissection then becomes a node of the tree as follows: the polygon that contains $S$ has at least two more sides (which may be diagonals or sides of $P_{n+2}$); these become the children of the root. Recursively, if any of these nodes corresponds to a diagonal of $P_{n+2}$, this diagonal is part of another polygon, whose other sides become the children of this node, and so on. See Figure~\ref{fig:dissection} for an example.

\begin{figure}[h]
\centering
\begin{tikzpicture}[scale=1]
\foreach \i in {1,...,10} {
\draw (72+\i*36:3)coordinate (v\i) -- coordinate (e\i) (108+\i*36:3);
\filldraw[red] (e\i) circle (.07);
}
\draw (v1)--coordinate (v13) (v3);
\draw (v1)--coordinate (v15) (v5);
\draw (v6)--coordinate (v68) (v8);
\draw (v6)--coordinate (v69) (v9);
\foreach \i in {13,15,68,69} {
\filldraw[red] (v\i) circle (.07);
}
\draw[red,thick] (e10)-- ++(0,.5);
\draw[red,thick] (e10)--(v15)--(v13)--(e1);
\draw[red,thick] (v13)--(e2);
\draw[red,thick] (v15)--(e3);
\draw[red,thick] (v15)--(e4);
\draw[red,thick] (e10)--(e5);
\draw[red,thick] (e10)--(v69)--(v68)--(e6);
\draw[red,thick] (v68)--(e7);
\draw[red,thick] (v69)--(e8);
\draw[red,thick] (e10)--(e9);

\begin{scope}[shift={(8,1.5)},scale=.7,red]
\draw (0,0)--(0,.5);
\fill (0,0) circle (.1) coordinate (v0);
\fill (-3,-1) circle (.1) coordinate (v1);
\fill (-1,-1) circle (.1) coordinate (v2);
\fill (1,-1) circle (.1) coordinate (v3);
\fill (3,-1) circle (.1) coordinate (v4);
\fill (-4,-2) circle (.1) coordinate (v5);
\fill (-3,-2) circle (.1) coordinate (v6);
\fill (-2,-2) circle (.1) coordinate (v7);
\fill (.5,-2) circle (.1) coordinate (v8);
\fill (1.5,-2) circle (.1) coordinate (v9);
\fill (-4.5,-3) circle (.1) coordinate (v10);
\fill (-3.5,-3) circle (.1) coordinate (v11);
\fill (0,-3) circle (.1) coordinate (v12);
\fill (1,-3) circle (.1) coordinate (v13);
\draw (v10)--(v5)--(v1)--(v0)--(v4);
\draw (v5)--(v11);
\draw (v6)--(v1)--(v7);
\draw (v2)--(v0)--(v3)--(v8)--(v12);
\draw (v8)--(v13);
\draw (v3)--(v9);
\end{scope}
\end{tikzpicture}
\caption{A dissection of a $10$-gon and the corresponding Schr\"oder tree.}
\label{fig:dissection}
\end{figure}
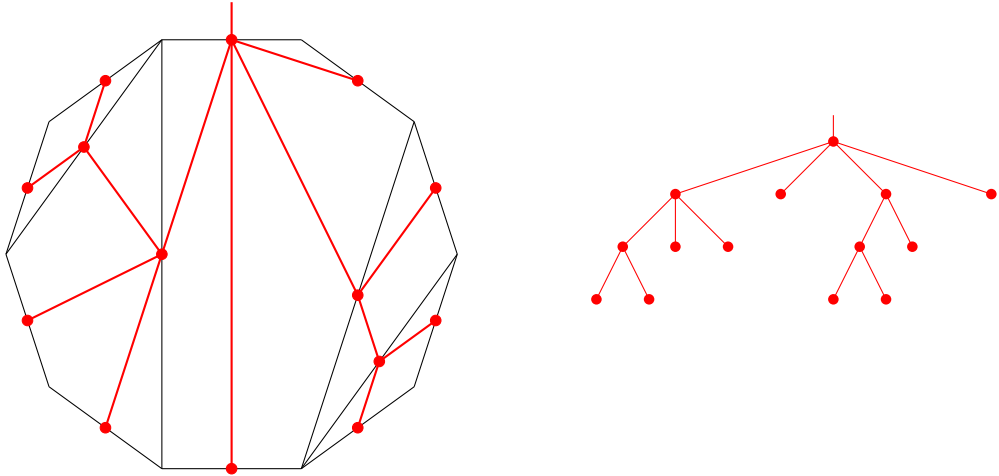

By construction, the resulting plane tree has $n+1$ leaves, and their ordering from left to right corresponds to the counterclockwise ordering of the $n+1$ sides of $P_{n+2}$ other than $S$. Additionally, all the internal nodes of this tree have at least two children.

As we did for triangulations, we fix two sides $S$ and $S'$ of $P_{n+2}$ that have $r$ sides strictly in between, and we define a statistic on dissections that gives the number of diagonals that separate $S$ and $S'$.
Under the above bijection, the depth of the $r$th leaf in the Schr\"oder tree (which corresponds to side $S'$) equals one plus the number of diagonals that separate $S$ and $S'$ in the dissection.

Thus, the coefficient of $x^ry^dz^n$ in the generating function $A(x,y,z)$ from Theorem~\ref{thm:A} is the number of dissections of $P_{n+2}$ where $S$ and $S'$ are separated by $d-1$ diagonals, and the results from Section~\ref{sec:distribution_Schroeder} about leaf depths in Schr\"oder trees translates to results about the number of diagonals separating two sides in dissections.

\section{Noncrossing trees}\label{sec:noncrossing}

Noncrossing trees are plane trees whose nodes lie on the boundary of a circle, and whose edges are straight segments that do not cross.
In a tree with $n$ edges, we label the $n+1$ nodes from $0$ to $n$ in counterclockwise order around the circle. Dulucq and Penaud~\cite{dulucq_cordes_1993} showed that the number of noncrossing trees with $n$ edges is \begin{equation}\label{eq:Cat3}
t_n\coloneqq\frac{1}{2n+1}\binom{3n}{n}.
\end{equation}
In addition, Noy~\cite{noy_enumeration_1998}, together with Flajolet~\cite{flajolet_analytic_2000} and Deutsch~\cite{deutsch_statistics_2002}, enumerated noncrossing trees with respect to various parameters.

\begin{figure}[h]
\centering
\begin{tikzpicture}[scale=.7]
\foreach \i in {0,...,9} {
\draw[dotted] (0,0) circle (3) ;
\draw (90+\i*36:3)coordinate (v\i);
\draw (90+\i*36:3.3) node {$\i$};
\filldraw (v\i) circle (.07);
}
\draw[thick] (v1)--(v0)--(v2)--(v3);
\draw[thick] (v2)--(v4);
\draw[thick] (v0)--(v6)--(v5);
\draw[thick] (v7)--(v6)--(v9)--(v8);

\begin{scope}[shift={(8,1.5)},scale=1.2]
\fill (0,0) circle (.1) coordinate (v0);
\fill (-2,-1) circle (.1) coordinate (v1);
\fill (-.5,-1) circle (.1) coordinate (v2);
\fill (0,-2) circle (.1) coordinate (v3);
\fill (.5,-2) circle (.1) coordinate (v4);
\fill (1.5,-2) circle (.1) coordinate (v5);
\fill (2,-1) circle (.1) coordinate (v6);
\fill (2.5,-2) circle (.1) coordinate (v7);
\fill (3,-3) circle (.1) coordinate (v8);
\fill (3.5,-2) circle (.1) coordinate (v9);
\draw (v1)--(v0)--(v2)--(v3);
\draw (v2)--(v4);
\draw (v0)--(v6)--(v5);
\draw (v6)--(v7);
\draw (v6)--(v9)--(v8);
\foreach \i in {2,6,9} {
\draw[blue,thick,dotted] (v\i)-- ++(0,-.5);
}
\foreach \i in {0,6} {
\draw (v\i) node[above] {$\i$};
}
\foreach \i in {1,2} {
\draw (v\i) node[left] {$\i$};
}
\foreach \i in {3,4,5,7,8} {
\draw (v\i) node[below] {$\i$};
}
\foreach \i in {9} {
\draw (v\i) node[right] {$\i$};
}
\draw (5,1) node[blue] {depth};
\foreach \i in {0,...,3}{
\draw (5,-\i) node[blue] {$\i$};
}
\end{scope}
\end{tikzpicture}
\caption{A noncrossing tree with 9 edges and its recursive decomposition (right). The dotted lines under nodes 2, 6, and 9 separate the two (possibly empty) noncrossing subtrees rooted at these nodes.}
\label{fig:noncrossing}
\end{figure}
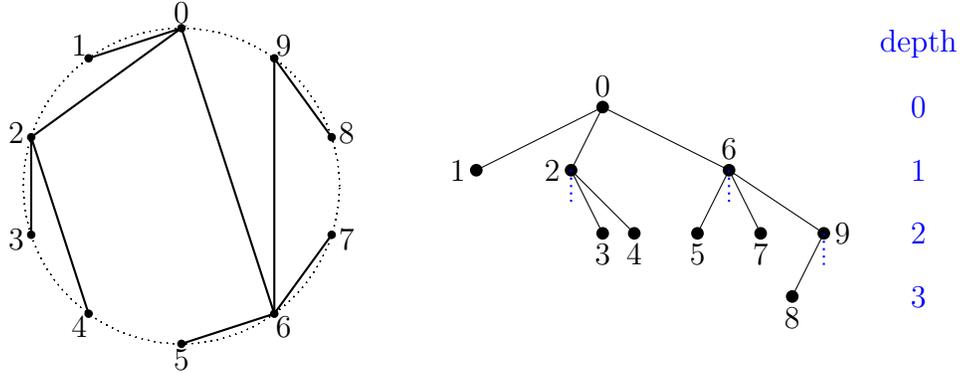

It is convenient to think of noncrossing trees as being rooted at node $0$. Then each child of this node is itself the root of two noncrossing trees, one on each side of the edge coming from node $0$;
see Figure~\ref{fig:noncrossing} for an example.
Letting $T(z)$ denote the generating function for noncrossing trees where $z$ marks the number of edges, this decomposition gives the equation
$$T(z)=\frac{1}{1-zT(z)^2},$$
which can be rewritten as 
\begin{equation}\label{eq:T}
T(z)-1=zT(z)^3,
\end{equation}
from where the expression~\eqref{eq:Cat3} for the coefficients of $T(z)$ follows by Lagrange inversion. Note that $T(z)$  is also the generating function for ternary trees.

\subsection{Enumeration with respect to node depths}

Next we enumerate noncrossing trees with respect to the depth of each node, defined as the number of edges in the path from the root~$0$. Let $G(x,y,z)$ be the generating function where the coefficient of $x^ry^dz^n$ is the number of noncrossing trees with $n$ edges where node $r$ has depth $d$.

\begin{theorem}\label{thm:N}
$$G(x,y,z)=\frac{T(z)+y(1-T(z))T(xz)}{1-yzT(z)T(xz)(T(z)+xT(xz))}.$$
\end{theorem}

\begin{proof}
Thinking of $G(x,y,z)$ as the generating function for noncrossing trees with a distinguished node, where $x$ marks the index of this node, and $y$ marks its depth, we obtain
\begin{equation}\label{eq:N}
G(x,y,z)=T(z)+T(xz)T(z)\left[yz(G(x,y,z)-T(z))T(z)+xyzT(xz)G(x,y,z)\right].
\end{equation}

To see this, consider the possible options for the location of the distinguished node. If the distinguished node is the root $0$, we get a contribution of $T(z)$.
If it is not the root, denote the subtrees of the root by $S_1,S_2,\dots,S_k$ in counterclockwise order.
Each $S_i$ consists of an ordered pair $(L_i,R_i)$ of noncrossing subtrees with a common root $a_i$. The indices of the nodes in each subtree are contiguous, with $a_i$ being the largest in $L_i$ and smallest in $R_i$.
Let $S_j$ be the subtree containing the distinguished node. 
Node $0$ together with the subtrees $S_i$ with $i<j$ form a noncrossing tree, which contributes a factor $T(xz)$ to the generating function, as it shifts the index of the distinguished node by the size of this tree. Similarly, node $0$ together with the subtrees $S_i$ with $i>j$ form a noncrossing tree, which contributes a factor $T(z)$ to the generating function. 

Finally, to compute the contribution of the subtree $S_j$, we consider two cases. If the distinguished node is in $L_j$ (not including its root $a_j$), we get a contribution $yz(G(x,y,z)-T(z))T(z)$, where $yz$ reflects the increase in depth and size caused by the edge from $0$ to $a_j$, the factor $G(x,y,z)-T(z)$ comes from $L_j$ having a distinguished node that is not the root, and the factor $T(z)$ comes from $R_j$.   If the distinguished node is in $R_j$ (including its root $a_j$), we get a contribution $xyzT(xz)G(x,y,z)$, where now $xT(xz)$ comes from the indexing shift caused by $L_j$ and the node $0$, and $G(x,y,z)$ comes from $R_j$. 

Solving~\eqref{eq:N} for $G(x,y,z)$ and using equation~\eqref{eq:T} yields the stated formula.
\end{proof}

\subsection{The depth of node $r$}

To find the average depth of each node in noncrossing trees, we compute 
\begin{equation} \label{eq:dG}
\dG(x,z)\coloneqq\left.\frac{\partial G(x,y,z)}{\partial y}\right|_{y=1}
=xz\left(\frac{T(z)^2-xT(xz)^2}{1-x}\right)^2,
\end{equation}
after some simplifications using equation~\eqref{eq:T}.
Denote by $t'_n\coloneqq[z^n]T(z)^2$ the number of pairs of noncrossing trees with $n$ edges in total. 
We think of such pairs as sharing a common root, so that they have $n+1$ nodes, and refer to them as {\em double noncrossing trees}. 
Then
\begin{equation}\label{eq:T2x} \frac{T(z)^2-xT(xz)^2}{1-x}=\sum_{n\ge0} t'_n [n+1]_x z^n \end{equation}
is the generating function double noncrossing trees with a distinguished node.
By Lagrange inversion using equation~\eqref{eq:T}, we have $$t'_n= \frac{1}{n+1}\binom{3n+1}{n}.$$ 

Equation~\eqref{eq:dG} has a direct combinatorial explanation similar to that of equation~\eqref{eq:dBxz}. First we interpret $\dG(x,z)$ as counting noncrossing trees with a distinguished node $a_1$ (whose index is marked by $x$) and another distinguished node $a_2\neq a_1$ in the path from $a_1$ to the root. By splitting at $a_2$, we obtain a double noncrossing tree rooted at $a_2$ with a distinguished node $a_1$, contributing~\eqref{eq:T2x} to the generating function, plus a non-empty noncrossing tree with a distinguished leaf $a_2$. Thus, to prove equation~\eqref{eq:dG} combinatorially, it suffices to show that the generating function for non-empty noncrossing trees with a distinguished leaf (whose index is marked by $x$) is 
\begin{equation}\label{eq:distinguished_leaf} xz\frac{T(z)^2-xT(xz)^2}{1-x}. \end{equation}

We do this by reversing the roles of the root and the distinguished leaf, which gives a bijection between such trees and noncrossing trees with a distinguished node (the former leaf) whose root (the former distinguished leaf) has degree one. Additionally, indexing the nodes of the new tree in clockwise direction around the circle, with the new root being node~$0$, the index of the distinguished node in the new tree equals the index of the distinguished leaf in the old tree. Finally, after removing the root and the unique edge incident to it, which contributes $xz$, we are left with a double noncrossing tree with a distinguished node, which contributes the expression~\eqref{eq:T2x}. This explains the formula~\eqref{eq:distinguished_leaf}, and thus completes the combinatorial proof of equation~\eqref{eq:dG}.

Extracting coefficients in equation~\eqref{eq:dG}, $$[z^n]\dG(x,z)=x\sum_{i=0}^{n-1} t'_i t'_{n-1-i} [i+1]_x[n-i]_x,$$
and so, for $1\le r\le n/2$,
\begin{align}\label{eq:coefdG}
[x^rz^n]\dG(x,z)&=2\sum_{i=1}^{r-1} i t'_{i-1} t'_{n-i} + r\sum_{i=r}^{n-r+1}t'_{i-1}t'_{n-i}\\
&=r(t'_n-t_n)-2\sum_{i=1}^{r-1}(r-i)t'_{i-1}t'_{n-i},
\label{eq:coefdG_smallr}
\end{align}
where in the last equality we used the identity $\sum_{i=1}^n t'_{i-1} t'_{n-i}=t'_n-t_n$, which follows by equating coefficients  in $z(T(z)^2)^2=T(z)^2-T(z)$.
By equation~\eqref{eq:average}, dividing \eqref{eq:coefdG} or \eqref{eq:coefdG_smallr} by $t_n$ yields the average depth of the node $r$ over all noncrossing trees with $n$ edges. Next we describe the asymptotic behavior of this average.

\begin{theorem}\label{thm:average_depths_noncrossing}
Suppose that $1\le r\le n$. As $n\to\infty$, the average depth of node $r$ over all noncrossing trees with $n$ edges is asymptotically equal to
$$\begin{cases}
2r-6\displaystyle\sum_{i=1}^{r-1}(r-i)t'_{i-1}\frac{4^i}{27^i}
 & \text{if $r$ is fixed},\smallskip\\
\displaystyle\frac{8}{\sqrt{3\pi}}\sqrt{r\left(1-\frac{r}{n}\right)} & \text{if $r=r(n)$ is such that $r\to\infty$ and $n-r\to\infty$}.
\end{cases}$$
\end{theorem}

\begin{proof}
We will use the approximations
$$t_k\sim \frac{\sqrt{3}}{4\sqrt{\pi}k^{3/2}}\left(\frac{27}{4}\right)^k,\quad t'_k\sim \frac{3\sqrt{3}}{4\sqrt{\pi}k^{3/2}}\left(\frac{27}{4}\right)^k$$
as $k\to\infty$, obtained using Stirling's formula.

Consider first the case where $r$ is fixed. Using equation~\eqref{eq:coefdG_smallr}, the average depth of node $r$, for $r$ fixed and $n\to\infty$, is
\[
\frac{[x^rz^n]\dG(x,z)}{t_{n}}=r\left(\frac{t'_n}{t_n}-1\right)-2\sum_{i=1}^{r-1}\frac{(r-i)t'_{i-1}t'_{n-i}}{t_{n}}
\sim 2r-6\sum_{i=1}^{r-1}(r-i)t'_{i-1}\frac{4^i}{27^i}.
\]

If instead $r\to\infty$ and $n-r\to\infty$, then equation~\eqref{eq:coefdG} and a computation analogous to equation~\eqref{eq:integral} give
\[
\frac{[x^rz^n]\dG(x,z)}{t_{n}}\sim \frac{2n^{3/2}}{\sqrt{3\pi}}\left(\int_{0}^{r}\frac{1}{\sqrt{u(n-u)^3}}\,du+r \int_{r}^{n/2}\frac{1}{\sqrt{u^3(n-u)^3}}\,du\right)
=\frac{8}{\sqrt{3\pi}}\sqrt{r\left(1-\frac{r}{n}\right)}.
\qedhere
\]
\end{proof}

The first few values of the sequence of asymptotic average depths given by Theorem~\ref{thm:average_depths_noncrossing} appear in Table~\ref{tab:small_r}, and are plotted on the right of Figure~\ref{fig:dist_noncrossing_smallr}.

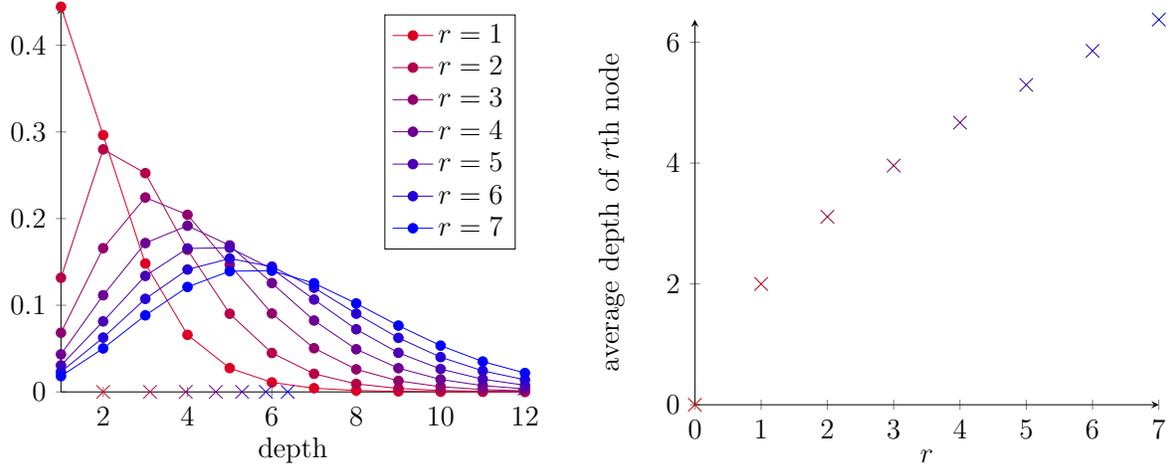
\begin{figure}[h]
\centering
\begin{tikzpicture}[scale=.9]
\begin{axis}[
    axis lines = left,
    xlabel = {depth},
    ymin=0,
]
    
\addplot[
    color=blue!14.3!red,
    mark=*,
    ]
    coordinates {
(1, 4/9) (2, 8/27) (3, 4/27) (4, 16/243) (5, 20/729) (6, 8/729) (7, 28/6561) (8, 32/19683) (9, 4/6561) (10, 40/177147) (11, 44/531441) (12, 16/531441)
    };
    \addlegendentry{$r=1$}
    
\addplot[
    color=blue!28.6!red,
    mark=*,
    ]
    coordinates {
(1, 32/243) (2, 68/243) (3, 184/729) (4, 1076/6561) (5, 592/6561) (6, 884/19683) (7, 3704/177147) (8, 548/59049) (9, 704/177147) (10, 7916/4782969) (11, 3224/4782969) (12, 3868/14348907)
 };
    \addlegendentry{$r=2$}
    
\addplot[
    color=blue!42.9!red,
    mark=*,
    ]
    coordinates {
(1, 448/6561) (2, 1088/6561) (3, 1472/6561) (4, 36208/177147) (5, 25984/177147) (6, 5344/59049) (7, 241024/4782969) (8, 41552/1594323) (9, 60992/4782969) (10, 771712/129140163) (11, 349376/129140163) (12, 153776/129140163)
    };
   \addlegendentry{$r=3$}
    
\addplot[
    color=blue!57.1!red,
    mark=*,
    ]
    coordinates {
(1, 2560/59049) (2, 59264/531441) (3, 91264/531441) (4, 917200/4782969) (5, 2425792/14348907) (6, 1802080/14348907) (7, 10635584/129140163) (8, 6353744/129140163) (9, 3533440/129140163) (10, 16714304/1162261467) (11, 75545600/10460353203) (12, 36545968/10460353203)
    };
    \addlegendentry{$r=4$}

\addplot[
    color=blue!71.4!red,
    mark=*,
    ]
    coordinates {
(1, 146432/4782969) (2, 389120/4782969) (3, 1920512/14348907) (4, 21377536/129140163) (5, 21479744/129140163) (6, 6112256/43046721) (7, 123860800/1162261467) (8, 27983488/387420489) (9, 157346944/3486784401) (10, 2484864512/94143178827) (11, 1376814464/94143178827) (12, 2187764480/282429536481)
    };
    \addlegendentry{$r=5$}
    
\addplot[
    color=blue!85.7!red,
    mark=*,
    ]
    coordinates {
(1, 2981888/129140163) (2, 8092672/129140163) (3, 13864960/129140163) (4, 164108288/1162261467) (5, 178827776/1162261467) (6, 55999936/387420489) (7, 11315320192/94143178827) (8, 2833804736/31381059609) (9, 5869627904/94143178827) (10, 101950811264/2541865828329) (11, 61796785408/2541865828329) (12, 35614329472/2541865828329)
    };
    \addlegendentry{$r=6$}
    
\addplot[
    color=blue!100!red,
    mark=*,
    ]
    coordinates {
(1, 21168128/1162261467) (2, 174915584/3486784401) (3, 308051968/3486784401) (4, 3799146496/31381059609) (5, 39347520512/282429536481) (6, 39504505856/282429536481) (7, 318763256320/2541865828329) (8, 259532832512/2541865828329) (9, 194512148480/2541865828329) (10, 407211191296/7625597484987) (11, 801518414848/22876792454961) (12, 498521490944/22876792454961)
    };
    \addlegendentry{$r=7$}

    \addplot[color=blue!14.3!red,mark=x,mark options={scale=2}]   coordinates {(2,0)};
    \addplot[color=blue!28.6!red,mark=x,mark options={scale=2}]   coordinates {(28/9,0)};
    \addplot[color=blue!42.9!red,mark=x,mark options={scale=2}]   coordinates {(962/243,0)};
    \addplot[color=blue!57.1!red,mark=x,mark options={scale=2}]   coordinates {(30640/6561,0)};
    \addplot[color=blue!71.4!red,mark=x,mark options={scale=2}]   coordinates {(312634/59049,0)};
    \addplot[color=blue!85.7!red,mark=x,mark options={scale=2}]   coordinates {(28017284/4782969,0)};
    \addplot[color=blue!100!red,mark=x,mark options={scale=2}]   coordinates {(823239002/129140163,0)};
\end{axis}
\end{tikzpicture}
\quad 
\begin{tikzpicture}[scale=.9]
\begin{axis}[
    axis lines = left,
    xlabel = \(r\),
    ylabel = {average depth of $r$th node},
    ymin=0,
    colormap={my colormap}{
                color=(red)
                color=(blue)
            },
]
\addplot[
scatter,
point meta=x,
    mark=x,mark options={scale=2},
    only marks,
    ]
    coordinates {
    (0, 0) (1, 2) (2, 28/9) (3, 962/243) (4, 30640/6561) (5, 312634/59049) (6, 28017284/4782969) (7, 823239002/129140163)
    };
\end{axis}
\end{tikzpicture}
\caption{The limiting distribution (left) and average (right) of the depth of node $r$ in noncrossing trees with $n\to\infty$ edges, for $1\le r\le 7$.}
\label{fig:dist_noncrossing_smallr}
\end{figure}

It follows from Theorem~\ref{thm:average_depths_noncrossing} that the average depth of a uniformly random node in noncrossing trees with $n$ edges is asymptotically equal to
$$\int_0^1 \frac{8}{\sqrt{3\pi}}\sqrt{\alpha n (1-\alpha)} \,d\alpha=\sqrt{\frac{\pi n}{3}},$$
which agrees with \cite[Thm.\ 6]{deutsch_statistics_2002}. The exact value of this average depth is $\frac{[z^n]\dG(1,z)}{t_n}$, which can be computed from equation~\eqref{eq:dG}.

For fixed $r$, we can describe the limiting distribution of the depth of node $r$ of a noncrossing tree as its size tends to infinity.

\begin{theorem}\label{thm:distribution_noncrossing}
Let $r\ge1$ and $d\ge0$. The limit as $n\to\infty$ of the probability that node $r$ in a random noncrossing tree with $n$ edges has depth $d$ exists. Denoting this limit by $p_{r,d}$, we have
$$\sum_{r\ge1}\sum_{d\ge0} p_{r,d}\, x^r y^d=\frac{P_2(x,y)T(z)^2+P_1(x,y) T(z)+P_0(x,y)}{\left(\dfrac{y^2(3-y)}{2}-\dfrac{x}{2}\left(4-12 y+ 15y^2-3y^3  \right)+x^2y^3 \right)^2},$$
where 
\begin{align*}
P_2(x,y)&=\frac{1}{9}x^2y(1-y)\left(12y^2 - 5y^3 + y^4 + x(16- 44y+ 36y^2 - 27y^3 + 3y^4) + 4x^2y^3(1+y)\right),\\
P_1(x,y)&=x^2y^2(1-y)^2(2-y)\left(4-y-x(2+y)\right),\\
P_0(x,y)&=xy^2(1-xy)\left(y^3 - x(8- 18y + 12y^2) + x^2y^3\right).
\end{align*}
\end{theorem}

\begin{proof}
Since the root always has depth $0$, it is convenient to subtract the contribution of this node and 
consider the generating function 
\begin{equation}\label{eq:N-T} G(x,y,z)-T(z)=\frac{xyzT(z)^2T(xz)^2}{1-yzT(z)T(xz)(T(z)+xT(xz))}.\end{equation}
Even though $T(z)$ is defined implicitly by equation~\eqref{eq:T}, we can use \cite[Thm.\ VI.6]{flajolet_analytic_2009} to obtain its singular expansion at the dominant singularity $z=\rho'$, where $\rho'\coloneqq4/27$, namely
$$T(z)=\frac{3}{2}-\frac{\sqrt{3}}{2}\sqrt{1-z/\rho'}+o(\sqrt{1-z/\rho'}).$$
One can now continue as in the proof of Theorem~\ref{thm:distribution_Schroeder}: make the substitution $x=t/z$ in equation~\eqref{eq:N-T}, find its singular expansion at $z=\rho'$, and use singularity analysis to estimate $[x^rz^n](G(x,y,z)-T(z))/t_n$ when $r$ is fixed and $n\to\infty$, which gives the stated expression for the limiting distribution of the depth of node~$r$.
\end{proof}

For example, extracting the coefficient of $x$ in Theorem~\ref{thm:distribution_noncrossing},
the limiting distribution of the depth of node $1$ in noncrossing trees is given by the probability generating function
$$\frac{4y}{(3-y)^2},$$ which describes a (shifted) negative binomial $\NB(2,1/3)$. Extracting the coefficient of $x^2$, we find that the limiting distribution of the depth of node $2$ is given by
$$\frac{4y(8+9y+y^2)}{9(3 - y)^3}.$$
The limiting distribution of the depths of the first few nodes is plotted on the left of Figure~\ref{fig:dist_noncrossing_smallr}.

Differentiating the generating function in Theorem~\ref{thm:distribution_noncrossing} with respect to $y$ and evaluating at $y=1$, we obtain
\[
\frac{18x-8x^2T(\rho' x)^2}{9(1-x)^2}.
\]
Extracting the coefficient of $x^r$, we recover the expression in Theorem~\ref{thm:average_depths_noncrossing} for the average depth of node $r$ in the limit, for fixed~$r$.

\section{Increasing binary trees}\label{sec:increasing}

In this section we show that some of the tools from in previous sections also apply to trees whose nodes are labeled, with exponential generating functions playing the role of ordinary generating functions.
We focus on increasing binary trees, although similar methods can be applied to other kinds of labeled trees.

An increasing binary tree of size $n$ is a binary tree with $n$ internal nodes labeled with distinct labels $1,2,\dots,n$ in such a way that labels increase when moving away from the root. It is well known (see e.g.\ \cite[Sec.\ 5.1]{stanley_enumerative_2012}) that increasing binary trees of size $n$ are in bijection with permutations of $\{1,2,\dots,n\}$.
Indeed, the permutation associated to an increasing binary tree is obtained by reading the labels of its internal nodes in an inorder traversal: recursively perform an inorder traversal of the left subtree, followed by the root, followed by an inorder traversal of the right subtree. See Figure~\ref{fig:increasing-permutation} for an example. 

\begin{figure}[htb]
\centering
\begin{tikzpicture}[scale=.65] 
\node[draw,circle,inner sep=1pt,scale=.9] (v0) at (0,0) {1};
\node[draw,circle,inner sep=1pt,scale=.9] (v1) at (-2,-1) {2};
\node[draw,circle,inner sep=1pt,scale=.9] (v2) at (2,-1) {4};
\node[draw,circle,inner sep=1pt,scale=.9] (v3) at (-3,-2) {7};
\node[draw,circle,inner sep=1pt,scale=.9] (v4) at (-1,-2) {3};
\node[draw,circle,inner sep=1pt,scale=.9] (v5) at (1,-2) {5};
\fill (3,-2) circle (.1) coordinate (v6); 
\fill (-3.5,-3) circle (.1) coordinate (v7); 
\node[draw,circle,inner sep=1pt,scale=.9] (v8) at (-2.5,-3) {8};
\fill (-1.5,-3) circle (.1) coordinate (v9); 
\node[draw,circle,inner sep=1pt,scale=.9] (v10) at (-.5,-3) {6};
\fill (.5,-3) circle (.1) coordinate (v11); 
\fill (1.5,-3) circle (.1) coordinate (v12); 
\fill (-2.75,-4) circle (.1) coordinate (v13);
\fill (-2.25,-4) circle (.1) coordinate (v14);
\fill (-.75,-4) circle (.1) coordinate (v15); 
\fill (-.25,-4) circle (.1) coordinate (v16); 

\draw (v0)--(0,.5);
\draw (v7)--(v3)--(v1)--(v0)--(v2)--(v6);
\draw (v3)--(v8);
\draw (v13)--(v8)--(v14);
\draw (v1)--(v4);
\draw (v9)--(v4)--(v10);
\draw (v15)--(v10)--(v16);
\draw (v2)--(v5);
\draw (v11)--(v5)--(v12);

\draw[<->] (4,-2)--(4.5,-2);
\draw (5,-2) node[right] {$78236154$};
\end{tikzpicture}
\caption{An increasing binary tree of size $8$ and its associated permutation.}
\label{fig:increasing-permutation}
\end{figure}
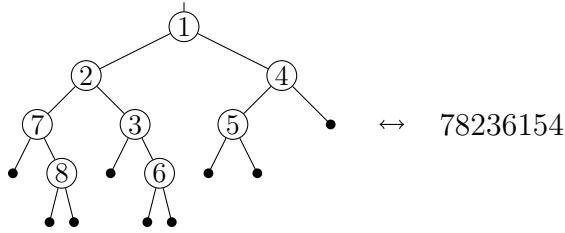

To describe the inverse bijection, we need to relax the conditions on the labels so that they can be any set of distinct positive integers. Then, for any permutation $w$ of this set, we can define an increasing binary tree recursively as follows. If $w$ is empty, the corresponding tree is the one with one leaf and no internal nodes; otherwise, factor it as $w=uiv$ where $i$ is the smallest letter of $w$, and construct a tree by placing $i$ at the root, and letting its left and right subtrees be the increasing binary trees corresponding to $u$ and $v$, respectively.

The above bijection implies that the exponential generating function for increasing binary trees with respect to size is simply 
$$F(z)=\frac{1}{1-z}.$$

\subsection{Enumeration with respect to leaf depths}\label{sec:increasing-leaves}

As in the unlabeled case from Section~\ref{sec:binary}, we index the $n+1$ leaves of an increasing binary tree of size $n$ from $0$ to $n$ from left to right, starting at $0$.
Let $I(x,y,z)$ be the multivariate exponential generating function where the coefficient of $x^ry^dz^n/n!$ is the number of increasing binary trees of size $n$ whose $r$th leaf has depth $d$.

\begin{theorem}\label{thm:I}
$$I(x,y,z)=\left((1-z)(1-xz)\right)^{-y}.$$
\end{theorem}

\begin{proof}
We think of $I(x,y,z)$ as the exponential generating function for increasing binary trees with a distinguished leaf, where $x$ marks the index of this leaf and $y$ marks its depth.  We claim that $I(x,y,z)$ satisfies the differential equation
\begin{equation}\label{eq:I} \frac{\partial}{\partial z}I(x,y,z)=y I(x,y,z)F(z)+xyF(xz)I(x,y,z)\end{equation}
with initial condition $I(x,y,0)=1$.

Indeed, taking the derivative has the effect of removing the root and decreasing the labels of its two subtrees by one.
If the distinguished leaf is in the left subtree, then this subtree contributes $yI(x,y,z)$, since the depth of this leaf is one more in the original tree, and the right subtree contributes $F(z)$.
If the distinguished leaf is in the right subtree, then the left subtree contributes $xI(xz)$, 
since it shifts the indexing of the leaves by the number of leaves in this subtree (which is one more than its size), and the right subtree contributes $yI(x,y,z)$ considering the increase in depth.

The tree consisting of one node has size $0$, and so it does not appear when taking the derivative, but it is counted in the initial condition $I(x,y,0)=1$.
Integrating~\eqref{eq:I} and using this initial condition, the expression follows.
\end{proof}

As an example, the coefficient of $z^3/3!$ in $I(x,y,z)$ equals
$$(2y+3y^2+y^3)+(3y^2+3y^3)\,x+(3y^2+3y^3)\,x^2+(2y+3y^2+y^3)\,x^3,$$
whose four coefficients, as a polynomial in $x$, describe the distribution of the depths of each of the four leaves in the trees in Figure~\ref{fig:n=4increasing}.

\begin{figure}[htb]
\centering
\begin{tikzpicture}[scale=.65] 
\node[draw,circle,inner sep=1pt,scale=.9] (v0) at (0,0) {1};
\node[draw,circle,inner sep=1pt,scale=.9] (v1) at (-1,-1) {2};
\node[draw,circle,inner sep=1pt,scale=.9] (v3) at (-1.7,-2) {3};
\draw (v0)--(0,.5);
\fill (1,-1) circle (.1) coordinate (v2) node[below,blue] {1};
\fill (-.3,-2) circle (.1) coordinate (v4) node[below,blue] {2};
\fill (-2.2,-3) circle (.1) coordinate (v5) node[below,red] {3};
\fill (-1.2,-3) circle (.1) coordinate (v6) node[below,blue] {3};
\draw (v2)--(v0)--(v1)--(v3)--(v5);
\draw (v3)--(v6);
\draw (v1)--(v4);

\begin{scope}[shift={(4,0)}]
\node[draw,circle,inner sep=1pt,scale=.9] (v0) at (0,0) {1};
\node[draw,circle,inner sep=1pt,scale=.9] (v1) at (-1,-1) {2};
\node[draw,circle,inner sep=1pt,scale=.9] (v4) at (-.3,-2) {3};
\draw (v0)--(0,.5);
\fill (1,-1) circle (.1) coordinate (v2) node[below,blue] {1};
\fill (-1.7,-2) circle (.1) coordinate (v3) node[below,red] {2};
\fill (-.8,-3) circle (.1) coordinate (v5) node[below,blue] {3};
\fill (.2,-3) circle (.1) coordinate (v6) node[below,blue] {3};
\draw (v2)--(v0)--(v1)--(v4)--(v5);
\draw (v4)--(v6);
\draw (v1)--(v3);
\end{scope}

\begin{scope}[shift={(8,0)}]
\node[draw,circle,inner sep=1pt,scale=.9] (v0) at (0,0) {1};
\node[draw,circle,inner sep=1pt,scale=.9] (v1) at (-1,-1) {2};
\node[draw,circle,inner sep=1pt,scale=.9] (v2) at (1,-1) {3};
\draw (v0)--(0,.5);
\fill (-1.7,-2) circle (.1) coordinate (v3) node[below,red] {2};
\fill (-.3,-2) circle (.1) coordinate (v4) node[below,blue] {2};
\fill (.3,-2) circle (.1) coordinate (v5) node[below,blue] {2};
\fill (1.7,-2) circle (.1) coordinate (v6) node[below,blue] {2};
\draw (v6)--(v2)--(v0)--(v1)--(v3);
\draw (v1)--(v4);
\draw (v2)--(v5);
\end{scope}

\begin{scope}[shift={(12,0)}]
\node[draw,circle,inner sep=1pt,scale=.9] (v0) at (0,0) {1};
\node[draw,circle,inner sep=1pt,scale=.9] (v1) at (-1,-1) {3};
\node[draw,circle,inner sep=1pt,scale=.9] (v2) at (1,-1) {2};
\draw (v0)--(0,.5);
\fill (-1.7,-2) circle (.1) coordinate (v3) node[below,red] {2};
\fill (-.3,-2) circle (.1) coordinate (v4) node[below,blue] {2};
\fill (.3,-2) circle (.1) coordinate (v5) node[below,blue] {2};
\fill (1.7,-2) circle (.1) coordinate (v6) node[below,blue] {2};
\draw (v6)--(v2)--(v0)--(v1)--(v3);
\draw (v1)--(v4);
\draw (v2)--(v5);
\end{scope}

\begin{scope}[shift={(16,0)}]
\node[draw,circle,inner sep=1pt,scale=.9] (v0) at (0,0) {1};
\node[draw,circle,inner sep=1pt,scale=.9] (v1) at (1,-1) {2};
\node[draw,circle,inner sep=1pt,scale=.9] (v4) at (.3,-2) {3};
\draw (v0)--(0,.5);
\fill (-1,-1) circle (.1) coordinate (v2) node[below,red] {1};
\fill (1.7,-2) circle (.1) coordinate (v3) node[below,blue] {2};
\fill (.8,-3) circle (.1) coordinate (v5) node[below,blue] {3};
\fill (-.2,-3) circle (.1) coordinate (v6) node[below,blue] {3};
\draw (v2)--(v0)--(v1)--(v4)--(v5);
\draw (v4)--(v6);
\draw (v1)--(v3);
\end{scope}

\begin{scope}[shift={(20,0)}]
\node[draw,circle,inner sep=1pt,scale=.9] (v0) at (0,0) {1};
\node[draw,circle,inner sep=1pt,scale=.9] (v1) at (1,-1) {2};
\node[draw,circle,inner sep=1pt,scale=.9] (v3) at (1.7,-2) {3};
\draw (v0)--(0,.5);
\fill (-1,-1) circle (.1) coordinate (v2) node[below,red] {1};
\fill (.3,-2) circle (.1) coordinate (v4) node[below,blue] {2};
\fill (2.2,-3) circle (.1) coordinate (v5) node[below,blue] {3};
\fill (1.2,-3) circle (.1) coordinate (v6) node[below,blue] {3};
\draw (v2)--(v0)--(v1)--(v3)--(v5);
\draw (v3)--(v6);
\draw (v1)--(v4);
\end{scope}
\end{tikzpicture}
\caption{The six increasing binary trees of size 3, with each leaf labeled by its depth. The depths of the $0$th leaf, in red, are encoded by the polynomial $2y+3y^2+y^3$.}
\label{fig:n=4increasing}
\end{figure}
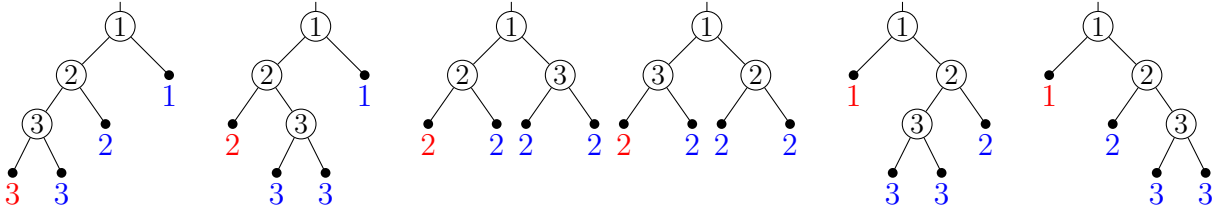

\begin{theorem}\label{thm:average_depths_increasing}
For $0\le r\le n$, the average depth of the $r$th leaf over all increasing binary trees of size $n$ equals
$$H_r+H_{n-r},$$
where $H_k=\sum_{i=1}^k \frac{1}{i}$ is the $k$th harmonic number.
\end{theorem}

\begin{proof}
Consider the partial derivative
\begin{equation} \label{eq:dIxz}
\dI(x,z)\coloneqq\left.\frac{\partial I(x,y,z)}{\partial y}\right|_{y=1}=\frac{1}{(1-z)(1-xz)}\ln\left(\frac{1}{(1-z)(1-xz)}\right).
\end{equation}
Since the number of increasing binary trees of size $n$ equals $n!$,
the coefficient of $x^rz^n$ in $\dI(x,z)$ is the average depth of the $r$th leaf over all increasing binary trees of size $n$. 
Extracting coefficients we get
$$[z^n]\dI(x,z)=
[z^n]\left(\sum_{k\ge0} z^k\right)\left(\sum_{j\ge0} x^jz^j\right)\left(\sum_{i\ge1} \frac{(1+x^i)z^i}{i}\right)
=\sum_{j\ge0}\sum_{i=1}^{n-j} \frac{x^j(1+x^i)}{i}.$$
Thus, for $0\le r\le n$, 
$$[x^rz^n]\dI(x,z)=H_r+H_{n-r}.\qedhere$$
\end{proof}

The average leaf depths computed using Theorem~\ref{thm:average_depths_increasing} for $n=20$ are shown in Figure~\ref{fig:average_depths_increasing}. 
An important difference with the unlabeled case studied in Section~\ref{sec:binary} is that, as $n\to\infty$, the average depth of the $r$th leaf in increasing binary trees of size $n$ is unbounded even for fixed $r$. For example, the leftmost leaf ($r=0$) has average depth $H_n$.

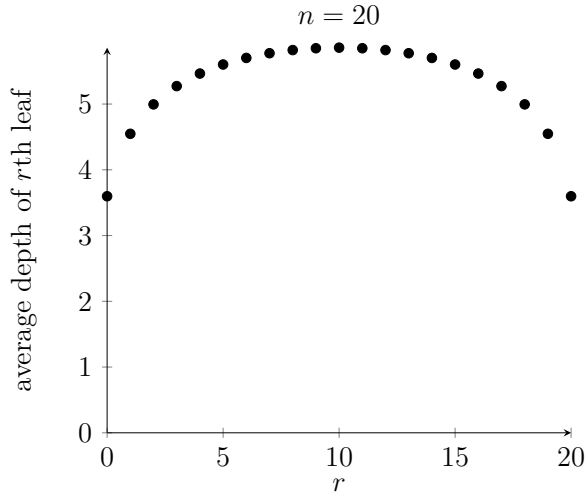
\begin{figure}[h]
\centering
\begin{tikzpicture}[scale=.9]
\begin{axis}[
	title={$n=20$},
    axis lines = left,
    xlabel = \(r\),
    ylabel = {average depth of $r$th leaf},
    ymin=0, xmax=20,
]
\addplot[
    mark=*,
    only marks,
    ]
    coordinates {
(0, 55835135/15519504) (1, 352893319/77597520) (2, 20400421/4084080) (3, 64604663/12252240) (4, 3938059/720720) (5, 2018579/360360) (6, 410923/72072) (7, 416071/72072) (8, 4034/693) (9, 16213/2772) (10, 7381/1260) (11, 16213/2772) (12, 4034/693) (13, 416071/72072) (14, 410923/72072) (15, 2018579/360360) (16, 3938059/720720) (17, 64604663/12252240) (18, 20400421/4084080) (19, 352893319/77597520) (20, 55835135/15519504)
    };
\end{axis}
\end{tikzpicture}
\caption{The average depth of each leaf in increasing binary trees of size $n=20$.}
\label{fig:average_depths_increasing}
\end{figure}

If $r=r(n)=\alpha n$ for some constant $0<\alpha<1$, the average depth of the $r$th leaf is asymptotically given by
$2\ln n+\ln\alpha+\ln(1-\alpha)$.
Dividing by the average depth of all the leaves in increasing binary trees of size $n$, which equals
$$\frac{[z^n]\dI(1,z)}{n+1}\sim 2\ln n,$$
it follows that, as $n\to\infty$, the normalized average depth of the $r$th leaf is the constant~$1$. 
In particular, the analogue of the plot on the right of Figure~\ref{fig:average_depths_large_r} in the case of increasing binary trees would be the function that is equal to $1$ for $0<\alpha<1$, and to $1/2$ for $\alpha\in\{0,1\}$.

\subsection{Enumeration with respect to internal node depths}\label{sec:increasing-nodes}

The above argument can be easily modified to keep track of the depths of the internal nodes instead of the leaves.
We index the $n$ internal nodes of an increasing binary tree of size $n$ following an inorder traversal, starting at $0$.
Recall that when reading the node labels in this order we obtain the permutation associated to the tree, as shown in Figure~\ref{fig:increasing-permutation}.
This allows us to think of the depths of the internal nodes as statistics on each of the entries of the permutation. 
Let $J(x,y,z)$ be the exponential generating function where 
the coefficient of $x^ry^dz^n/n!$ is the number of increasing binary trees of size $n$ whose $r$th internal node has depth $d$.

We remark that our indexing of the nodes following an inorder traversal is unrelated to their labels. The distribution of the depth of the node with a given label over increasing binary trees, which is a different problem, has been studied by Panholzer and Prodinger in~\cite{panholzer_level_2007}.

\begin{theorem}\label{thm:J}
$$J(x,y,z)=\left((1-z)(1-xz)\right)^{-y}\int_0^z \left((1-t)(1-xt)\right)^{y-1}\,dt.$$
\end{theorem}

\begin{proof}
Thinking of $J(x,y,z)$ as the exponential generating function for increasing binary trees with a distinguished internal node, where $x$ marks its index and $y$ marks its depth, we obtain the differential equation
\begin{equation}\label{eq:J} \frac{\partial}{\partial z}J(x,y,z)=F(xz)F(z)+y J(x,y,z)F(z)+xyF(xz)J(x,y,z)\end{equation}
with initial condition $J(x,y,0)=0$.
The argument is similar to the one for equation~\eqref{eq:I}, with the difference that the distinguished node can be the root, in which case we get a contribution $F(xz)$ from the left subtree and $F(z)$ from the right subtree.
The initial condition is now set to $0$ because the tree of size $0$ has no internal nodes.
The expression for $J(x,y,z)$ follows by integrating~\eqref{eq:J}.
\end{proof}

The analysis of the average depth of each internal node is very similar to one we did for leaves in the proof of Theorem~\ref{thm:average_depths_increasing}, so we leave out the details. We obtain the following average.

\begin{theorem}
For $0\le r\le n-1$, the average depth of the $r$th internal node over all increasing binary trees of size $n$ equals
$$[x^rz^n]\left.\frac{\partial J(x,y,z)}{\partial y}\right|_{y=1}=H_{r+1}+H_{n-r}-2.$$
\end{theorem}

\section{Concluding remarks}

We have explored a technique that allows one to keep track of infinitely many statistics using generating functions with finitely many variables, and we have applied it to the enumeration of various types of trees with respect to the depths of individual leaves or nodes, as well as Dyck paths with respect to some height statistics.
These are many other combinatorial objects where this method could be similarly applied, including Motzkin paths with respect to vertex or step heights, $k$-ary trees with respect to leaf depths, and plane increasing trees with respect to node or leaf depths, for example.

We note that when computing averages in most of our examples (see equations~\eqref{eq:coefdB}, \eqref{eq:coefdD}, \eqref{eq:coefdU}, \eqref{eq:coefdA} and \eqref{eq:coefdG}), we were able to extract the coefficients of the generating functions and then deduce their asymptotic behavior. It is likely that, more generally, tools from analytic combinatorics in several variables~\cite{pemantle_analytic_2013,melczer_algorithmic_2021} can be used to describe the asymptotic behavior of the coefficients of the resulting generating functions even when one does not have exact formulas for these coefficients.

\subsection*{Acknowledgments}
This work was partially supported by Simons Collaboration Grant \#929653. The author thanks Helmut Prodinger for providing useful references, Robert Noble for pointing out the connections of leaf depths to the Sackin index in mathematical biology, and two anonymous referees for their careful reading.

\bibliographystyle{plain}
\bibliography{leaf_depth}

\end{document}